\documentclass[11pt]{article}
\usepackage[margin=1.25in]{geometry}
\usepackage{amssymb,amsfonts,amsmath,amsthm,amscd,dsfont,mathrsfs,bbold,pifont}
\usepackage{blkarray}
\usepackage{graphicx,float,psfrag,epsfig,color}
\usepackage{comment}
\usepackage{subcaption}

\usepackage{algorithm}
\usepackage[noend]{algpseudocode}

\usepackage{pict2e}

\makeatletter
\def\BState{\State\hskip-\ALG@thistlm}
\makeatother

	\usepackage{microtype}
	\usepackage[pdftex,pagebackref=true,colorlinks]{hyperref}
	\hypersetup{linkcolor=[rgb]{.7,0,0}}
	\hypersetup{citecolor=[rgb]{0,.7,0}}
	\hypersetup{urlcolor=[rgb]{.7,0,.7}}

\newcommand\independent{\protect\mathpalette{\protect\independent}{\perp}} 
\def\independent#1#2{\mathrel{\rlap{$#1#2$}\mkern2mu{#1#2}}}

\newcommand{\gss}{\mathbb{G}_1}
\newcommand{\gs}{\mathbb{G}_2}
\newcommand{\bin}{\mathrm{Bin}}
\newcommand{\dd}{D_+}
\newcommand{\argmax}{\mathrm{argmax}}

\newcommand{\mR}{\mathbb{R}} 

\newcommand{\mZ}{\mathbb{Z}}

\newcommand{\pp}{\mathbb{P}}

\newcommand{\e}{\varepsilon}

\DeclareMathOperator{\diag}{diag}

\newcommand{\X}{\mathcal{X}}

\newtheorem{theorem}{Theorem}
\newtheorem{lemma}{Lemma}
\newtheorem{corollary}{Corollary}
\newtheorem{definition}{Definition}
\newtheorem{remark}{Remark}

\begin{document}

\title{Recovering communities in the general stochastic block model without knowing the parameters}
\author{Emmanuel Abbe and Colin Sandon}

\author{Emmanuel Abbe\thanks{Program in Applied and Computational Mathematics, and EE department, Princeton University, Princeton, USA, \texttt{eabbe@princeton.edu}. This research was partially supported by the 2014 Bell Labs Prize.} \and 
Colin Sandon\thanks{Department of Mathematics, Princeton University, USA,
\texttt{sandon@princeton.edu}.
}}

\maketitle

\begin{abstract}
Most recent developments on the stochastic block model (SBM) rely on the knowledge of the model parameters, or at least on the number of communities. This paper introduces efficient algorithms that do not require such knowledge and yet achieve the optimal information-theoretic tradeoffs identified in [AS15] for linear size communities. The results are three-fold: (i) in the constant degree regime, an algorithm is developed that requires only a lower-bound on the relative sizes of the communities and detects communities with an optimal accuracy scaling for large degrees; (ii) in the regime where degrees are scaled by $\omega(1)$ (diverging degrees), this is enhanced into a fully agnostic algorithm that only takes the graph in question and simultaneously learns the model parameters (including the number of communities) and detects communities with accuracy $1-o(1)$, with an overall quasi-linear complexity; (iii) in the logarithmic degree regime, an agnostic algorithm is developed that learns the parameters and achieves the optimal CH-limit for exact recovery, in quasi-linear time. These provide the first algorithms affording efficiency, universality and information-theoretic optimality for strong and weak consistency in the general SBM with linear size communities.
\end{abstract}

\thispagestyle{empty}
\newpage

\tableofcontents

\thispagestyle{empty}
\newpage

\pagenumbering{arabic}


\section{Introduction}
This paper studies the problem of recovering communities in the general stochastic block model with linear size communities, for constant and slowly diverging degree regimes. In contrast to \cite{colin1}, this paper does not require knowledge of the SBM parameters. In particular, the problem of learning the model parameters is solved when average degrees are diverging. We next provide some motivations on the problem and further background on the model. 

Detecting communities (or clusters) in graphs is a fundamental problem in networks, computer science and machine learning. This applies to a large variety of complex networks (e.g., social and biological networks) as well as to data sets engineered as networks via similarly graphs, where one often attempts to get a first impression on the data by trying to identify groups with similar behavior. In particular, finding communities allows one to find like-minded people in social networks \cite{newman-girvan,social1}, to improve recommendation systems \cite{amazon,xu-rec}, to segment or classify images \cite{image1,image2}, to detect protein complexes \cite{ppi2,marcotte}, to find genetically related sub-populations \cite{genetics,gene-survey}, or discover new tumor subclasses \cite{tumor}. 

While a large variety of community detection algorithms have been deployed in the past decades, the understanding of the fundamental limits of community detection has only appeared more recently, in particular for the SBM \cite{coja-sbm,decelle,massoulie-STOC,Mossel_SBM2,abh,mossel-consist,chen-xu,colin1}. The SBM is a canonical model for community detection \cite{holland,sbm1,sbm3,sbm4,bickel,newman2,bui,dyer,boppana,jerrum,condon,carson,snij,mcsherry,bickel,rohe,choi,sbm-algos}, where $n$ vertices are partitioned into $k$ communities of relative size $p_i$, $i \in [k]$, and pairs of nodes in communities $i$ and $j$ connect independently with probability $W_{i,j}$. 

Recently the SBM came back to the center of the attention at both the practical level, due to extensions allowing overlapping communities \cite{mixed-core} that have proved to fit well real data sets in massive networks \cite{prem}, and at the theoretical level due to new phase transition phenomena \cite{coja-sbm,decelle,massoulie-STOC,Mossel_SBM2,abh,mossel-consist}. The latter works focus exclusively on the SBM with two symmetric communities, i.e., each community is of the same size and the connectivity in each community is identical. Denoting by $p$ the intra- and $q$ the extra-cluster probabilities, most of the results are concerned with two figure of merits: {\bf (i) recovery} (also called exact recovery or strong consistency), which investigates the regimes of $p$ and $q$ for which there exists an algorithm that recovers with high probability the two communities completely \cite{bui,dyer,boppana,jerrum,condon,carson,snij,mcsherry,bickel,rohe,choi,sbm-algos,Vu-arxiv,chen-xu},  {\bf (ii) detection}, which investigates the regimes for which there exists an algorithm that recovers with high probability a positively correlated partition \cite{coja-sbm,decelle,Mossel_SBM1,massoulie-STOC,Mossel_SBM2}.

The sharp threshold for exact recovery was obtained in \cite{abh,mossel-consist}, showing\footnote{\cite{mossel-consist} generalizes this to $a,b=\Theta(1)$.} that for $p=a \log(n)/n$, $q= b\log(n)/n$, $a,b>0$, exact recovery is solvable if and only if $\sqrt{a}- \sqrt{b} \geq 2$, with efficient algorithms achieving the threshold provided in \cite{abh,mossel-consist}. In addition, \cite{abh} introduces an SDP proved to achieve the threshold in \cite{new-xu,afonso_single}, while \cite{prout} shows that a spectral algorithm also achieves the threshold. Prior to these, the sharp threshold for detection was obtained in \cite{massoulie-STOC,Mossel_SBM2}, showing that detection is solvable (and so efficiently) if and only if $(a-b)^2 > 2(a+b)$, when $p=a/n$, $q=b/n$, settling a conjecture made in \cite{decelle} and improving on \cite{coja-sbm}. 

Besides the detection and the recovery properties, one may ask about the partial recovery of the communities, studied in \cite{mossel2,sbm-groth,Vu-arxiv,new-vu,colin1}. Of particular interest to this paper is the case of {\bf almost exact recovery} (also called weak consistency), where only a vanishing fraction of the nodes is allowed to be misclassified. For two-symmetric communities, \cite{mossel-consist} shows that almost exact recovery is possible if and only if $n(p-q)^2/(p+q)$ diverges, generalized in \cite{colin1} for general SBMs.

In the next section, we discuss the results for the general SBM of interest in this paper and the problem of learning the model parameters. We conclude this section by providing motivations on the problem of achieving the threshold with an efficient and universal algorithm.

 Threshold phenomena have long been studied in fields such as information theory (e.g., Shannon's capacity) and constraint satisfaction problems (e.g., the SAT threshold). In particular, the quest of achieving the threshold has generated major algorithmic developments in these fields (e.g., LDPC codes, polar codes, survey propagation to name a few). Likewise, identifying thresholds in community detection models is key to benchmark and guide the development of clustering algorithms. Most reasonable algorithms may succeed in some regimes, while in others they may be doomed to fail due to computational barriers. However, it is particularly crucial to develop benchmarks that do not depend sensitively on the knowledge of the model parameters. A natural question is hence whether one can solve the various recovery problems in the SBM {\it without} having access to the parameters. This paper answers this question by the affirmative for the exact and almost exact recovery of the communities.

\subsection{Related results on the general SBM with known parameters}
Most of the previous works are concerned with the SBM having symmetric communities (mainly 2 or sometimes $k$), with the exception of \cite{Vu-arxiv} which provides some achievability results for the general SBM.\footnote{\cite{sbm-groth} also study variations of the $k$-symmetric model.} Recently, \cite{colin1} studied the fundamental limits for the general SBM, with results as follows (where SBM$(n,p,W)$ is the SBM with community prior $p$ and connectivity matrix $W$). \\

\noindent
{\bf I. Partial and almost exact recovery in the general SBM.} The first result of \cite{colin1} concerns the regime where the connectivity matrix scales as $Q/n$ for a positive symmetric matrix $Q$ (i.e., the node average degree is constant). The following notion of SNR is introduced\footnote{Note that this in a sense the ``worst-case'' notion of SNR, which ensures that all of the communities can be separated (when amplified); one could consider other ratios of the kind $|\lambda_{j}|^2/\lambda_{\mathrm{max}}$, for subsequent eigenvalues ($j=2,3,\dots$), if interested in separating only subset of the communities.} 
\begin{align}
\mathrm{SNR}=|\lambda_{\mathrm{min}}|^2/\lambda_{\mathrm{max}}
\end{align}
where $\lambda_{\mathrm{min}}$ and $\lambda_{\mathrm{max}}$ are respectively the smallest\footnote{The smallest eigenvalue of $\diag(p)Q$ is the one with least magnitude.} and largest eigenvalue of $\diag(p)Q$.  

The algorithm {\tt Sphere-comparison} is proposed that solves partial recovery with exponential accuracy and quasi-linear complexity when the SNR diverges, solving in particular almost exact recovery. 
\begin{theorem}\cite{colin1}
Given any $k\in \mathbb{Z}$, $p\in (0,1)^k$ with $|p|=1$, and symmetric matrix $Q$ with no two rows equal, let $\lambda$ be the largest eigenvalue of $PQ$, and $\lambda'$ be the eigenvalue of $PQ$ with the smallest nonzero magnitude.
If $\rho:=\frac{|\lambda'|^2}{\lambda}>4$, $\lambda^7<(\lambda')^8,$ and $4\lambda^3<(\lambda')^4$, then for some $\e=\e(\lambda,\lambda')$ and $C=C(p,Q)>0$, {\tt Sphere-comparison} (see Section \ref{pt1}) detects with high probability communities in graphs drawn from SBM$(n,p,Q/n)$ with accuracy $1- 4ke^{-\frac{C \rho}{16k}}/(1-exp(-\frac{C\rho}{16k}\left(\frac{(\lambda')^4}{\lambda^3}-1\right)))$,
provided that the above is larger than $1-\frac{\min_i p_i}{2\ln(4k)}$, and runs in $O(n^{1+\epsilon})$ time. Moreover, $\e$ can be made arbitrarily small with $8\ln (\lambda\sqrt{2}/|\lambda'|)/\ln(\lambda)$, and $C(p,\alpha Q)$ is independent of $\alpha$.
\end{theorem}
Note that for $k$ symmetric clusters, SNR reduces to $\frac{(a-b)^2}{k(a+(k-1)b)}$, which is the quantity of interest for detection \cite{decelle,Mossel_SBM1}. Moreover, the SNR must diverge to ensure almost exact recovery in the symmetric case \cite{colin1}. The following is an important consequence of the previous theorem, as it shows that {\tt Sphere-comparison} achieves almost exact recovery when the entries of $Q$ are scaled. 
\begin{corollary}\cite{colin1}
For any $k\in \mathbb{Z}$, $p\in (0,1)^k$ with $|p|=1$, and symmetric  matrix $Q$ with no two rows equal, there exists $\epsilon(\delta)=O(1/\ln(\delta))$ such that for all sufficiently large $\delta$ there exists an algorithm ({\tt Sphere-comparison}) that detects communities in graphs drawn from SBM$(n,p,\delta Q)$ with accuracy $1-e^{-\Omega(\delta)}$ and complexity $O_n(n^{1+\epsilon(\delta)})$.
\end{corollary}

\noindent
{\bf II. Exact recovery in the general SBM.} The second result in \cite{colin1} is for the regime where the connectivity matrix scales as $\log(n)Q/n$, $Q$ fixed, where it is shown that exact recovery has a sharp threshold characterized by the divergence function $$D_+(f,g)=\max_{t \in [0,1]} \sum_{x \in [k]} \left( tf(x) + (1-t)g(x)- f(x)^t g(x)^{1-t} \right),$$ named the CH-divergence in \cite{colin1}. Specifically, if all pairs of columns in $\diag(p)Q$ are at $D_+$-distance at least 1 from each other, then exact recovery is solvable in the general SBM. This provides in particular an operational meaning to a new divergence function analog to the KL-divergence in the channel coding theorem (see Section 2.3 in \cite{colin1}).  
Moreover, an algorithm ({\tt Degree-profiling}) is developed that solves exact recovery down to the $D_+$ limit in quasi-linear time, showing that exact recovery has no informational to computational gap (as opposed to the conjectures made for detection with more than 4 communities \cite{decelle}). The following gives a more general statement characterizing which subset of communities can be extracted --- see Definition \ref{def-exact} for formal definitions.
\begin{theorem}\cite{colin1}
{\bf(i)} Exact recovery is solvable in the stochastic block model $\gs(n,p,Q)$ for a partition $[k] = \sqcup_{s=1}^t A_s$ if and only if for all $i$ and $j$ in different subsets of the partition,\footnote{The entries of $Q$ are assumed to be non-zero.}
\begin{align}
\dd ((PQ)_i , (PQ)_j) \geq 1, \label{d1}
\end{align}
In particular, exact recovery is information-theoretically solvable in SBM$(n,p,Q\log(n)/n)$ if and only if $\min_{i,j \in [k], i \neq j} \dd ((PQ)_i || (PQ)_j) \geq 1$.\\
{\bf(ii)} The {\tt Degree-profiling} algorithm (see \cite{colin1}) recovers the finest partition that can be recovered with probability $1-o_n(1)$ and runs in $o(n^{1+\epsilon})$ time for all $\epsilon>0$. In particular, exact recovery is efficiently solvable whenever it is information-theoretically solvable. 
\end{theorem}

In summary, exact or almost exact recovery is closed for the general SBM (and detection is closed for 2 symmetric communities). However this is for the case where the parameters of the SBM are assumed to be known, and with linear-size communities. 

\subsection{Estimating the parameters}
For the estimation of the parameter, some results are known for two-symmetric communities. In the logarithmic degree regime, since the SDP is agnostic to the parameters (it is a relaxation of the min-bisection), and the parameters can be estimated by recovering the communities \cite{abh,new-xu,afonso_single}. For the constant-degree regime, \cite{Mossel_SBM1} shows that the parameters can be estimated above the threshold by counting cycles (which is efficiently approximated by counting non-backtracking walks). These are however for a fixed number of communities, namely 2. We also became recently aware of a parallel work \cite{borgs_nips}, which considers private graphon estimation (including SBMs). In particular, for the logarithmic degree regime, \cite{borgs_nips} obtains a procedure to estimate parameters of graphons in an appropriate version of the $L_2$ norm. This procedure is however not efficient. 

For the general SBM, the results of \cite{colin1} allow to find communities efficiently, however these rely on the knowledge of the parameters. Hence, a major open problem is to understand if these results can be extended without such a knowledge.

\section{Results}\label{res-sec}
Agnostic algorithms are developed for the constant and diverging node degrees. These afford optimal accuracy scaling for large node degrees and achieve the CH-divergence limit for logarithmic node degrees in quasi-linear time. In particular, these solve the parameter estimation problems for SBM$(n,p,\omega(1)Q)$ without knowing the number of communities. An example with real data is provided in Section \ref{data-sec}.

\subsection{Definitions and terminologies}\label{def-term}
The general stochastic block model $\text{SBM}(n,p,W)$ is a random graph ensemble defined on the vertex-set $V=[n]$, 
where each vertex $v \in V$ is assigned independently a hidden (or planted) label $\sigma_v$ in $[k]$ under a probability distribution $p=(p_1,\dots,p_k)$ on $[k]$, and each (unordered) pair of nodes $(u,v) \in V\times V$ is connected independently with probability $W_{\sigma_u,\sigma_v}$, where $W_{\sigma_u,\sigma_v}$ is specified by a symmetric $k \times k$ matrix $W$ with entries in $[0,1]$.
Note that $G\sim \text{SBM}(n,p,W)$ denotes a random graph drawn under this model, without the hidden (or planted) clusters (i.e., the labels $\sigma_v$ ) revealed. The goal is to recover these labels by observing only the graph. 

This paper focuses on $p$ independent of $n$ (the communities have linear size), $W$ dependent on $n$ such that the average node degrees are either constant or logarithmically growing, and $k$ fixed. These assumptions on $p$ and $k$ could be relaxed, for example to slowly growing $k$, but we leave this for future work. As discussed in the introduction, the above regimes for $W$ are both motivated by applications, as networks are typically sparse \cite{sparse-network1,sparse-network2} though the average degrees may not be small, and by the fact that interesting mathematical phenomena take place in these regimes. For convenience, we attribute specific notations for the model in these regimes:
\begin{definition}
For a symmetric matrix $Q \in \mR_+^{k \times k}$, $\gss(n,p,Q)$ denotes $\text{SBM}(n,p,Q/n)$,
and $\gs(n,p,Q)$ denotes $\text{SBM}(n,p,\ln(n)Q/n)$.
\end{definition}
\begin{definition} (Partial recovery.) 
An algorithm recovers or detects communities in $\text{SBM}(n,p,W)$ with an accuracy of $\alpha \in [0,1]$, if it outputs a labelling of the nodes $\{\sigma'(v), v \in V\}$, which agrees with the true labelling $\sigma$ on a fraction $\alpha$ of the nodes with probability $1-o_n(1)$. The agreement is maximized over relabellings of the communities.
\end{definition}

\begin{definition}\label{def-exact} (Exact recovery.) 
Exact recovery is solvable in $\text{SBM}(n,p,W)$ for a community partition $[k] = \sqcup_{s=1}^t A_s$, where $A_s$ is a subset of $[k]$, if there exists an algorithm that takes $G \sim SBM(n,p,W)$ and assigns to each node in $G$ an element of $\{A_1,\dots,A_t\}$ that contains its true community\footnote{This is again up to relabellings of the communities.} with probability $1-o_n(1)$.  Exact recovery is solvable in $\text{SBM}(n,p,W)$ if it is solvable for the partition of $[k]$ into $k$ singletons, i.e., all communities can be recovered.  
\end{definition}
The problem is solvable information-theoretically if there exists an algorithm that solves it, and efficiently if the algorithm runs in polynomial-time in $n$. Note that exact recovery for the partition $[k]=\{i\} \sqcup ([k] \setminus \{i\})$ is equivalent to extracting community $i$. In general, recovering a partition $[k] = \sqcup_{s=1}^t A_s$ is equivalent to merging the communities that are in a common subset $A_s$ and recovering the merged communities. 
Note also that exact recovery in $\text{SBM}(n,p,W)$ requires the graph not to have vertices of degree $0$ in multiple communities with high probability (i.e., connectivity in the symmetric case). 
Therefore, for exact recovery, we focus below on $W=\frac{\ln(n)}{n}Q$ where $Q$ is fixed.

\subsection{Partial recovery} 
Our main result in the Appendix (Theorme \ref{thm1}) applies to SBM$(n,p,Q/n)$ with arbitrary $Q$. 
We provided here a specific instance which is easier to parse. 


\begin{theorem}[See Theorem \ref{thm1}]
Given $\delta>0$ and for any $k\in \mathbb{Z}$, $p\in (0,1)^k$ with $\sum p_i=1$ and $0<\delta\le \min p_i$, and any symmetric matrix $Q$ with no two rows equal such that every entry in $Q^k$ is positive (in other words, $Q$ such that there is a nonzero probability of a path between vertices in any two communities in a graph drawn from $\gss(n,p,c Q)$), there exists $\epsilon(c)=O(1/\ln(c))$ such that for all sufficiently large $c$, {\tt Agnostic-sphere-comparison}$(G,\delta)$ detects communities in graphs drawn from $\gss(n,p,c Q)$ with accuracy at least $1-e^{-\Omega(c)}$ in $O_n(n^{1+\epsilon(c)})$ time.
\end{theorem}

Note that a vertex in community $i$ has degree $0$ with probability exponential in $c$, and there is no way to differentiate between vertices of degree $0$ from different communities. So, an error rate that decreases exponentially with $c$ is optimal. The above gives in particular the parameter estimation in the case $c=\omega(1)$ (see also Lemma \ref{estim-lemma} in the Appendix).

The general result in the Appendix yields the following refined results in the $k$-block symmetric case.

\begin{theorem}
Consider the $k$-block symmetric case. In other words, $p_i=\frac{1}{k}$ for all $i$, and $Q_{i,j}$ is $\alpha$ if $i=j$ and $\beta$ otherwise. The vector whose entries are all $1$s is an eigenvector of $PQ$ with eigenvalue $\frac{\alpha+(k-1)\beta}{k}$, and every vector whose entries add up to $0$ is an eigenvector of $PQ$ with eigenvalue $\frac{\alpha-\beta}{k}$. So, $\lambda=\frac{\alpha+(k-1)\beta}{k}$ and $\lambda'=\frac{\alpha-\beta}{k}$
and 
$\frac{(\lambda')^2}{\lambda}= \frac{(a-b)^2}{k(a+(k-1)\beta)}.$
Then, as long as 
$\frac{160}{9}k(\alpha+(k-1)\beta)^7<(\alpha-\beta)^8$ and $4k(\alpha+(k-1)\beta)^3<(\alpha-\beta)^4$, 
there exist a constant $c>0$ such that {\tt Agnostic-sphere-comparison($G,1/k$)} detects communities with an accuracy of $1-O(e^{-c(\alpha-\beta)^2/(\alpha+(k-1)\beta)})$ for sufficiently large $(\alpha-\beta)^2/(\alpha+(k-1)\beta)$. 
\end{theorem}


We refer to Section \ref{data-sec} for an example of implementation with real data.
\subsection{Exact recovery} 
Recall that from \cite{colin1}, exact recovery is information-theoretically solvable in the stochastic block model $\gs(n,p,Q)$ for a partition $[k] = \sqcup_{s=1}^t A_s$ if and only if for all $i$ and $j$ in different subsets of the partition,
\begin{align}
\dd ((PQ)_i , (PQ)_j) \geq 1. \label{d1}
\end{align}
We next show that this can be achieved without knowing the parameters. Recall that the finest partition is the largest partition of $[k]$ that ensure \eqref{d1}.
\begin{theorem}(See Theorem \ref{thm2})
The {\tt Agnostic-degree-profiling} algorithm (see Section \ref{pt2}) recovers the finest partition in any $\gs(n,p,Q)$, it uses no input except the graph in question, and runs in $o(n^{1+\epsilon})$ time for all $\epsilon>0$. In particular, exact recovery is efficiently and universally solvable whenever it is information-theoretically solvable. 
\end{theorem}
The proof assumes that the entries of $Q$ are non-zero, see Remark \ref{qzero} for zero entries. 
To achieve this result we rely on a two step procedure. First an algorithm is developed to recover all but a vanishing fraction of nodes --- this is the main focus of our partial recovery result --- and then a procedure is used to ``clean up'' the leftover graphs using the node degrees of the preliminary classification. This turns out to be much more efficient than aiming for an algorithm that directly achieves exact recovery. We already used this technique in \cite{colin1}, but here we also deal with the difficulties resulting from not knowing the SBM's parameters.

\section{Proof Techniques and Algorithms}
\subsection{Partial recovery and the {\tt Agnostic-sphere-comparison} algorithm}\label{pt1}
The first key observation used to classify graphs' vertices is that if $v$ is a vertex in a graph drawn from $\gss(n,p,Q)$ then for all small $r$ the expected number of vertices in community $i$ that are $r$ edges away from $v$ is approximately $e_i\cdot(PQ)^re_{\sigma_v}$. So, we define: 

\begin{definition}\label{def-n1}
For any vertex $v$, let $N_r(v)$ be the set of all vertices with shortest path to $v$ of length $r$. We also refer to the vector with $i$-th entry equal to the number of vertices in $N_r(v)$ that are in community $i$ as $N_r(v)$. If there are multiple graphs that $v$ could be considered a vertex in, let $N_{r[G]}(v)$ be the set of all vertices with shortest paths in $G$ to $v$ of length $r$.
\end{definition}

One could probably determine $PQ$ and $e_{\sigma}$ given the values of $(PQ)^re_{\sigma_v}$ for a few different $r$, but using $N_r(v)$ to approximate that would require knowing how many of the vertices in $N_r(v)$ are in each community. So, we attempt to get information relating to how many vertices in $N_r(v)$ are in each community by checking how it connects to $N_{r'}(v')$ for some vertex $v'$ and integer $r'$. The obvious way to do this would be to compute the cardinality of their intersection. Unfortunately, whether a given vertex in community $i$ is in $N_r(v)$ is not independent of whether it is in $N_{r'}(v')$, which causes the cardinality of $|N_r(v)\cap N_{r'}(v')|$ to differ from what one would expect badly enough to disrupt plans to use it for approximations.

In order to get around this, we randomly assign every edge in $G$ to a set $E$ with probability $c$. We hence define the following. 

\begin{definition}
For any vertices $v, v'\in G$, $r,r'\in \mathbb{Z}$, and subset of $G$'s edges $E$, let $N_{r,r'[E]}(v\cdot v')$ be the number of pairs of vertices $(v_1,v_2)$ such that $v_1\in N_{r[G\backslash E]}(v)$, $v_2\in N_{r'[G\backslash E]}(v')$, and $(v_1,v_2)\in E$.
\end{definition}

Note that $E$ and $G\backslash E$ are disjoint; however, $G$ is sparse enough that even if they were generated independently a given pair of vertices would have an edge between them in both with probability $O(\frac{1}{n^2})$. So, $E$ is approximately independent of $G\backslash E$. Thus, for any $v_1\in N_{r[G/E]}(v)$ and $v_2\in N_{r'[G/E]}(v')$, $(v_1,v_2)\in E$ with a probability of approximately $cQ_{\sigma_{v_1},\sigma_{v_2}}/n$. As a result,

\begin{align*}
N_{r,r'}[E](v\cdot v')&\approx N_{r[G\backslash E]}(v)\cdot \frac{cQ}{n} N_{r'[G\backslash E]}(v')\\
&\approx ((1-c)PQ)^re_{\sigma_v}\cdot \frac{cQ}{n} ((1-c)PQ)^{r'}e_{\sigma_{v'}}\\
&=c(1-c)^{r+r'}e_{\sigma_v}\cdot Q(PQ)^{r+r'}e_{\sigma_{v'}}/n
\end{align*}

 Let $\lambda_1,...,\lambda_h$ be the distinct eigenvalues of $PQ$, ordered so that $|\lambda_1|\ge|\lambda_2|\ge...\ge|\lambda_h|\ge 0$. Also define $h'$ so that $h'=h$ if $\lambda_h\ne 0$ and $h'=h-1$ if $\lambda_h=0$. If $W_i$ is the eigenspace of $PQ$ corresponding to the eigenvalue $\lambda_i$, and $P_{W_i}$ is the projection operator on to $W_i$, then 
\begin{align}
N_{r,r'[E]}(v\cdot v')&\approx c(1-c)^{r+r'}e_{\sigma_v}\cdot Q(PQ)^{r+r'}e_{\sigma_{v'}}/n\\
&=\frac{c(1-c)^{r+r'}}{n} \left(\sum_i P_{W_i}(e_{\sigma_v})\right)\cdot Q(PQ)^{r+r'}\left(\sum_j P_{W_j}(e_{\sigma_{v'}})\right)\\
&=\frac{c(1-c)^{r+r'}}{n} \sum_{i,j}P_{W_i}(e_{\sigma_v})\cdot Q(PQ)^{r+r'}P_{W_j}(e_{\sigma_{v'}})\\
&=\frac{c(1-c)^{r+r'}}{n} \sum_{i,j}P_{W_i}(e_{\sigma_v})\cdot P^{-1}(\lambda_j)^{r+r'+1}P_{W_j}(e_{\sigma_{v'}})\\
&=\frac{c(1-c)^{r+r'}}{n} \sum_{i}\lambda_i^{r+r'+1}P_{W_i}(e_{\sigma_v})\cdot P^{-1}P_{W_i}(e_{\sigma_{v'}})\label{NformB}
\end{align}
where the final equality holds because for all $i\ne j$, 
\begin{align*}
\lambda_i  P_{W_i}(e_{\sigma_v})\cdot P^{-1} P_{W_j}(e_{\sigma_{v'}})&=(PQ  P_{W_i}(e_{\sigma_v}))\cdot P^{-1} P_{W_j}(e_{\sigma_{v'}})\\
&= P_{W_i}(e_{\sigma_v})\cdot Q P_{W_j}(e_{\sigma_{v'}})\\
&= P_{W_i}(e_{\sigma_v})\cdot P^{-1} \lambda_j P_{W_j}(e_{\sigma_{v'}}),
\end{align*}
and since $\lambda_i\ne \lambda_j$, this implies that $P_{W_i}(e_{\sigma_v})\cdot P^{-1} P_{W_j}(e_{\sigma_{v'}})=0$. In order to simplify the terminology,
\begin{definition}
Let $\zeta_i(v\cdot v')=P_{W_i}(e_{\sigma_v})\cdot P^{-1}P_{W_i}(e_{\sigma_{v'}})$ for all $i$, $v$, and $v'$.
\end{definition}

Equation \eqref{Nform} is dominated by the $\lambda_1^{r+r'+1}$ term, so getting good estimate of the $\lambda_2^{r+r'+1}$ through $\lambda_{h'}^{r+r'+1}$ terms requires cancelling it out somehow. As a start, if $\lambda_1>\lambda_2>\lambda_3$ then 
\begin{align*}
&N_{r+2,r'[E]}(v\cdot v')\cdot N_{r,r'[E]}(v\cdot v')-N^2_{r+1,r'[E]}(v\cdot v')\\
&\approx \frac{c^2(1-c)^{2r+2r'+2}}{n^2}(\lambda_1^2+\lambda_2^2-2\lambda_1\lambda_2)\lambda_1^{r+r'+1}\lambda_2^{r+r'+1} \zeta_1(v\cdot v')\zeta_2(v\cdot v')
\end{align*}

Note that the left hand side of this expression is equal to 
$\det{}\begin{vmatrix}N_{r,r'[E]}(v\cdot v')&N_{r+1,r'[E]}(v\cdot v')\\
N_{r+1,r'[E]}(v\cdot v')&N_{r+2,r'[E]}(v\cdot v')\end{vmatrix}$. More generally, in order to get an expression that can be used to estimate the $\lambda_i$ and $\zeta_i(v\cdot v')$, we consider the determinant of the following.
\begin{definition}
Let $M_{m,r,r'[E]}(v\cdot v')$ be the $m\times m$ matrix such that $M_{m,r,r'[E]}(v\cdot v')_{i,j}=N_{r+i+j,r'[E]}(v\cdot v')$ for each $i$ and $j$.
\end{definition}

To the degree that approximation $\ref{NformB}$ holds and $c$ is small, each column of $M_{m,r,r'[E]}(v\cdot v')$ is a linear combination of the vectors 
\[\frac{c(1-c)^{r+r'}}{n}\zeta_i(v\cdot v')\lambda_i^{r+r'}[1,\lambda_i,\lambda_i^2,...,\lambda_i^{m-1}]^{t}\]
 with coefficients that depend only on $\{\lambda_1,...,\lambda_{h}\}$. So, by linearity of the determinant in one column, $\det(M_{m,r,r'[E]}(v\cdot v'))$ is a linear combination of these vector's wedge products with coefficients that are independent of $r$ and $r'$. By antisymmetry of wedge products, only the products that use $m$ different such vectors contribute to the determinant, and the products involving the eigenvalues of highest magnitude will dominate. As a result, there exist constants $\gamma(\lambda_1,...,\lambda_m)$ and $\gamma'(\lambda_1,...,\lambda_m)$ such that 
\[\det(M_{m,r,r'[E]}(v\cdot v'))\approx \frac{c^m(1-c)^{m(r+r')}}{n^m}\gamma(\lambda_1,...,\lambda_m)\prod_{i=1}^m \lambda_i^{r+r'+1}\zeta_i(v\cdot v')\]
 if $|\lambda_m|>|\lambda_{m+1}|$, and 
\begin{align*}
&\det(M_{m,r,r'[E]}(v\cdot v'))\\
&\approx \frac{c^m(1-c)^{m(r+r')}}{n^m}\cdot\prod_{i=1}^{m-1} \lambda_i^{r+r'+1}\zeta_i(v\cdot v')\\
&\indent \cdot\left(\gamma(\lambda_1,...,\lambda_m) \lambda_m^{r+r'+1}\zeta_m(v\cdot v')+\gamma'(\lambda_1,...,\lambda_m) \lambda_{m+1}^{r+r'+1} \zeta_{m+1}(v\cdot v')\right)
\end{align*}
if $|\lambda_m|=|\lambda_{m+1}|$. These facts suggest the following plan for estimating the eigenvalues corresponding to a graph. First, pick several vertices at random. Then, use the fact that $|N_{r[G\backslash E]}(v)|\approx ((1-c)\lambda_1)^r$ for any good vertex $v$ to estimate $\lambda_1$. Next, use the formulas above about $\det(M_{m,r,r'[E]}(v\cdot v))$ to get an approximation of $h'$ and all of $PQ$'s eigenvalues for each selected vertex. Finally, take the median of these estimates.

Now, note that whether or not $|\lambda_m|=|\lambda_{m+1}|$, we have 
\begin{align*}
&\det(M_{m,r+1,r'[E]}(v\cdot v'))-(1-c)^m\lambda_{m+1}\prod_{i=1}^{m-1} \lambda_i \det(M_{m,r,r'[E]}(v\cdot v'))\\
&\approx \frac{c^m}{n^m} \gamma(\lambda_1,...,\lambda_m)\frac{\lambda_m-\lambda_{m+1}}{(1-c)^m\lambda_m}\prod_{i=1}^{m} ((1-c)\lambda_i)^{r+r'+2}\zeta_i(v\cdot v')
\end{align*}

That means that 
\begin{align*}
&\frac{\det(M_{m,r+1,r'[E]}(v\cdot v'))-(1-c)^m\lambda_{m+1}\prod_{i=1}^{m-1} \lambda_i \det(M_{m,r,r'[E]}(v\cdot v'))}{\det(M_{m-1,r+1,r'[E]}(v\cdot v'))-(1-c)^{m-1}\lambda_{m}\prod_{i=1}^{m-2} \lambda_i \det(M_{m-1,r,r'[E]}(v\cdot v'))}\\
&\approx\frac{c}{(1-c)n}\frac{\gamma(\lambda_1,...,\lambda_m)}{\gamma(\lambda_1,...,\lambda_{m-1})}\frac{\lambda_{m-1}(\lambda_m-\lambda_{m+1})}{\lambda_m(\lambda_{m-1}-\lambda_m)}((1-c)\lambda_m)^{r+r'+2} \zeta_m(v\cdot v')
\end{align*}
This fact can be used in combination with estimates of $PQ$'s eigenvalues to approximate $\zeta_i(v\cdot v')$ for arbitrary $v$, $v'$, and $i$.

Of course, this requires $r$ and $r'$ to be large enough that \[\frac{c(1-c)^{r+r'}}{n} \lambda_i^{r+r'+1}\zeta_i(v\cdot v')\] is large relative to the error terms for all $i\le h'$. At a minimum, that requires that $|(1-c)\lambda_i|^{r+r'+1}=\omega(n)$.

On a different note, for any $v$ and $v'$, 
\[0 \le P_{W_i}(e_{\sigma_v}-e_{\sigma_{v'}})\cdot P^{-1}P_{W_i}(e_{\sigma_v}-e_{\sigma_{v'}})=\zeta_i(v\cdot v)-2\zeta_i(v\cdot v')+\zeta_i(v'\cdot v')\]
with equality for all $i$ if and only if $\sigma_v=\sigma_{v'}$, so sufficiently good approximations of $\zeta_i(v\cdot v),\zeta_i(v\cdot v')$ and $\zeta_i(v'\cdot v')$ can be used to determine which pairs of vertices are in the same community.

One could generate a reasonable classification based solely on this method of comparing vertices. However, that would require computing $N_{r,r'[E]}(v\cdot v)$ for every vertex in the graph with fairly large $r+r'$, which would be slow. Instead, we use the fact that for any vertices $v$, $v'$, and $v''$ with $\sigma_v=\sigma_{v'}\ne\sigma_{v''}$, 
\[\zeta_i(v'\cdot v')-2\zeta_i(v\cdot v')+\zeta_i(v\cdot v)=0\le \zeta_i(v''\cdot v'')-2\zeta_i(v\cdot v'')+\zeta_i(v\cdot v)\]
for all $i$, and the inequality is strict for at least one $i$. So, subtracting $\zeta_i(v\cdot v)$ from both sides gives us that
\[\zeta_i(v'\cdot v')-2\zeta_i(v\cdot v')\le \zeta_i(v''\cdot v'')-2\zeta_i(v\cdot v'')\]
for all $i$, and the inequality is still strict for at least one $i$.

So, given a representative vertex in each community, we can determine which of them a given vertex, $v$, is in the same community as without needing to know the value of $\zeta_i(v\cdot v)$. 

This runs fairly quickly if $\zeta_i(v\cdot v')$ is approximated using $N_{r,r'[E]}(v'\cdot v)$ such that $r$ is large and $r'$ is small because the algorithm only requires focusing on $|N_{r'}(v)|$ vertices. This leads to the following plan for partial recovery. First, randomly select a set of vertices that is large enough to contain at least one vertex from each community with high probability. Next, compare all of the selected vertices in an attempt to determine which of them are in the same communities. Then, pick one in each community. After that, use the algorithm referred to above to attempt to determine which community each of the remaining vertices is in. As long as there actually was at least one vertex from each community in the initial set and none of the approximations were particularly bad, this should give a reasonably accurate classification.

The risk that this randomly gives a bad classification due to a bad set of initial vertices can be mitigated as follows. First, repeat the previous classification procedure several times. Assuming that the procedure gives a good classification more often than not, the good classifications should comprise a set that contains more than half the classifications and which has fairly little difference between any two elements of the set. Furthermore, any such set would have to contain at least one good classification, so none of its elements could be too bad. So, find such a set and average its classifications together. This completes the {\tt Agnostic-Sphere-comparison-algorithm}. We refer to Section \ref{details} for a detailed version.

\subsection{Exact recovery and the {\tt Agnostic-degree-profiling} algorithm}\label{pt2}
The exact recovery part is similar to \cite{colin1} and uses the fact that once a good enough clustering has been obtained from {\tt Agnostic-sphere-comparison}, the classification can be finished by making local improvements based on the nodes' neighborhoods. 
The key result here is that, when testing between two multivariate Poisson distributions of means $\log(n) \lambda_1$ and $\log(n) \lambda_2$ respectively, where $\lambda_1,\lambda_2 \in \mZ_+^k$, the probability of error (of say maximum a posteriori decoding) is
\begin{align}
\Theta\left(n^{- \dd(\lambda_1,\lambda_2) - o(1)} \right).
\end{align}
This is proved in \cite{colin1}. In the case of unknown parameters, the algorithmic approach is largely unchanged, adding a step where the best known classification is used to estimate $P$ and $Q$ prior to any step in which vertices are classified based on their neighbors.  The analysis of the algorithm requires however some careful handling. 

 First, it is necessary to prove that given a labelling of the graph's vertices with an error rate of $x$, one can compute approximations of $P$ and $Q$ that are within $O(x+\log(n)/\sqrt{n})$ of their true values with probability $1-o(1)$. Secondly, one needs to modify the robust degree profiling lemma to show that attempting to determine vertices' communities based on estimates of $p$ and $Q$ that are off by at most $\delta$, $p'$ and $Q'$,  and a classification of its neighbors that has an error rate of $\delta$ classifies the vertices with an error rate only $e^{O(\delta\log n )}$ times higher than it would be given accurate values of $p$ and $Q$ and accurate classifications of the vertices' neighbors. Combining these yields the conclusion that any errors in the estimates of the SBM's parameters do not disrupt vertex classification any worse than the errors in the preliminary classifications already were.

{\bf The {\tt Agnostic-degree-profiling} algorithm.} The inputs are $(G, \gamma)$, where $G$ is a graph, and $\gamma\in [0,1]$ (see Theorem \ref{thm2} for how to set $\gamma$).

The algorithm outputs an assignment of each vertex to one of the groups of communities $\{A_1,\dots,A_t\}$, where $A_1,\dots,A_t$ is the partition of $[k]$ in to the largest number of subsets such that $\dd((pQ)_i,(pQ)_j) \geq 1$ for all $i,j$ in $[k]$ that are in different subsets (this is called the ``finest partition''). It does the following:

(1) Define the graph $g'$ on the vertex set $[n]$ by selecting each edge in $g$ independently with probability $\gamma$, and define the graph $g''$ that contains the edges in $g$ that are not in $g'$. 

(2) Run {\tt Agnostic-sphere-comparison} on $g'$ to obtain the preliminary classification $\sigma' \in [k]^n$ (see Section \ref{partial-sec}.) 

(3) Determine the size of each alleged community, and the edge density between each pair of alleged communities. 

(4) For each node $v \in [n]$, determine in which community node $v$ is most likely to belong to based on its degree profile computed from the preliminary classification $\sigma'$ (see Section \ref{testing}), and call it $\sigma''_v$

(5) Use $\sigma''_v$ to get new estimates of $p$ and $Q$.

(6) For each node $v \in [n]$, determine in which group $A_1,\dots,A_t$ node $v$ is most likely to belong to based on its degree profile computed from the preliminary classification $\sigma''$ (see Section \ref{testing}). 

\section{An example with real data}\label{data-sec}

We have tested a simplified version of our algorithm on the data from ``The political blogosphere and the 2004 US Election" \cite{blogs}, which contains a list of political blogs that were classified as liberal or conservative, and links between the blogs. 

The algorithm we used has a few major modifications relative to our standard algorithm. First of all, instead of using $N_{r,r'}(v\cdot v')$ as its basic tool for comparing vertices, it uses a different measure, $N'_{r,r'}(v\cdot v')$ which is defined as the fraction of pairs of an edge leaving the ball of radius $r$ centered on $v$ and an edge leaving the ball of radius $r'$ centered on $v'$ which hit the same vertex but are not the same edge. Making the measure a fraction of the pairs rather than a count of pairs was necessary to prevent $N'_{r,r'}(v\cdot v')$ from being massively dependent on the degrees of $v$ and $v'$, which would have resulted in the increased variance in vertex degree obscuring the effects of $\sigma_v$ and $\sigma_{v'}$ on $N'_{r,r'}(v\cdot v')$. The other changes to the definition make the measure somewhat less reliable, but it is still useable as long as the average degree is fairly high and $v\ne v'$.

Secondly, the version of {\it Vertex-comparison-algorithm} we used simply concludes that two vertices, $v$ and $v'$, are in different communities if  $N'_{r,r'}(v\cdot v')$ is below average and the same community otherwise. This is reasonable because of the following facts. For one thing, because the normalization converts the dominant term to a constant, $N'_{r,r'}(v\cdot v')$ is approximately affine in $(\lambda_2/\lambda_1)^{r+r'}\zeta_2(v\cdot v')/n$. Also, as a result of the symmetry between communities, $\zeta_2(v\cdot v)$ is the same for all $v$. So, $\zeta_2(v\cdot v)-2\zeta_2(v\cdot v')+\zeta_2(v'\cdot v')$ is also affine in $\zeta_2(v\cdot v')$. Furthermore, there are only two possible values of $\zeta_2(v\cdot v')$ by symmetry, and $\lambda_2>0$, so $\zeta_2(v\cdot v)-2\zeta_2(v\cdot v')+\zeta_2(v'\cdot v')>0$ iff $(\lambda_2/\lambda_1)^{r+r'}\zeta_2(v\cdot v')/n$ is below average. Finally, because both communities have the same average degree, $\zeta_1(v,v')$ is independent of $v$ and $v'$ so $\zeta_1(v\cdot v)-2\zeta_1(v\cdot v')+\zeta_1(v'\cdot v')$ is always $0$. The version of {\it Vertex-classification-algorithm} we used is comparably modified.

Finally, our algorithm generates reference vertices by repeatedly picking two vertices at random and comparing them. If it concludes that they are in different communities and they both have above-average degree, it accepts them as reference vertices; otherwise it tries again. Requiring above-average degree is useful because a higher degree vertex is less likely to have its neighborhood distorted by a couple of atypical neighbors.

Out of $40$ trials, the resulting algorithm gave a reasonably good classification $37$ times. Each of these classified all but $56$ to $67$ of the $1222$ vertices in the graph's main component correctly. The state-of-the-art described in \cite{harrison} gives a lowest value at 58, with the best algorithms around $60$, while algorithms regularized spectral methods such as the one in \cite{rohe80} obtain about 80 errors.

\begin{figure}[h]
\centering
  \includegraphics[width=.5\linewidth]{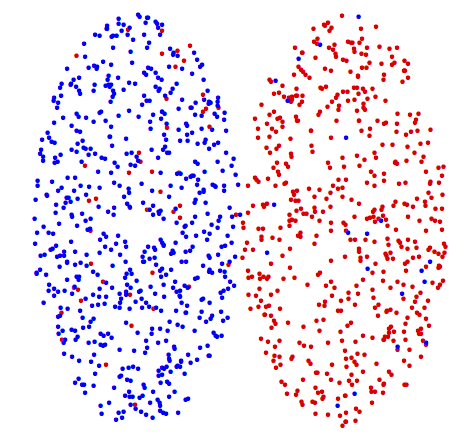}
  \caption{Visual representation of the clustering obtained on the Adamic and Glance '05 blog data.}
  \label{blog_plot}
\end{figure}

\section{Open problems}
The current result should also extend directly to a slowly growing number of communities (e.g., up to logarithmic). It would be interesting to extend the current approach to smaller sized communities or larger numbers of communities (watching the complexity scaling with the number of communities), as well as more general models with corrected-degrees, labeled-edges, or overlapping communities (though linear-sized overlapping communities can be treated with the approach of \cite{colin1}).

\bibliographystyle{amsalpha}
\bibliography{gen_sbm}

\newcommand{\etalchar}[1]{$^{#1}$}
\providecommand{\bysame}{\leavevmode\hbox to3em{\hrulefill}\thinspace}
\providecommand{\MR}{\relax\ifhmode\unskip\space\fi MR }
\providecommand{\MRhref}[2]{%
  \href{http://www.ams.org/mathscinet-getitem?mr=#1}{#2}
}
\providecommand{\href}[2]{#2}
\begin{thebibliography}{XWZ{\etalchar{+}}14}

\bibitem[ABFX08]{mixed-core}
E.~M. Airoldi, D.~M. Blei, S.~E. Fienberg, and E.~P. Xing, \emph{Mixed
  membership stochastic blockmodels}, J. Mach. Learn. Res. \textbf{9} (2008),
  1981--2014.

\bibitem[ABH14]{abh}
E.~Abbe, A.~S. Bandeira, and G.~Hall, \emph{Exact recovery in the stochastic
  block model}, Available at ArXiv:1405.3267. (2014).

\bibitem[AG05]{blogs}
Lada~A. Adamic and Natalie Glance, \emph{The political blogosphere and the 2004
  u.s. election: Divided they blog}, Proceedings of the 3rd International
  Workshop on Link Discovery (New York, NY, USA), LinkKDD '05, ACM, 2005,
  pp.~36--43.

\bibitem[AS15]{colin1}
E.~Abbe and C.~Sandon, \emph{Community detection in general stochastic block
  models: fundamental limits and efficient recovery algorithms},
  arXiv:1503.00609 (2015).

\bibitem[Ban15]{afonso_single}
A.~S. Bandeira, \emph{Random laplacian matrices and convex relaxations},
  arXiv:1504.03987 (2015).

\bibitem[BC09]{bickel}
P.~J. Bickel and A.~Chen, \emph{A nonparametric view of network models and
  newmanÐgirvan and other modularities}, Proceedings of the National Academy of
  Sciences (2009).

\bibitem[BCLS87]{bui}
T.N. Bui, S.~Chaudhuri, F.T. Leighton, and M.~Sipser, \emph{Graph bisection
  algorithms with good average case behavior}, Combinatorica \textbf{7} (1987),
  no.~2, 171--191 (English).

\bibitem[BCS15]{borgs_nips}
C.~Borgs, J.~Chayes, and A.~Smith, \emph{Private graphon estimation for sparse
  graphs}, In preparation (2015).

\bibitem[BH14]{new-xu}
J.~Xu B.~Hajek, Y.~Wu, \emph{Achieving exact cluster recovery threshold via
  semidefinite programming}, arXiv:1412.6156 (2014).

\bibitem[Bop87]{boppana}
R.B. Boppana, \emph{Eigenvalues and graph bisection: An average-case analysis},
  In 28th Annual Symposium on Foundations of Computer Science (1987), 280--285.

\bibitem[CG15]{harrison}
A.~Y.~Zhang H.~H.~Zhou C.~Gao, Z.~Ma, \emph{Achieving optimal misclassification
  proportion in stochastic block model}, arXiv:1505.03772 (2015).

\bibitem[CI01]{carson}
T.~Carson and R.~Impagliazzo, \emph{Hill-climbing finds random planted
  bisections}, Proc. 12th Symposium on Discrete Algorithms (SODA 01), ACM
  press, 2001, 2001, pp.~903--909.

\bibitem[CK99]{condon}
A.~Condon and R.~M. Karp, \emph{Algorithms for graph partitioning on the
  planted partition model}, Lecture Notes in Computer Science \textbf{1671}
  (1999), 221--232.

\bibitem[Co10]{coja-sbm}
A.~Coja-oghlan, \emph{Graph partitioning via adaptive spectral techniques},
  Comb. Probab. Comput. \textbf{19} (2010), no.~2, 227--284.

\bibitem[CRV15]{new-vu}
P.~Chin, A.~Rao, and V.~Vu, \emph{Stochastic block model and community
  detection in the sparse graphs: A spectral algorithm with optimal rate of
  recovery}, arXiv:1501.05021 (2015).

\bibitem[CSX12]{sbm-algos}
Y.~Chen, S.~Sanghavi, and H.~Xu, \emph{{Clustering Sparse Graphs}},
  arXiv:1210.3335 (2012).

\bibitem[CWA12]{choi}
D.~S. Choi, P.~J. Wolfe, and E.~M. Airoldi, \emph{Stochastic blockmodels with a
  growing number of classes}, Biometrika (2012).

\bibitem[CY06]{ppi2}
J.~Chen and B.~Yuan, \emph{Detecting functional modules in the yeast
  proteinÐprotein interaction network}, Bioinformatics \textbf{22} (2006),
  no.~18, 2283--2290.

\bibitem[DF89]{dyer}
M.E. Dyer and A.M. Frieze, \emph{The solution of some random {NP}-hard problems
  in polynomial expected time}, Journal of Algorithms \textbf{10} (1989),
  no.~4, 451 -- 489.

\bibitem[DKMZ11]{decelle}
A.~Decelle, F.~Krzakala, C.~Moore, and L.~Zdeborov\'a, \emph{Asymptotic
  analysis of the stochastic block model for modular networks and its
  algorithmic applications}, Phys. Rev. E \textbf{84} (2011), 066106.

\bibitem[FMW85]{sbm3}
S.~E. Fienberg, M.~M. Meyer, and S.~S. Wasserman, \emph{Statistical analysis of
  multiple sociometric relations}, Journal of The American Statistical
  Association (1985), 51--67.

\bibitem[GB13]{prem}
P.~K. Gopalan and D.~M. Blei, \emph{Efficient discovery of overlapping
  communities in massive networks}, Proceedings of the National Academy of
  Sciences (2013).

\bibitem[GN02]{newman-girvan}
M.~Girvan and M.~E.~J. Newman, \emph{Community structure in social and
  biological networks}, Proceedings of the National Academy of Sciences
  \textbf{99} (2002), no.~12, 7821--7826.

\bibitem[GV14]{sbm-groth}
O.~{Gu{\'e}don} and R.~{Vershynin}, \emph{{Community detection in sparse
  networks via Grothendieck's inequality}}, ArXiv:1411.4686 (2014).

\bibitem[HLL83]{holland}
P.~W. Holland, K.~Laskey, and S.~Leinhardt, \emph{{Stochastic blockmodels:
  First steps}}, Social Networks \textbf{5} (1983), no.~2, 109--137.

\bibitem[JS98]{jerrum}
Mark Jerrum and Gregory~B. Sorkin, \emph{The metropolis algorithm for graph
  bisection}, Discrete Applied Mathematics \textbf{82} (1998), no.~1Ð3, 155 --
  175.

\bibitem[JTZ04]{gene-survey}
D.~Jiang, C.~Tang, and A.~Zhang, \emph{Cluster analysis for gene expression
  data: a survey}, Knowledge and Data Engineering, IEEE Transactions on
  \textbf{16} (2004), no.~11, 1370--1386.

\bibitem[KN11]{newman2}
B.~Karrer and M.~E.~J. Newman, \emph{Stochastic blockmodels and community
  structure in networks}, Phys. Rev. E \textbf{83} (2011), 016107.

\bibitem[LLDM08]{sparse-network1}
J.~Leskovec, K.~J. Lang, A.~Dasgupta, and M.~W. Mahoney, \emph{Statistical
  properties of community structure in large social and information networks},
  Proceedings of the 17th international conference on World Wide Web (New York,
  NY, USA), WWW '08, ACM, 2008, pp.~695--704.

\bibitem[LSY03]{amazon}
G.~Linden, B.~Smith, and J.~York, \emph{Amazon.com recommendations:
  Item-to-item collaborative filtering}, IEEE Internet Computing \textbf{7}
  (2003), no.~1, 76--80.

\bibitem[Mas14]{massoulie-STOC}
L.~Massouli{\'e}, \emph{{Community detection thresholds and the weak Ramanujan
  property}}, {STOC 2014: 46th Annual Symposium on the Theory of Computing}
  (New York, United States), June 2014, pp.~1--10.

\bibitem[McS01]{mcsherry}
F.~McSherry, \emph{Spectral partitioning of random graphs}, In 42nd Annual
  Symposium on Foundations of Computer Science (2001), 529--537.

\bibitem[MNSa]{mossel2}
E.~Mossel, J.~Neeman, and A.~Sly, \emph{Belief propagation, robust
  reconstruction, and optimal recovery of block models}, Arxiv:arXiv:1309.1380.

\bibitem[MNSb]{mossel-consist}
\bysame, \emph{Consistency thresholds for binary symmetric block models},
  Arxiv:arXiv:1407.1591. To appear in STOC15.

\bibitem[MNS12]{Mossel_SBM1}
E.~Mossel, J.~Neeman, and A.~Sly, \emph{Stochastic block models and
  reconstruction}, Available online at arXiv:1202.1499 [math.PR] (2012).

\bibitem[MNS14]{Mossel_SBM2}
\bysame, \emph{A proof of the block model threshold conjecture}, Available
  online at arXiv:1311.4115 [math.PR] (2014).

\bibitem[MPN{\etalchar{+}}99]{marcotte}
E.M. Marcotte, M.~Pellegrini, H.-L. Ng, D.W. Rice, T.O. Yeates, and
  D.~Eisenberg, \emph{Detecting protein function and protein-protein
  interactions from genome sequences}, Science \textbf{285} (1999), no.~5428,
  751--753.

\bibitem[NWS]{social1}
M.~E.~J. Newman, D.~J. Watts, and S.~H. Strogatz, \emph{Random graph models of
  social networks}, Proc. Natl. Acad. Sci. USA \textbf{99}, 2566--2572.

\bibitem[PSD00]{genetics}
J.~K. Pritchard, M.~Stephens, and P.~Donnelly, \emph{{Inference of Population
  Structure Using Multilocus Genotype Data}}, Genetics \textbf{155} (2000),
  no.~2, 945--959.

\bibitem[QR13]{rohe80}
T.~Qin and K.~Rohe, \emph{Regularized spectral clustering under the
  degree-corrected stochastic blockmodel}, Advances in Neural Information
  Processing Systems 26 (C.j.c. Burges, L.~Bottou, M.~Welling, Z.~Ghahramani,
  and K.q. Weinberger, eds.), 2013, pp.~3120--3128.

\bibitem[RCY11]{rohe}
K.~Rohe, S.~Chatterjee, and B.~Yu, \emph{Spectral clustering and the
  high-dimensional stochastic blockmodel}, The Annals of Statistics \textbf{39}
  (2011), no.~4, 1878--1915.

\bibitem[SHB07]{image2}
M.~Sonka, V.~Hlavac, and R.~Boyle, \emph{Image processing, analysis, and
  machine vision}, Thomson-Engineering, 2007.

\bibitem[SM97]{image1}
J.~Shi and J.~Malik, \emph{Normalized cuts and image segmentation}, IEEE
  Transactions on Pattern Analysis and Machine Intelligence \textbf{22} (1997),
  888--905.

\bibitem[SN97]{snij}
T.~A.~B. Snijders and K.~Nowicki, \emph{{Estimation and Prediction for
  Stochastic Blockmodels for Graphs with Latent Block Structure}}, Journal of
  Classification \textbf{14} (1997), no.~1, 75--100.

\bibitem[SPT{\etalchar{+}}01]{tumor}
T.~Sorlie, C.M. Perou, R.~Tibshirani, T.~Aas, S.~Geisler, H.~Johnsen,
  T.~Hastie, Mi.B. Eisen, M.~van~de Rijn, S.S. Jeffrey, T.~Thorsen, H.~Quist,
  J.C. Matese, P.O. Brown, D.~Botstein, P.E. Lonning, and A.~Borresen-Dale,
  \emph{Gene expression patterns of breast carcinomas distinguish tumor
  subclasses with clinical implications}, no.~19, 10869--10874.

\bibitem[Str01]{sparse-network2}
S.~H. Strogatz, \emph{{Exploring complex networks}}, Nature \textbf{410}
  (2001), no.~6825, 268--276.

\bibitem[Ver86]{verdu-hell}
S.~Verd\'u, \emph{Asymptotic error probability of binary hypothesis testing for
  poisson point-process observations (corresp.)}, Information Theory, IEEE
  Transactions on \textbf{32} (1986), no.~1, 113--115.

\bibitem[Vu14]{Vu-arxiv}
V.~Vu, \emph{A simple svd algorithm for finding hidden partitions}, Available
  online at arXiv:1404.3918 (2014).

\bibitem[WBB76]{sbm1}
H.~C. White, S.~A. Boorman, and R.~L. Breiger, \emph{Social structure from
  multiple networks}, American Journal of Sociology \textbf{81} (1976),
  730--780.

\bibitem[WW87]{sbm4}
Y.~J. Wang and G.~Y. Wong, \emph{{Stochastic blockmodels for directed graphs}},
  Journal of the American Statistical Association (1987), 8--19.

\bibitem[XWZ{\etalchar{+}}14]{xu-rec}
J.~Xu, R.~Wu, K.~Zhu, B.~Hajek, R.~Srikant, and L.~Ying, \emph{Jointly
  clustering rows and columns of binary matrices: Algorithms and trade-offs},
  SIGMETRICS Perform. Eval. Rev. \textbf{42} (2014), no.~1, 29--41.

\bibitem[YC14]{chen-xu}
J.~Xu Y.~Chen, \emph{Statistical-computational tradeoffs in planted problems
  and submatrix localization with a growing number of clusters and
  submatrices}, arXiv:1402.1267 (2014).

\bibitem[YP14]{prout}
S.~Yun and A.~Proutiere, \emph{Accurate community detection in the stochastic
  block model via spectral algorithms}, arXiv:1412.7335 (2014).

\end{thebibliography}

\section{The {\tt Agnostic-sphere-comparison} algorithm in details}\label{details}
Recall the following motivation and definitions.

\begin{definition}\label{def-n1}
For any vertex $v$, let $N_r(v)$ be the set of all vertices with shortest path to $v$ of length $r$. We also refer to the vector with $i$-th entry equal to the number of vertices in $N_r(v)$ that are in community $i$ as $N_r(v)$. If there are multiple graphs that $v$ could be considered a vertex in, let $N_{r[G]}(v)$ be the set of all vertices with shortest paths in $G$ to $v$ of length $r$.
\end{definition}

One could probably determine $PQ$ and $e_{\sigma}$ given the values of $(PQ)^re_{\sigma_v}$ for a few different $r$, but using $N_r(v)$ to approximate that would require knowing how many of the vertices in $N_r(v)$ are in each community. So, we attempt to get information relating to how many vertices in $N_r(v)$ are in each community by checking how it connects to $N_{r'}(v')$ for some vertex $v'$ and integer $r'$. The obvious way to do this would be to compute the cardinality of their intersection. Unfortunately, whether a given vertex in community $i$ is in $N_r(v)$ is not independent of whether it is in $N_{r'}(v')$, which causes the cardinality of $|N_r(v)\cap N_{r'}(v')|$ to differ from what one would expect badly enough to disrupt plans to use it for approximations.

In order to get around this, we randomly assign every edge in $G$ to a set $E$ with probability $c$. We hence define the following. 

\begin{definition}
For any vertices $v, v'\in G$, $r,r'\in \mathbb{Z}$, and subset of $G$'s edges $E$, let $N_{r,r'[E]}(v\cdot v')$ be the number of pairs of vertices $(v_1,v_2)$ such that $v_1\in N_{r[G\backslash E]}(v)$, $v_2\in N_{r'[G\backslash E]}(v')$, and $(v_1,v_2)\in E$.
\end{definition}

Note that $E$ and $G\backslash E$ are disjoint; however, $G$ is sparse enough that even if they were generated independently a given pair of vertices would have an edge between them in both with probability $O(\frac{1}{n^2})$. So, $E$ is approximately independent of $G\backslash E$. Thus, for any $v_1\in N_{r[G/E]}(v)$ and $v_2\in N_{r'[G/E]}(v')$, $(v_1,v_2)\in E$ with a probability of approximately $cQ_{\sigma_{v_1},\sigma_{v_2}}/n$. As a result,

\begin{align*}
N_{r,r'}[E](v\cdot v')&\approx N_{r[G\backslash E]}(v)\cdot \frac{cQ}{n} N_{r'[G\backslash E]}(v')\\
&\approx ((1-c)PQ)^re_{\sigma_v}\cdot \frac{cQ}{n} ((1-c)PQ)^{r'}e_{\sigma_{v'}}\\
&=c(1-c)^{r+r'}e_{\sigma_v}\cdot Q(PQ)^{r+r'}e_{\sigma_{v'}}/n
\end{align*}

 Let $\lambda_1,...,\lambda_h$ be the distinct eigenvalues of $PQ$, ordered so that $|\lambda_1|\ge|\lambda_2|\ge...\ge|\lambda_h|\ge 0$. Also define $h'$ so that $h'=h$ if $\lambda_h\ne 0$ and $h'=h-1$ if $\lambda_h=0$. If $W_i$ is the eigenspace of $PQ$ corresponding to the eigenvalue $\lambda_i$, and $P_{W_i}$ is the projection operator on to $W_i$, then 
\begin{align}
N_{r,r'[E]}(v\cdot v')&\approx c(1-c)^{r+r'}e_{\sigma_v}\cdot Q(PQ)^{r+r'}e_{\sigma_{v'}}/n\\
&=\frac{c(1-c)^{r+r'}}{n} \left(\sum_i P_{W_i}(e_{\sigma_v})\right)\cdot Q(PQ)^{r+r'}\left(\sum_j P_{W_j}(e_{\sigma_{v'}})\right)\\
&=\frac{c(1-c)^{r+r'}}{n} \sum_{i,j}P_{W_i}(e_{\sigma_v})\cdot Q(PQ)^{r+r'}P_{W_j}(e_{\sigma_{v'}})\\
&=\frac{c(1-c)^{r+r'}}{n} \sum_{i,j}P_{W_i}(e_{\sigma_v})\cdot P^{-1}(\lambda_j)^{r+r'+1}P_{W_j}(e_{\sigma_{v'}})\\
&=\frac{c(1-c)^{r+r'}}{n} \sum_{i}\lambda_i^{r+r'+1}P_{W_i}(e_{\sigma_v})\cdot P^{-1}P_{W_i}(e_{\sigma_{v'}})\label{Nform}
\end{align}
where the final equality holds because for all $i\ne j$, 
\begin{align*}
\lambda_i  P_{W_i}(e_{\sigma_v})\cdot P^{-1} P_{W_j}(e_{\sigma_{v'}})&=(PQ  P_{W_i}(e_{\sigma_v}))\cdot P^{-1} P_{W_j}(e_{\sigma_{v'}})\\
&= P_{W_i}(e_{\sigma_v})\cdot Q P_{W_j}(e_{\sigma_{v'}})\\
&= P_{W_i}(e_{\sigma_v})\cdot P^{-1} \lambda_j P_{W_j}(e_{\sigma_{v'}}),
\end{align*}
and since $\lambda_i\ne \lambda_j$, this implies that $P_{W_i}(e_{\sigma_v})\cdot P^{-1} P_{W_j}(e_{\sigma_{v'}})=0$. In order to simplify the terminology,
\begin{definition}
Let $\zeta_i(v\cdot v')=P_{W_i}(e_{\sigma_v})\cdot P^{-1}P_{W_i}(e_{\sigma_{v'}})$ for all $i$, $v$, and $v'$.
\end{definition}
Equation \eqref{Nform} is dominated by the $\lambda_1^{r+r'+1}$ term, so getting good estimate of the $\lambda_2^{r+r'+1}$ through $\lambda_{h'}^{r+r'+1}$ terms requires cancelling it out somehow. As a start, if $\lambda_1>\lambda_2>\lambda_3$ then 
\begin{align*}
&N_{r+2,r'[E]}(v\cdot v')\cdot N_{r,r'[E]}(v\cdot v')-N^2_{r+1,r'[E]}(v\cdot v')\\
&\approx \frac{c^2(1-c)^{2r+2r'+2}}{n^2}(\lambda_1^2+\lambda_2^2-2\lambda_1\lambda_2)\lambda_1^{r+r'+1}\lambda_2^{r+r'+1} \zeta_1(v\cdot v')\zeta_2(v\cdot v')
\end{align*}

Note that the left hand side of this expression is equal to 
$\det{}\begin{vmatrix}N_{r,r'[E]}(v\cdot v')&N_{r+1,r'[E]}(v\cdot v')\\
N_{r+1,r'[E]}(v\cdot v')&N_{r+2,r'[E]}(v\cdot v')\end{vmatrix}$. More generally, in order to get an expression that can be used to estimate the $\lambda_i$ and $\zeta_i(v\cdot v')$, we consider the determinant of the following.
\begin{definition}
Let $M_{m,r,r'[E]}(v\cdot v')$ be the $m\times m$ matrix such that $M_{m,r,r'[E]}(v\cdot v')_{i,j}=N_{r+i+j,r'[E]}(v\cdot v')$ for each $i$ and $j$.
\end{definition}

To the degree that approximation $\ref{NformB}$ holds and $c$ is small, each column of $M_{m,r,r'[E]}(v\cdot v')$ is a linear combination of the vectors 
\[\frac{c(1-c)^{r+r'}}{n}\zeta_i(v\cdot v')\lambda_i^{r+r'}[1,\lambda_i,\lambda_i^2,...,\lambda_i^{m-1}]^{t}\]
 with coefficients that depend only on $\{\lambda_1,...,\lambda_{h}\}$. So, by linearity of the determinant in one column, $\det(M_{m,r,r'[E]}(v\cdot v'))$ is a linear combination of these vector's wedge products with coefficients that are independent of $r$ and $r'$. By antisymmetry of wedge products, only the products that use $m$ different such vectors contribute to the determinant, and the products involving the eigenvalues of highest magnitude will dominate. As a result, there exist constants $\gamma(\lambda_1,...,\lambda_m)$ and $\gamma'(\lambda_1,...,\lambda_m)$ such that 
\[\det(M_{m,r,r'[E]}(v\cdot v'))\approx \frac{c^m(1-c)^{m(r+r')}}{n^m}\gamma(\lambda_1,...,\lambda_m)\prod_{i=1}^m \lambda_i^{r+r'+1}\zeta_i(v\cdot v')\]
 if $|\lambda_m|>|\lambda_{m+1}|$, and 
\begin{align*}
&\det(M_{m,r,r'[E]}(v\cdot v'))\\
&\approx \frac{c^m(1-c)^{m(r+r')}}{n^m}\cdot\prod_{i=1}^{m-1} \lambda_i^{r+r'+1}\zeta_i(v\cdot v')\\
&\indent \cdot\left(\gamma(\lambda_1,...,\lambda_m) \lambda_m^{r+r'+1}\zeta_m(v\cdot v')+\gamma'(\lambda_1,...,\lambda_m) \lambda_{m+1}^{r+r'+1} \zeta_{m+1}(v\cdot v')\right)
\end{align*}
if $|\lambda_m|=|\lambda_{m+1}|$. In the later case, $\lambda_m=-\lambda_{m+1}$, so either $\gamma(\lambda_1,...,\lambda_m) \lambda_m^{r+r'+1}$ and $\gamma'(\lambda_1,...,\lambda_m) \lambda_{m+1}^{r+r'+1}$ have the same sign or $\gamma(\lambda_1,...,\lambda_m) \lambda_m^{r+r'+2}$ and $\gamma'(\lambda_1,...,\lambda_m) \lambda_m^{r+r'+2}$ have the same sign. In any of these cases, 
\[ \max(|\det M_{m,r,r[E]}(v\cdot v)|,|\det M_{m,r+1,r[E]}(v\cdot v)|)\approx \frac{c^m(1-c)^{m(r+r')}}{n^m}\prod_{i=1}^m |\lambda_i|^{r+r'+1}\]
This suggests the following algorithm for finding $PQ$'s eigenvalues.

\vspace{5 mm}\noindent
{\bf The Basic-Eigenvalue-approximation-algorithm}
The inputs are $(E,c,v)$, where $v$ is a vertex of the graph, $c\in (0,1)$, and $E$ is a subset of $G$'s edges.

The algorithm ouputs a claim about how many distinct nonzero eigenvalues $PQ$ has and a list of approximations of them.

(1) Compute $N_{r[G\backslash E]}(v)$ for each $r$ until $|N_{r[G\backslash E]}(v)|>\sqrt{n}$, and then set $\lambda_1''=\sqrt[2r]{n}/(1-c)$.

(2) Set $r=r'=\frac{2}{3}\log n/\log ((1-c)\lambda''_1)-\sqrt{\ln n}$. Then, compute 
\[\sqrt[2r]{\frac{n \max(|\det M_{m,r,r[E]}(v\cdot v)|,|\det M_{m,r+1,r[E]}(v\cdot v)|)}{c\max(|\det M_{m-1,r,r[E]}(v\cdot v)|,|\det M_{m-1,r+1,r[E]}(v\cdot v)|)}}\]
 until an $m$ is found for which this expression is less than $((1-c)\lambda''_1)^{3/4}+\frac{1}{\sqrt{\ln n}}$. Then, set $h''=m-1$.

(3) Then, set \[|\lambda'_i|=\frac{1}{1-c}\sqrt{\det(M_{i,r+3,r'[E]}(v\cdot v'))/\det(M_{i,r+1,r'[E]}(v\cdot v'))}/\prod_{j=1}^{i-1}(1-c)|\lambda'_j|\] unless $|\det(M_{i,r+1,r'[E]}(v\cdot v'))|<\sqrt{|\det(M_{i,r,r'[E]}(v\cdot v'))|\cdot|\det(M_{i,r+2,r'[E]}(v\cdot v'))|}$, in which case set \[|\lambda'_i|=\frac{1}{1-c}\sqrt{\det(M_{i,r+2,r'[E]}(v\cdot v'))/\det(M_{i,r,r'[E]}(v\cdot v'))}/\prod_{j=1}^{i-1}|\lambda'_j|\] Repeat this for each $i\le h''$

(4) Next, for each $i<h''$, if $||\lambda'_i|-|\lambda'_{i+1}||<\frac{1}{\ln n}$ then set $\lambda'_i=|\lambda'_i|$ and $\lambda'_{i+1}=-|\lambda'_{i+1}|$. For each $i\le h''$ such that $||\lambda'_i|-|\lambda'_{i+1}||\ge\frac{1}{\ln n}$ and $||\lambda'_{i-1}|-|\lambda'_i||\ge\frac{1}{\ln n}$ set\[\lambda'_i=\frac{1}{1-c}\det(M_{i,r+1,r'[E]}(v\cdot v'))/\det(M_{i,r,r'[E]}(v\cdot v'))/\prod_{j=1}^{i-1}\lambda'_j\]

(5) Return $(\lambda'_1,...,\lambda'_{h''})$\vspace{15 mm}

The risk when using this algorithm is that if the set of edges in $v$'s immediate neighborhood is sufficiently atypical it may not work correctly. This can be solved by repeating it for several vertices and taking the median estimates.

\noindent
{\bf The Improved-Eigenvalue-approximation-algorithm}
The input is $c\in(0,1)$

The algorithm ouputs a claim about how many distinct nonzero eigenvalues $PQ$ has and a list of approximations of them.

(1) Create a set of edges $E$, that each of $G$'s edges is independently assigned to with probability $c$.

(2) Randomly select $\sqrt{\ln n}$ of $G$'s vertices, $v[1]$, $v[2]$,..., $v[\sqrt{\ln n}]$.

(3) Run {\it Basic-Eigenvalue-approximation-algorithm(E,c,v[i])} for each $i\le \sqrt{\ln n}$, stopping the algorithm prematurely if it takes more than $O(n\sqrt{\ln n})$ time.

(4) Return $(\lambda'_1,...,\lambda'_{h''})$ where $h''$ and $\lambda_i$ are the median outputs of the executions of {\it Basic-Eigenvalue-approximation-algorithm} for each $i$.\vspace{15 mm}

Now, note that whether or not $|\lambda_m|=|\lambda_{m+1}|$, we have 
\begin{align*}
&\det(M_{m,r+1,r'[E]}(v\cdot v'))-(1-c)^m\lambda_{m+1}\prod_{i=1}^{m-1} \lambda_i \det(M_{m,r,r'[E]}(v\cdot v'))\\
&\approx \frac{c^m}{n^m} \gamma(\lambda_1,...,\lambda_m)\frac{\lambda_m-\lambda_{m+1}}{(1-c)^m\lambda_m}\prod_{i=1}^{m} ((1-c)\lambda_i)^{r+r'+2}\zeta_i(v\cdot v')
\end{align*}

That means that 
\begin{align*}
&\frac{\det(M_{m,r+1,r'[E]}(v\cdot v'))-(1-c)^m\lambda_{m+1}\prod_{i=1}^{m-1} \lambda_i \det(M_{m,r,r'[E]}(v\cdot v'))}{\det(M_{m-1,r+1,r'[E]}(v\cdot v'))-(1-c)^{m-1}\lambda_{m}\prod_{i=1}^{m-2} \lambda_i \det(M_{m-1,r,r'[E]}(v\cdot v'))}\\
&\approx\frac{c}{(1-c)n}\frac{\gamma(\lambda_1,...,\lambda_m)}{\gamma(\lambda_1,...,\lambda_{m-1})}\frac{\lambda_{m-1}(\lambda_m-\lambda_{m+1})}{\lambda_m(\lambda_{m-1}-\lambda_m)}((1-c)\lambda_m)^{r+r'+2} \zeta_m(v\cdot v')
\end{align*}
This fact can be used in combination with estimates of $PQ$'s eigenvalues to approximate $\zeta_i(v\cdot v')$ for arbitrary $v$, $v'$, and $i$ as follows. 

{\bf The Vertex-product-approximation-algorithm} The inputs are \\$(v,v',r,r',E,c,(\lambda'_1,...,\lambda'_{h''}))$, where $v,v'$ are vertices, $r,r'$ are positive integers, $E$ is a subset of $G$'s edges, $c\in (0,1)$, and $\lambda'_i\in\mathbb{R}$ for all $i$. It is assumed that $N_{r''[G\backslash E]}(v)$ has already been computed for $r''\le r+2h''+3$ and that $N_{r''[G\backslash E]}(v')$ has already been computed for $r''\le r'$.

The algorithm outputs $(z_1(v\cdot v'),...,z_{h''}(v\cdot v'))$ such that $z_i(v\cdot v')\approx\zeta_i(v\cdot v')$ for all $i$.

(1) For each $i\le h''$, set 
\begin{align*}
z_i(v\cdot v')&=\frac{\det(M_{i,r+1,r'[E]}(v\cdot v')-(1-c)^i\lambda'_{i+1}\prod_{j=1}^{i-1} \lambda'_j\det(M_{i,r,r'[E]}(v\cdot v')}{\det(M_{i-1,r+1,r'[E]}(v\cdot v')-(1-c)^{i-1}\lambda'_{i}\prod_{j=1}^{i-2} \lambda'_j \det(M_{i-1,r,r'[E]}(v\cdot v')}\\
&\cdot\frac{ n(\lambda'_{i-1}-\lambda'_i)\gamma(\{(1-c)\lambda'_j\},i-1)}{c\lambda'_{i-1}(\lambda'_i-\lambda'_{i+1})\gamma(\{(1-c)\lambda'_j\},i)}((1-c)\lambda'_i)^{-r-r'-1}
\end{align*}

(2) Return $(z_1(v\cdot v'),...,z_{h''}(v\cdot v'))$.\vspace{15 mm}

Of course, this requires $r$ and $r'$ to be large enough that \[\frac{c(1-c)^{r+r'}}{n} \lambda_i^{r+r'+1}\zeta_i(v\cdot v')\] is large relative to the error terms for all $i\le h'$. At a minimum, that requires that $|(1-c)\lambda_i|^{r+r'+1}=\omega(n)$, so \[r+r' >\log( n)/\log((1-c)|\lambda_{h'}|).\] On the flip side, one also needs \[r,r'<\log(n)/\log ((1-c)\lambda_1)\] because otherwise the graph will start running out of vertices before one gets $r$ steps away from $v$ or $r'$ steps away from $v'$.

Furthermore, for any $v$ and $v'$, 
\begin{align*}
0 &\le P_{W_i}(e_{\sigma_v}-e_{\sigma_{v'}})\cdot P^{-1}P_{W_i}(e_{\sigma_v}-e_{\sigma_{v'}})\\
&=\zeta_i(v\cdot v)-2\zeta_i(v\cdot v')+\zeta_i(v'\cdot v')
\end{align*}
with equality for all $i$ if and only if $\sigma_v=\sigma_{v'}$, so sufficiently good approximations of $\zeta_i(v\cdot v),\zeta_i(v\cdot v')$ and $\zeta_i(v'\cdot v')$ can be used to determine which pairs of vertices are in the same community as follows. \\

\noindent
{\bf The Vertex-comparison-algorithm}. 
The inputs are $(v, v', r, r',E, x, c, (\lambda'_1,...,\lambda'_{h''}))$, where $v,v'$ are two vertices, $r, r'$ are positive integers, $E$ is a subset of $G$'s edges, $x$ is a positive real number, $c$ is a real number between $0$ and $1$, and $(\lambda'_1,...,\lambda'_{h''})$ are real numbers.

The algorithm outputs a decision on whether $v$ and $v'$ are in the same community or not. It proceeds as follows. 

(1) Run {\it Vertex-product-approximation-algorithm(v,v',r,r',E,c,$(\lambda'_1,...,\lambda'_{h''})$)},{\it Vertex-product-approximation-algorithm(v,v,r,r',E,c,$(\lambda'_1,...,\lambda'_{h''})$)}, and {\it Vertex-product-approximation-\\algorithm(v',v',r,r',E,c,$(\lambda'_1,...,\lambda'_{h''})$)}.

(2) If $\exists i: z_i(v\cdot v)-2z_i(v\cdot v')+z_i(v'\cdot v')> 5(2x(\min p_j)^{-1/2}+x^2)$ then conclude that $v$ and $v'$ are in different communities. Otherwise, conclude that $v$ and $v'$ are in the same community.\vspace{15 mm}

One could generate a reasonable classification based solely on this method of comparing vertices (with an appropriate choice of the parameters, as later detailed). However, that would require computing $N_{r,r'[E]}(v\cdot v)$ for every vertex in the graph with fairly large $r+r'$, which would be slow. Instead, we use the fact that for any vertices $v$, $v'$, and $v''$ with $\sigma_v=\sigma_{v'}\ne\sigma_{v''}$, 
\begin{align*}
&\zeta_i(v'\cdot v')-2\zeta_i(v\cdot v')+\zeta_i(v\cdot v)=0\\
&\le \zeta_i(v''\cdot v'')-2\zeta_i(v\cdot v'')+\zeta_i(v\cdot v)
\end{align*}
for all $i$, and the inequality is strict for at least one $i$. So, subtracting $\zeta_i(v\cdot v)$ from both sides gives us that
\[\zeta_i(v'\cdot v')-2\zeta_i(v\cdot v')\le \zeta_i(v''\cdot v'')-2\zeta_i(v\cdot v'')\]
for all $i$, and the inequality is still strict for at least one $i$.

So, given a representative vertex in each community, we can determine which of them a given vertex, $v$, is in the same community as without needing to know the value of $\zeta_i(v\cdot v)$ as follows.\vspace{5 mm}

\noindent
{\bf The Vertex-classification-algorithm}. The inputs are $(v[],v', r,r',E,c, \lambda'_1,...,\lambda'_{h''}))$, where $v[]$ is a list of vertices, $v'$ is a vertex, $r,r'$ are positive integers, $E$ is a subset of $G$'s edges, $c$ is a real number between $0$ and $1$, and $(\lambda'_1,...,\lambda'_{h''})$ are real numbers. It is assumed that $z_i(v[\sigma]\cdot v[\sigma] )$ has already been computed for each $i$ and $\sigma$.

The algorithm is supposed to output $\sigma$ such that $v'$ is in the same community as $v[\sigma]$. It works as follows. 

(1) Run {\it Vertex-product-approximation-algorithm($v[\sigma]$,v',r,r',E,c,$(\lambda'_1,...,\lambda'_{h''})$)} for each $\sigma$.

(2) Find a $\sigma$ that minimizes the value of 
\[\max_{\sigma'\ne\sigma, i\le h''}z_i(v[\sigma]\cdot v[\sigma])-2z_i(v[\sigma]\cdot v')- (z_i(v[\sigma']\cdot v[\sigma'])-2z_i(v[\sigma']\cdot v'))\] and conclude that $v'$ is in the same community as $v[\sigma]$.\vspace{15 mm}

This runs fairly quickly if $r$ is large and $r'$ is small because the algorithm only requires focusing on $N_{r'}(v')$ vertices. This leads to the following plan for partial recovery. First, randomly select a set of vertices that is large enough to contain at least one vertex from each community with high probability. Next, compare all of the selected vertices in an attempt to determine which of them are in the same communities. Then, pick one anchor vertex in each community. After that, use the algorithm above to attempt to determine which community each of the remaining vertices is in. As long as there actually was at least one vertex from each community in the initial set and none of the approximations were particularly bad, this should give a reasonably accurate classification. \vspace{5 mm}

\noindent
{\bf The Unreliable-graph-classification-algorithm}. The inputs are $(G,c,m,\epsilon,x, (\lambda'_1,...,\lambda'_{h''}))$, where $G$ is a graph, $c$ is a real number between $0$ and $1$, $m$ is a positive integer, $\epsilon$ is a real number between $0$ and $1$, $x$ is a positive real number, and $(\lambda'_1,...,\lambda'_{h''})$ are real numbers.

The algorithm outputs an alleged list of communities for $G$. It works as follows.

(1) Randomly assign each edge in $G$ to $E$ independently with probability $c$.

(2) Randomly select $m$ vertices in $G$, $v[0],...,v[m-1]$.

(3) Set $r=(1-\frac{\epsilon}{3})\log n/\log ((1-c)\lambda'_1)-\sqrt{\ln n}$ and $r'=\frac{2\epsilon}{3}\cdot \log n/\log((1-c) \lambda'_1)$

(4) Compute $N_{r''[G\backslash E]}(v[i])$ for each $r''<r+2h''+3$ and $0\le i<m$.

(5) Run {\it vertex-comparison-algorithm}(v[i],v[j],r,r',E,x,$(\lambda'_1,...,\lambda'_{h''})$) for every $i$ and $j$

(6) If these give consistent results, randomly select one alleged member of each community $v'[\sigma]$. Otherwise, fail.

(7) For every $v''$ in the graph, compute $N_{r''[G\backslash E]}(v'')$ for each $r''\le r'$. Then, run\newline {\it Vertex-classification-algorithm}(v'[],v'', r,r',E,$(\lambda'_1,...,\lambda'_{h''})$) in order to get a hypothesized classification of $v''$

(8) Return the resulting classification.\\

The risk that this randomly gives a bad classification due to a bad set of initial vertices can be mitigated as follows. First, repeat the previous classification procedure several times. Assuming that the procedure gives a good classification more often than not, the good classifications should comprise a set that contains more than half the classifications and which has fairly little difference between any two elements of the set. Furthermore, any such set would have to contain at least one good classification, so none of its elements could be too bad. So, find such a set and average its classifications together. 

So, the overall {\tt Agnostic-sphere-comparison-algorithm} starts by estimating $PQ$'s eigenvalues. Then, it uses those estimates to pick appropriate values of $x$ and $\epsilon$ for the {\it Unreliable-graph-classification-algorithm}. Finally, it runs it several times and combines the resulting classifications as explained above. The only inputs it requires are the graph itself and some $\delta>0$ such that $p_i\ge \delta$ for all $i$.\vspace{5 mm}

\noindent
{\bf The Reliable-graph-classification-algorithm (i.e., {\tt Agnostic Sphere comparison})}. The inputs are $(G,m,\delta,T(n))$, where $G$ is a graph, $m$ is a positive integer, $\delta$ is a positive real number, and $T$ is a function from the positive integers to itself. 

The algorithm outputs an alleged list of communities for $G$. It works as follows.

(1) Run {\it Improved-Eigenvalue-approximation-algorithm}(.1) in order to compute $(\lambda'_1,...,\lambda'_{h''})$

(2) Let $\lambda''_1=\lambda'_1+2\ln^{-3/2}(n)$, $\lambda''_{h''}=\lambda'_{h''}-2\ln^{-3/2}(n)$, and $k'=\lfloor 1/\delta\rfloor$  

(3) Let $x$ be the smallest rational number of minimal numerator such that 
\[k'(1-\delta)^m+m\cdot 2k'e^{-\frac{x^2(\lambda''_{h''})^2 \delta}{16\lambda''_1(k')^{3/2}((\delta)^{-1/2}+x)}}/\left(1-e^{-\frac{x^2(\lambda''_{h''})^2\delta}{16\lambda''_1(k')^{3/2}((\\delta)^{-1/2}+x)}\cdot((\frac{(\lambda''_{h''})^4}{4(\lambda''_1)^3})-1)}\right)<\frac{1}{2}\]

(4) Let $\epsilon$ be the smallest rational number of the form $\frac{1}{z}$ or $1-\frac{1}{z}$ such that $(2(\lambda''_1)^3/(\lambda''_{h''})^2)^{1-\epsilon/3}<\lambda''_1$ and
$(1+\epsilon/3)>\log(\lambda''_1)/\log ((\lambda''_{h''})^2/2\lambda''_1)$

(5) Let $c$ be the largest unit reciprocal less than $1/9$ such that all of the following hold:
\begin{align*}
&(1-c)(\lambda''_{h''})^4>4(\lambda''_1)^3\\
&(2(1-c)(\lambda''_1)^3/(\lambda''_{h'})^2)^{1-\epsilon/3}<(1-c)\lambda''_1\\
&(1+\epsilon/3)>\log((1-c)\lambda''_1)/\log ((1-c)(\lambda''_{h'})^2/2\lambda''_1)\\
&k'(1-\delta)^m+m\cdot 2k'e^{-\frac{x^2(1-c)(\lambda''_{h'})^2\delta}{16\lambda''_1(k')^{3/2}((\delta)^{-1/2}+x)}}/\left(1-e^{-\frac{x^2(1-c)(\lambda''_{h''})^2\delta}{16\lambda''_1(k')^{3/2}((\delta)^{-1/2}+x)}\cdot((\frac{(1-c)(\lambda''_{h''})^4}{4(\lambda''_1)^3})-1)}\right)<\frac{1}{2}
\end{align*}

(6) Run {\it Unreliable-graph-classification-algorithm}$(G,c,m,\epsilon,x,(\lambda'_1,...,\lambda'_{h''}))$ $T(n)$ times and record the resulting classifications.

(7) Find the smallest $y''$ such that there exists a set of more than half of the classifications no two of which have more than $y''$ disagreement, and discard all classifications not in the set. In this step, define the disagreement between two classifications as the minimum disagreement over all bijections between their communities.

(8) For every vertex in $G$, randomly pick one of the remaining classifications and assert that it is in the community claimed by that classification, where a community from one classification is assumed to correspond to the community it has the greatest overlap with in each other classification.

(9) Return the resulting combined classification.

If the conditions of theorem $2$ are satisfied, then there exists $\delta$ such that $$ \text{{\it Reliable-graph-classification-algorithm}}(G,\ln(4\lfloor 1/\delta\rfloor)/\delta,\delta,\ln n)$$ classifies at least 
\begin{align}
1- \frac{6ke^{-\frac{C \rho^2}{4k}}}{1-e^{-\frac{C\rho^2}{4k}\left(\frac{.9(\lambda_{h'})^2}{\lambda_1^2}\rho^2-1\right)}}
\end{align}
of $G$'s vertices correctly with probability $1-o(1)$ and it runs in $O(n^{1+\epsilon})$ time, for appropriate $C$ and $\rho=\sqrt{\lambda_{h'}^2/4\lambda_1}$.

\section{Appendix}

\subsection{Partial Recovery}\label{partial-sec}
\subsubsection{Formal results}

\begin{theorem}\label{thm1}
For any $\delta>0$, there exists an algorithm ({\tt Agnostic-sphere-comparison}) such that the following holds. Given any $k\in \mathbb{Z}$, $p\in (0,1)^k$ with $|p|=1$, and symmetric matrix $Q$ with no two rows equal, let $\lambda$ be the largest eigenvalue of $PQ$, and $\lambda'$ be the eigenvalue of $PQ$ with the smallest nonzero magnitude. For any $x$, $x'$, and $\epsilon$ such that $x$ is either a unit reciprocal or an integer, $\epsilon$ is a rational number of the form $\frac{1}{z}$ or $1-\frac{1}{z}$, and all of the following hold:
\begin{align*}
&2ke^{-\frac{.9x^2\lambda'^2\min p_i}{16\lambda k^{3/2}((\min p_i)^{-1/2}+x)}}/\left(1-e^{-\frac{.9x^2\lambda'^2\min p_i}{16\lambda k^{3/2}((\min p_i)^{-1/2}+x)}\cdot((\frac{.9\lambda'^4}{4\lambda^3})-1)}\right)<\frac{1}{2}\\
&.9(\lambda'^2/2)^{4}>\lambda^{7}\\
&0<x\le x'<\frac{\lambda k}{\lambda'\min p_i}\\
&(2\lambda^3/\lambda'^2)^{1-\epsilon/3}<\lambda\\
&(1+\epsilon/3)>\log(\lambda)/\log (\lambda'^2/2\lambda)\\
&13(2x'(\min p_j)^{-1/2}+(x')^2)<\min_{\ne0} (w_i(\{v\})- w_i(\{v'\}))\cdot P^{-1}(w_i(\{v\})- w_i(\{v'\}))\\
&\text{Every entry of } Q^k \text{ is positive}\\
&\exists w\in \mathbb{R}^k \text{ such that } QPw=\lambda w, w\cdot Pw=1, \text{ and } x\le \min w_i/2.\\
&\delta\le \min p_i\\
& 8\ln(4\lfloor 1/\delta\rfloor)\lfloor 1/\delta\rfloor e^{-\frac{x^2\lambda'^2\delta}{16\lambda\lfloor 1/\delta\rfloor^{3/2}(\delta^{-1/2}+x)}}/\left(1-e^{-\frac{x^2\lambda'^2\delta}{16\lambda\lfloor 1/\delta\rfloor^{3/2}(\delta^{-1/2}+x)}\cdot((\frac{\lambda'^4}{4\lambda^3})-1)}\right)< \delta\\
&\min p_i>8ke^{-\frac{.9x^{\prime 2}\lambda'^2\min p_i}{16\lambda k^{3/2}((\min p_i)^{-1/2}+x')}}/\left(1-e^{-.9\frac{x^{\prime 2}\lambda'^2\min p_i}{16\lambda k^{3/2}((\min p_i)^{-1/2}+x')}\cdot((\frac{.9\lambda'^4}{4\lambda^3})-1)}\right)
\end{align*}

With probability $1-o(1)$, the algorithm runs in $O(n^{1+\frac{2}{3}\epsilon}\log n)$ time and detects communities in graphs drawn from $\gss(n,p,Q)$ with accuracy at least $1-3y'$ without any input beyond $\delta$ and the graph, where
\[y'=2ke^{-\frac{.9x^{\prime 2}\lambda'^2\min p_i}{16\lambda k^{3/2}((\min p_i)^{-1/2}+x')}}/\left(1-e^{-.9\frac{x^{\prime 2}\lambda'^2\min p_i}{16\lambda k^{3/2}((\min p_i)^{-1/2}+x')}\cdot((\frac{.9\lambda'^4}{4\lambda^3})-1)}\right)\]
\end{theorem}

Considering the way $\delta$, $\epsilon$, $x$, and $x'$ scale when $Q$ is multiplied by a scalar yields the following corollary.

\begin{corollary}\label{partial-delta}
For any $k\in \mathbb{Z}$, $p\in (0,1)^k$ with $|p|=1$, and symmetric  matrix $Q$ with no two rows equal such that $Q^k$ has all positive entries, there exist $\epsilon(c)=O(1/\ln(c))$ such that for all sufficiently large $c$, {\tt Agnostic-sphere-comparison} detects communities in graphs drawn from $\gss(n,p,c Q)$ with accuracy at least $1-e^{-\Omega(c)}$ in $O_n(n^{1+\epsilon(c)})$.
\end{corollary}


If instead of having constant average degree, one has an average degree which increases as $n$ increases, one can slowly reduce $b$, $\delta$, and $\epsilon$ as $n$ increases, leading to the following corollary.

\begin{corollary}
For any $k\in \mathbb{Z}$, $p\in [0,1]^k$ with $|p|=1$, symmetric  matrix $Q$ with no two rows equa such that $Q^m$ has all positive entries for sufficiently large $m$l, and $c(n)$ such that $c=\omega(1)$, {\tt Agnostic-sphere-comparison} detects the communities with accuracy $1-o(1)$ in $\gss(n,p,c(n)Q)$ and runs in $o(n^{1+\epsilon})$ time for all $\epsilon>0$.
\end{corollary}

These corollaries are important as they show that if the entries of the connectivity matrix $Q$ are amplified by a coefficient growing with $n$, almost exact recovery is achieved by ({\tt Agnostic-sphere-comparison}) without parameter knowledge. 

\subsubsection{Proof of Theorem \ref{thm1}}
Proving Theorem \ref{thm1} will require establishing some terminology. First, let $\lambda_1,...,\lambda_h$ be the distinct eigenvalues of $PQ$, ordered so that $|\lambda_1|\ge|\lambda_2|\ge...\ge|\lambda_h|\ge 0$ and if $|\lambda_i|=|\lambda_{i+1}|$ then $\lambda_i>0>\lambda_{i+1}$. Also define $h'$ so that $h'=h$ if $\lambda_h\ne 0$ and $h'=h-1$ if $\lambda_h=0$. In addition to this, let $d$ be the largest sum of a column of $PQ$.
\begin{definition}
For any graph $G$ drawn from $\gss(n,p,Q)$ and any set of vertices in $G$, $V$, let $\overrightarrow{V}$ be the vector such that $\overrightarrow{V}_i$ is the number of vertices in $V$ that are in community $i$. Define $w_1(V)$, $w_2(V)$, ..., $w_h(V)$ such that $\overrightarrow{V}=\sum w_i(V)$ and $w_i(V)$ is an eigenvector of $PQ$ with eigenvalue $\lambda_i$ for each $i$.
\end{definition}

$w_1(V), ...,w_h(V)$ are well defined because $\mathbb{R}^k$ is the direct sum of $PQ$'s eigenspaces. The key intuition behind their importance is that if $V'$ is the set of vertices adjacent to vertices in $V$ then $\overrightarrow{V'}\approx PQ\overrightarrow{V}$, so $w_i(V')\approx PQ\cdot w_i(V)=\lambda_iw_i(V)$.

\begin{definition}
For any vertex $v$, let $N_r(v)$ be the set of all vertices with shortest path to $v$ of length $r$. If there are multiple graphs that $v$ could be considered a vertex in, let $N_{r[G']}(v)$ be the set of all vertices with shortest paths in $G'$ to $v$ of length $r$.
\end{definition}
We also typically refer to $\overrightarrow{N_{r[G']}(v)}$ as simply $N_{r[G']}(v)$, as the context will make it clear whether the expression refers to a set or vector. 

\begin{definition}
A vertex $v$ of a graph drawn from $\gss(n,p,Q)$ is $(R,x)$-good if for all $0\le r<R$ and $w\in \mathbb{R}^k$ with $w\cdot Pw=1$ \[|w\cdot N_{r+1}(v)-w\cdot PQN_r(v)|\le \frac{x\lambda_{h'}}{2}\left(\frac{\lambda_{h'}^2}{2\lambda_1}\right)^{r}\] and $(R,x)$-bad otherwise.
\end{definition}

Note that since any such $w$ can be written as a linear combination of the $e_i$, $v$ is $(R,x)$-good if $|e_i\cdot N_{r+1}(v)-e_i\cdot PQN_r(v)|\le \frac{x\lambda_{h'}}{2}\left(\frac{\lambda_{h'}^2}{2\lambda_1}\right)^{r}\sqrt{p_i/k}$ for all $1\le i\le k$ and $0\le r<R$.

\begin{lemma}
If $v$ is a $(R,x)$-good vertex of a graph drawn from $\gss(n,p,Q)$, then for every $0\le r\le R$, $|N_r(v)|\le \lambda_1^r\sqrt{k}((\min p_i)^{-1/2}+x)$.
\end{lemma}

\begin{proof}
First, note that for any eigenvector of $PQ$, $w$, and $r<R$, \[|(P^{-1}w)\cdot N_{r+1}(v)-(P^{-1}w)\cdot PQN_r(v)|\le \frac{x\lambda_{h'}}{2}\left(\frac{\lambda_{h'}^2}{2\lambda_1}\right)^{r}\sqrt{w\cdot P^{-1}w}\] So, by the triangle inequality,
\begin{align*}
|(P^{-1}w)\cdot N_{r+1}(v)|&\le |(P^{-1} PQw)\cdot N_r(v)|+\frac{x\lambda_{h'}}{2}\left(\frac{\lambda_{h'}^2}{2\lambda_1}\right)^{r}\sqrt{w\cdot P^{-1}w}\\
&\le \lambda_1|(P^{-1}w)\cdot N_r(v)|+x\left(\frac{\lambda_1}{2}\right)^{r+1}\sqrt{w\cdot P^{-1}w}
\end{align*}

Thus, for any $r\le R$, it must be the case that 
\begin{align*}
|(P^{-1}w)\cdot N_r(v)|&\le \lambda_1^r|(P^{-1}w)\cdot N_0(v)|+\sum_{r'=1}^{r} \lambda_1^{r-r'}\cdot x\left(\frac{\lambda_1}{2}\right)^{r'}\sqrt{w\cdot P^{-1}w}\\
&\le \lambda_1^r\left (|w_{\sigma_v}/p_{\sigma_v}|+x\sqrt{w\cdot P^{-1}w}\right)
\end{align*}

Now, define $w_1$,..., $w_h$ such that $PQw_i=\lambda_iw_i$ for each $i$ and $p=\sum_{i=1}^h w_i$. For any $i,j$, 
\begin{align*}
\lambda_iw_i\cdot P^{-1}w_j&=(PQw_i)\cdot P^{-1}w_j\\
&=w_i\cdot P^{-1}PQw_j\\
&=\lambda_jw_i\cdot P^{-1}w_j
\end{align*}
If $i\ne j$, then $\lambda_i\ne \lambda_j$, so this implies that $w_i\cdot P^{-1} w_j=0$. It follows from this that
\begin{align*}
\sum_i w_i\cdot P^{-1} w_i&=\sum_{i,j} w_i\cdot P^{-1}w_j\\
&= \left(\sum_i w_i\right)\cdot P^{-1}\left(\sum_j w_j\right)\\
&=p\cdot P^{-1} p=1
\end{align*}

Also, for any $i$, it is the case that 
\[|(w_i)_{\sigma_v}/p_{\sigma_v}|\le \sqrt{(w_i)_{\sigma_v}\cdot p^{-1}_{\sigma_v}\cdot (w_i)_{\sigma_v}}/\sqrt{p_{\sigma_v}}\le (\min p_i)^{-1/2}\sqrt{w_i\cdot P^{-1}w_i}\]

Therefore, for any $r\le R$, we have that
\begin{align*}
|N_r(v)|&=|(P^{-1}p)\cdot N_r(v)|\\
&\le \sum_i |(P^{-1} w_i)\cdot N_r(v)|\\
&\le \lambda_1^r\sum_i |(w_i)_{\sigma_v}/p_{\sigma_v}|+\lambda_1^r x\sum_i \sqrt{w_i\cdot P^{-1}w_i}\\
&\le \lambda_1^r\sqrt{k}((\min p_i)^{-1/2}+x)
\end{align*}
\end{proof}
The following two lemmas are proved in \cite{colin1}.

\begin{lemma}
Let $k\in \mathbb{Z}$, $p\in (0,1)^k$ with $|p|=1$, $Q$ be a symmetric matrix such that $\lambda_{h'}^4>4\lambda_1^3$, and $0<x<\frac{\lambda_1k}{\lambda_{h'}\min p_i}$. Then there exists \[ y<2ke^{-\frac{x^2\lambda_{h'}^2\min p_i}{16\lambda_1 k^{3/2}((\min p_i)^{-1/2}+x)}}/\left(1-e^{-\frac{x^2\lambda_{h'}^2\min p_i}{16\lambda_1k^{3/2}((\min p_i)^{-1/2}+x)}\cdot((\frac{\lambda_{h'}^4}{4\lambda_1^3})-1)}\right)\] and $R(n)= \omega(1)$  such that at least $1-y$ of the vertices of a graph drawn from $\gss(n,p,Q)$ are $(R(n),x)$-good with probability $1-o(1)$.
\end{lemma}

\begin{lemma}
Let $k\in \mathbb{Z}$, $p\in (0,1)^k$ with $|p|=1$, $Q$ be a symmetric matrix such that $\lambda_{h'}^4>4\lambda_1^3$, $R(n)=\omega(1)$, and $\epsilon>0$ such that $(2\lambda_1^3/\lambda_{h'}^2)^{1-\epsilon/3}<\lambda_1$. A vertex of a graph drawn from $G(p, Q, n)$ is $(R(n),x)$-good but $(\frac{1-\epsilon/3}{\ln\lambda_1}\ln n,x)$-bad with probability $o(1)$.
\end{lemma}

\begin{definition}
For any vertices $v, v'\in G$, $r,r'\in \mathbb{Z}$, and subset of $G$'s edges $E$, let $N_{r,r'[E]}(v\cdot v')$ be the number of pairs of vertices $(v_1,v_2)$ such that $v_1\in N_{r[G\backslash E]}(v)$, $v_2\in N_{r'[G\backslash E]}(v')$, and $(v_1,v_2)\in E$.
\end{definition}

Note that if $N_{r[G\backslash E]}(v)$ and $N_{r'[G\backslash E]}(v')$ have already been computed, $N_{r,r'[E]}(v\cdot v')$ can be computed by means of the following algorithm, where $E[v]=\{v':(v,v')\in E\}$
\begin{algorithm}
Compute-$N_{r,r'[E]}(v\cdot v')$:

for $v_1\in N_{r'[G\backslash E]}(v')$:

$\phantom{xxx}$ for $v_2\in E[v_1]:$

$\phantom{xxxxxx}$ if $v_2\in N_{r[G\backslash E]}(v):$

$\phantom{xxxxxxxxx}$ count=count+1

return count
\end{algorithm}

Note that this runs in $O((d+1)|N_{r'[G\backslash E]}(v')|)$ average time. The plan is to independently put each edge in $G$ in $E$ with probability $c$. Then the probability distribution of $G\backslash E$ will be $\gss(n,p,(1-c)Q)$, so $N_{r[G\backslash E]}(v)\approx ((1-c)PQ)^re_{\sigma_v}$ and $N_{r'[G\backslash E]}(v')\approx ((1-c)PQ)^{r'}e_{\sigma_{v'}}$. So, it will hopefully be the case that \[N_{r,r'[E]}(v\cdot v')\approx ((1-c)PQ)^re_{\sigma_v}\cdot cQ((1-c)PQ)^{r'}e_{\sigma_{v'}}/n= c(1-c)^{r+r'} e_{\sigma_v}\cdot Q(PQ)^{r+r'}e_{\sigma_{v'}}/n.\] More rigorously, we have that:

\begin{lemma}
Choose $p$, $Q$, $G$ drawn from $\gss(n,p,Q)$, $E$ randomly selected from $G$'s edges such that each of $G$'s edges is independently assigned to $E$ with probability $c$, and $v,v'\in G$ chosen independently from $G$'s vertices. Then with probability $1-o(1)$, \[|N_{r,r'[E]}(v\cdot v')-N_{r[G\backslash E]}(v)\cdot cQN_{r'[G\backslash E]}(v')/n|<(1+\sqrt{|N_{r[G\backslash E]}(v)|\cdot |N_{r'[G\backslash E]}(v')|/n})\log n\]
\end{lemma}

\begin{proof}
Roughly speaking, for each $v_1\in N_{r[G\backslash E]}(v)$ and $v_2\in N_{r'[G\backslash E]}(v')$, $(v_1,v_2)\in E$ with probability $cQ_{\sigma_{v_1},\sigma_{v_2}}/n$. This is complicated by the facts that $(v_1,v_1)$ is never in $E$ and no edge is in $G\backslash E$ and $E$. However, this changes the expected value of $N_{r,r'[E]}(v\cdot v')$ given $G\backslash E$ by at most a constant unless $G$ has more than double its expected number of edges, something that happens with probability $o(1)$. Furthermore, whether $(v_1,v_2)$ is in $E$ is independent of whether $(v_1',v_2')$ is in $E$ unless $(v_1',v_2')=(v_1,v_2)$ or $(v_1',v_2')=(v_2,v_1)$. So, the variance of  $N_{r,r'[E]}(v\cdot v')$ is proportional to its expected value, which is $$O(|N_{r[G\backslash E]}(v)|\cdot |N_{r'[G\backslash E]}(v')|/n).$$ $N_{r,r'[E]}(v\cdot v')$ is within $\log n$ standard deviations of its expected value with probability $1-o(1)$, which completes the proof.
\end{proof}

Note that if $\overrightarrow{v}$ is an eigenvector of $(1-c)PQ$, $\sqrt{P}Q\overrightarrow{v}$ is an eigenvector of the symmetric matrix $(1-c)\sqrt{P}Q\sqrt{P}$. So, since eigenvectors of a symmetric matrix with different eigenvalues are orthogonal, we have \[N_{r[G\backslash E]}(v)\cdot cQN_{r'[G\backslash E]}(v')/n=\frac{c}{n} \sum_i w_i(N_{r[G\backslash E]}(v))\cdot Qw_i(N_{r'[G\backslash E]}(v'))\] 

\begin{lemma}[Determinant Lemma] 
Let $0<c<1$, $x>0$, $G$ be drawn from $\gss(n,p,Q)$, $E$ be a subset of $G$'s edges that independently contains each edge with probability $c$, and $m\in\mathbb{Z}^+$. For any $v,v'\in G$ and $r\ge r'\in \mathbb{Z}^+$, such that $((1-c)\lambda_{h'}^2/2)^{r+r'}>\lambda_1^{r+r'}n$ let $M_{m,r,r'[E]}(v\cdot v')$ be the $m\times m$ matrix such that $M_{m,r,r'[E]}(v\cdot v')_{i,j}=N_{r+i+j,r'[E]}(v\cdot v')$ for each $i$ and $j$. There exist $\gamma=\gamma((1-c)\lambda_i,m)$ and $\gamma'=\gamma'((1-c)\lambda_i,m)$ such that $\gamma$ is nonzero and for any $r,r'$, and vertices $v,v'\in G$, then with probability $1-o(1)$, either $v$ is $(r+2m+1,x)$-bad, $v'$ is $(r'+1,x)$-bad, or 
\begin{align*}
&|\det(M_{m,r,r'[E]}(v\cdot v'))-\frac{c^m}{n^m}\prod_{i=1}^{m-1}  w_i(N_{r[G\backslash E]}(v))\cdot Qw_i(N_{r'[G\backslash E]}(v'))\\
&\indent \cdot (\gamma w_m(N_{r[G\backslash E]}(v))\cdot Qw_m(N_{r'[G\backslash E]}(v'))+\gamma'w_{m+1}(N_{r[G\backslash E]}(v))\cdot Qw_{m+1}(N_{r'[G\backslash E]}(v')))|\\
&\le \frac{c^m}{n^m}\ln^{m+1} n(1-c)^{m(r+r')}|\lambda_{m+2}|^{r+r'}\prod_{i=1}^{m-1} |\lambda_i|^{r+r'}\\
&\indent\indent + \frac{c^m}{n^m}\ln^{m+1} n(1-c)^{m(r+r')}|\lambda_m|^{r+r'}|\lambda_{m+1}|^{r+r'}\prod_{i=1}^{m-2} |\lambda_i|^{r+r'}
\end{align*} where we temporarily adopt the convention that if $i>h'$, $\lambda_i=\lambda_{h'}/\sqrt{2}$ and $w_i(S)=0$ for all $S$.

Alternately, if $m=h'+1$ then with probability $1-o(1)$, either $v$ is $(r+2m+1,x)$-bad, $v'$ is $(r'+1,x)$-bad, or 
\begin{align*}
&|\det(M_{m,r,r'[E]}(v\cdot v'))|\\
&\le \frac{c^{h'+1}}{n^{h'+1}}\log^2(n)(1-c)^{h'(r+r')}\prod_{i=1}^{h'} |\lambda_i|^{r+r'}\left(\frac{((1-c)\lambda_1)^{2r}}{n}+((1-c)\lambda_1)^{r/2}\right)\cdot (1-c)^{r'}\lambda_1^{r'}
\end{align*}
\end{lemma}

\begin{proof}
For each $1\le l\le 2m$ and $1\le i\le h$, let \[x_l(i)=\frac{c}{n}(1-c)^{l}\lambda_i^{l}  w_i(N_{r[G\backslash E]}(v))\cdot Qw_i(N_{r'[G\backslash E]}(v'))\]
Next, for each $1\le i\le h$ and $0\le l\le m$, let $u_l(i)$ be the column vector thats $j$th entry is $x_{l+j}(i)$ for $1\le j\le m$. Also, for each $1\le l\le m$, let $u_{l}(h+1)$ be the length $m$ column vector thats $j$th entry is $N_{r+l+j,r'[E]}(v\cdot v')-\sum_{i=1}^h x_{l+j}(i)$ for $1\le j\le m$. Note that for each $1\le l\le m$, the $l$th column of $M_{m,r,r'[E]}(v\cdot v')$ is $\sum_{i=1}^{h+1} u_l(i)$. So,
\[\det(M_{m,r,r'[E]}(v\cdot v'))=\sum_{i\in (\mathbb{Z}\cap [1,h+1])^m} \det([u_1(i_1),u_2(i_2),...,u_m(i_m)])\]
For any $i\in (\mathbb{Z}\cap [1,h+1])^m$, if there exist $j\ne j'$ such that $i_j=i_{j'}\le h'$, then $u_j(i_j)=(1-c)^{j-j'}\lambda_{i_j}^{j-j'} u_{j'}(i_{j'})$, which implies that $ \det([u_1(i_1),u_2(i_2),...,u_m(i_m)])=0$.

If $m\le h$ and $i$ is some permutation of the integers from $1$ to $m$, then \[\det([u_1(i_1),u_2(i_2),...,u_m(i_m)])=\left(\prod_{j=1}^m (1-c)^j\lambda_{i_j}^j\right) sgn(i)\det([u_0(1),u_0(2),...,u_0(m)])\] 

The $j$th column of this matrix is proportional to $\frac{c}{n}w_j(N_{r[G\backslash E]}(v))\cdot Qw_j(N_{r'[G\backslash E]}(v'))$, so there exists some $\gamma=\gamma(\{(1-c)\lambda_j\},m)$ such that the sum of all such terms is \[\frac{c^m}{n^m}\gamma\prod_{j=1}^{m}  w_j(N_{r[G\backslash E]}(v))\cdot Qw_j(N_{r'[G\backslash E]}(v'))\] Alternately, the sum of all such terms is equal to \[\det\left(\left[\sum_{j=1}^m u_1(j),\sum_{j=1}^m u_2(j),...,\sum_{j=1}^m u_m(j)\right]\right)\] If $x_l(0)\ne 0$ for each $1\le l\le m$ and $u'\in \mathbb(R)^m$ such that \\$\left(\left[\sum_{j=1}^m u_1(j),\sum_{j=1}^m u_2(j),...,\sum_{j=1}^m u_m(j)\right]\right)u'=0$, then for each $1\le i\le m$, 
\begin{align*}
\sum_{l=1}^m u'_l\sum_{j=1}^m x_{l+i}(j)&=0\\ 
\sum_{j=1}^m \sum_{l=1}^m u'_l \frac{c}{n}(1-c)^{l+i}\lambda_j^{l+i}  w_j(N_{r[G\backslash E]}(v))\cdot Qw_j(N_{r'[G\backslash E]}(v'))&=0\\
\sum_{j=1}^m w_j(N_{r[G\backslash E]}(v))\cdot Qw_j(N_{r'[G\backslash E]}(v')) (1-c)^{i}\lambda_j^{i}\sum_{l=1}^m u'_l\cdot (1-c)^{l}\lambda_j^{l} &=0\\ 
\end{align*}
That can only hold for all such $i$ if $\sum_{l=1}^m u'_l (1-c)^{l}\lambda_j^{l} =0$ for all $1\le j\le m$, and that can only be the case if $u'=0$. Therefore, the determinant is nonzero unless $x_l(0)=0$ for some $l$, which implies that $\gamma\ne 0$. If $m< h$ then by similar logic, there exists $\gamma'=\gamma'(\{(1-c)\lambda_i\},m)$ such that the sum of all terms for which $i$ is a permutation of the integers from $1$ to $m-1$ and $m+1$ is 
\[\frac{c^m}{n^m}\gamma'\cdot w_{m+1}(N_{r[G\backslash E]}(v))\cdot Qw_{m+1}(N_{r'[G\backslash E]}(v'))\prod_{i=1}^{m-1}  w_i(N_{r[G\backslash E]}(v))\cdot Qw_i(N_{r'[G\backslash E]}(v'))\]

That accounts for all $i\in (\mathbb{Z}\cap [1,h+1])^m$ except for some of those such that there exists $j$ such that $i_j\ge\min(m+2,h+1)$ or there exist $j,j'$ such that $i_j=m$ and $i_{j'}=m+1$.

If $v$ is $(r+2m+1,x)$-good, then 
\[w_i(N_{r[G/E]}(v))P^{-1}w_i(N_{r[G/E]}(v))\le ((\min p_j)^{-1/2}+x)^2(1-c)^{2r}\lambda_i^{2r}\]

for all $i$. Similarly, if $v'$ is $(r'+1,x)$-good then 
\[w_i(N_{r'[G/E]}(v'))P^{-1}w_i(N_{r'[G/E]}(v'))\le ((\min p_j)^{-1/2}+x)^2(1-c)^{2r'}\lambda_i^{2r'}\]

for all $i$. If both hold, then $|x_l(i)|\le \frac{c}{n}(1-c)^{r+r'+l}|\lambda_i|^{r+r'+l+1}((\min p_j)^{-1/2}+x)^2$ for all $i$. Furthermore, for any $l$ and $j$, 
\begin{align*}
|u_l(h+1)_j|&=|N_{r+l+j,r'[E]}(v\cdot v')-\sum_{i=1}^h x_{l+j}(i)|\\
&\le |N_{r+l+j,r'[E]}(v\cdot v')-\frac{c}{n}N_{r+l+j[G\backslash E]}(v)\cdot QN_{r'[G\backslash E]}(v')|\\
&\indent\indent+|\frac{c}{n}N_{r+l+j[G\backslash E]}(v)\cdot QN_{r'[G\backslash E]}(v')-\sum_{i=1}^h x_{l+j}(i)|\\
&\le (1+\sqrt{|N_{r+l+j[G\backslash E]}(v)|\cdot |N_{r'[G\backslash E]}(v')|/n})\log n\\
&\indent\indent +\frac{c}{n}\sum_{i=1}^h |w_i(N_{r+l+j[G\backslash E]}(v))\cdot Qw_i(N_{r'[G\backslash E]}(v'))\\
&\indent\indent\indent\indent-(1-c)^{l+j}\lambda_i^{l+j}w_i(N_{r[G\backslash E]}(v))\cdot Qw_i(N_{r'[G\backslash E]}(v'))|
\end{align*}
hence
\begin{align*}
|u_l(h+1)_j|&\le (1+((1-c)\lambda_1)^{(r+r'+l+j)/2})\sqrt{k}((\min p_i)^{-1/2}+x)/\sqrt{n})\log n\\
&\indent\indent+\frac{c}{n}\sum_{i=1}^h (1-c)^{r+l+j-1}x\lambda_{h'}\left(\frac{\lambda_{h'}^2}{2\lambda_1}\right)^r|\lambda_i|^{l+j-1}\\
&\indent\indent\indent\cdot\sqrt{w_i(N_{r'[G\backslash E]}(v'))\cdot QPQw_i(N_{r'[G\backslash E]}(v'))}\\
&\le (1+((1-c)\lambda_1)^{(r+r'+l+j)/2})\sqrt{k}((\min p_i)^{-1/2}+x)/\sqrt{n})\log n\\
&\indent\indent+\frac{c}{n}\sum_{i=1}^h (1-c)^{r+l+j-1}x\lambda_{h'}\left(\frac{\lambda_{h'}^2}{2\lambda_1}\right)^r|\lambda_i|^{l+j-1}((\min p_j)^{-1/2}+x)(1-c)^{r'}|\lambda_i|^{r'+1}\\
&\le (1+((1-c)\lambda_1)^{(r+r'+l+j)/2})\sqrt{k}((\min p_i)^{-1/2}+x)/\sqrt{n})\log n\\
&\indent\indent+\frac{ch}{n}(1-c)^{l+j-1}x\lambda_{h'}\left(\frac{(1-c)^2\lambda_{h'}^2}{2}\right)^{(r+r')/2}\lambda_1^{l+j}((\min p_j)^{-1/2}+x)
\end{align*}
with probability $1-o(1)$.

In other words, under these circumstances $|x_l(i)|$ is upper bounded by a constant multiple of $\frac{c}{n}((1-c)|\lambda_i|)^{r+r'}$ if $v$ and $v'$ are both good, and every entry of $u_l(h+1)$ has a magnitude that is upper bounded by a constant multiple of $((1-c)\lambda_1)^{(r+r')/2}\log n/\sqrt{n}+\frac{c}{n}((1-c)^2\lambda_{h'}^2/2)^{(r+r')/2}$. Either way, every entry of $u_l(i)$ is upper bounded by a constant multiple of $\frac{c}{n}((1-c)|\lambda_i|)^{r+r'}\log n$.

That means that for any $i\in (\mathbb{Z}\cap [1,h+1])^m$ such that $i_j\ge m+1$ for some $i$, then $\det([u_1(i_1),u_2(i_2),...,u_m(i_m)])$ is upper bounded by a constant multiple of $\frac{c^m}{n^m}\log^m(n)(1-c)^{m(r+r')}|\lambda_{m+2}|^{r+r'}\prod_{i=1}^{m-1} |\lambda_i|^{r+r'}$ Similarly, for $i\in (\mathbb{Z}\cap [1,h+1])^m$ such that $i_j\ge {m-1}$ and $i_{j'}\ge {m-1}$ with $j\ne j'$, $\det([u_1(i_1),u_2(i_2),...,u_m(i_m)])$ is upper bounded by a constant multiple of $\frac{c^m}{n^m}\log^m(n)(1-c)^{m(r+r')}|\lambda_m|^{r+r'}|\lambda_{m+1}|^{r+r'}\prod_{i=1}^{m-2} |\lambda_i|^{r+r'}$. There are at most $m^m$ such $i$; therefore,

\begin{align*}
|\det(M_{m,r,r'[E]}(v\cdot v'))&-\frac{c^m}{n^m}\prod_{i=1}^{m-1}  w_i(N_{r[G\backslash E]}(v))\cdot Qw_i(N_{r'[G\backslash E]}(v'))\\
&\cdot (\gamma w_m(N_{r[G\backslash E]}(v))\cdot Qw_m(N_{r'[G\backslash E]}(v'))\\
&\indent\indent+\gamma'w_{m+1}(N_{r[G\backslash E]}(v))\cdot Qw_{m+1}(N_{r'[G\backslash E]}(v')))|\\
&\le \frac{c^m}{n^m}\ln^{m+1}(n)(1-c)^{m(r+r')}|\lambda_{m+2}|^{r+r'}\prod_{i=1}^{m-1} |\lambda_i|^{r+r'}\\
&\indent\indent + \frac{c^m}{n^m}\ln^{m+1}(n)(1-c)^{m(r+r')}|\lambda_m|^{r+r'}|\lambda_{m+1}|^{r+r'}\prod_{i=1}^{m-2} |\lambda_i|^{r+r'}
\end{align*}
with probability $1-o(1)$, as desired.

Alternately, recall that if $v$ is $(r+2m+1,x)$-good then for any $r''<r+2m+1$ and $i\le h$, \[||E[N_{r''+1[G\backslash E]}(v))-(1-c)PQN_{r''[G\backslash E]}(v)||=O\left(\frac{|N_{r''[G\backslash E]}(v)|^2}{n}\right)=O\left(\frac{((1-c)\lambda_1)^{2r''}}{n}\right)\]

Also, for fixed values of $N_{r'''}(v)$ for all $r'''\le r''\le r+2m+1$, each element of $N_{r''+1[G\backslash E]}(v)$ has a variance of $O(|N_{r''[G\backslash E]}(v)|)=O((1-c)\lambda_1^{r''})$. So, 
\begin{align*}
&||w_i(N_{r''+l+j[G\backslash E]}(v))-(1-c)^{l+j}\lambda_i^{l+j}w_i(N_{r''}[G\backslash E](v))|| \\&\le \left(\frac{((1-c)\lambda_1)^{2r''}}{n}+((1-c)\lambda_1)^{r''/2}\right)\ln n
\end{align*}
with probability $1-o(1)$ for all $l,j\le m$ and $i\le h$. This implies that if $v$ is $(r+2m+1,x)$-good and $v'$ is $(r'+1,x)$-good then
\begin{align*}
|u_l(h+1)_j|&=|N_{r+l+j,r'[E]}(v\cdot v')-\sum_{i=1}^h x_{l+j}(i)|\\
&\le |N_{r+l+j,r'[E]}(v\cdot v')-\frac{c}{n}N_{r+l+j[G\backslash E]}(v)\cdot QN_{r'[G\backslash E]}(v')|\\
&\indent\indent+|\frac{c}{n}N_{r+l+j[G\backslash E]}(v)\cdot QN_{r'[G\backslash E]}(v')-\sum_{i=1}^h x_{l+j}(i)|\\
&\le (1+\sqrt{|N_{r+l+j[G\backslash E]}(v)|\cdot |N_{r'[G\backslash E]}(v')|/n})\log n\\
&\indent\indent +\frac{c}{n}\sum_{i=1}^h |w_i(N_{r+l+j[G\backslash E]}(v))\cdot Qw_i(N_{r'[G\backslash E]}(v'))\\
&\indent\indent\indent\indent-(1-c)^{l+j}\lambda_i^{l+j}w_i(N_{r[G\backslash E]}(v))\cdot Qw_i(N_{r'[G\backslash E]}(v'))|\\
&\le (1+((1-c)\lambda_1)^{(r+r'+l+j)/2})\sqrt{k}(((\min p_i)^{-1/2}+x)/\sqrt{n})\log n\\
&\indent\indent+\frac{c}{n}\sum_{i=1}^h \left(\frac{((1-c)\lambda_1)^{2r}}{n}+((1-c)\lambda_1)^{r/2}\right)\log n\cdot ||Qw_i(N_{r'[G\backslash E]}(v'))||\\
&\le (1+((1-c)\lambda_1)^{(r+r'+l+j)/2})\sqrt{k}(((\min p_i)^{-1/2}+x)/\sqrt{n})\log n\\
&\indent\indent+\frac{c}{n}\sum_{i=1}^h \left(\frac{((1-c)\lambda_1)^{2r}}{n}+((1-c)\lambda_1)^{r/2}\right)\log n((\min p_j)^{-1/2}+x)(1-c)^{r'}|\lambda_i|^{r'+1}\\
&\le (1+((1-c)\lambda_1)^{(r+r'+l+j)/2})\sqrt{k}(((\min p_i)^{-1/2}+x)/\sqrt{n})\log n\\
&\indent\indent+\frac{ch}{n}\left(\frac{((1-c)\lambda_1)^{2r}}{n}+((1-c)\lambda_1)^{r/2}\right)\log n((\min p_j)^{-1/2}+x)(1-c)^{r'}\lambda_1^{r'+1}\\
\end{align*}
for any $l$ and $j$ with probability $1-o(1)$. If $m=h'+1$ then for any $i\in (\mathbb{Z}\cap [1,h+1])^m$, either there exist $j\ne j'$ such that $i_{j}=i_{j'}\le h'$, or there exists $j$ such that $i_j>h'$. Either way, $\det([u_1(i_1),u_2(i_2),...,u_m(i_m)])$ is upper bounded by a constant multiple of \[\frac{c^{h'+1}}{n^{h'+1}}\log(n)(1-c)^{h'(r+r')}\prod_{i=1}^{h'} |\lambda_i|^{r+r'}\left(\frac{((1-c)\lambda_1)^{2r}}{n}+((1-c)\lambda_1)^{r/2}\right)\cdot (1-c)^{r'}\lambda_1^{r'}\] 
with probability $1-o(1)$. There are only $m^m$ possible choices of $i$, so with probability $1-o(1)$, either $v$ is $(r+2m+1,x)$-bad, $v'$ is $(r'+1,x)$-bad, or 
\begin{align*}
&|\det(M_{m,r,r'[E]}(v\cdot v'))|\\
&\le \frac{c^{h'+1}}{n^{h'+1}}\log^2(n)(1-c)^{h'(r+r')}\prod_{i=1}^{h'} |\lambda_i|^{r+r'}\left(\frac{((1-c)\lambda_1)^{2r}}{n}+((1-c)\lambda_1)^{r/2}\right)\cdot (1-c)^{r'}\lambda_1^{r'}
\end{align*}
\end{proof}

While this is a helpful result, it turns out to be more useful to have an expression that links the determinant to $ w_i(N_{r[G\backslash E]}(v))\cdot Qw_i(N_{r'[G\backslash E]}(v'))$ for fixed values of $r$ and $r'$. So, we have the following:
\begin{lemma}
Let $0<c<1$, $x>0$, $G$ be drawn from $\gss(n,p,Q)$, $E$ be a subset of $G$'s edges that independently contains each edge with probability $c$, and $m\le h'$. Now, for any $v,v'\in G$ and $\sqrt{\ln n}\le r'\le r\in \mathbb{Z}^+$, such that $((1-c)\lambda_{h'}^2/2)^{r+r'}>\lambda_1^{r+r'}n$, with probability $1-o(1)$, either $v$ is $(r+2m+1,x)$-bad, $v'$ is $(r'+1,x)$-bad, or 
\begin{align*}
|\det(M_{m,r,r'[E]}&(v\cdot v'))-\frac{c^m}{n^m}\prod_{i=1}^{m-1} ((1-c)\lambda_i)^{r+r'-2\sqrt{\ln n}} w_i(N_{\sqrt{\ln n}[G\backslash E]}(v))\cdot Qw_i(N_{\sqrt{\ln n}[G\backslash E]}(v'))\\
&\cdot (\gamma((1-c)\lambda_m)^{r+r'-2\sqrt{\ln n}} w_m(N_{\sqrt{\ln n}[G\backslash E]}(v))\cdot Qw_m(N_{\sqrt{\ln n}[G\backslash E]}(v'))\\
&+\gamma'((1-c)\lambda_{m+1})^{r+r'-2\sqrt{\ln n}}w_{m+1}(N_{\sqrt{\ln n}[G\backslash E]}(v))\cdot Qw_{m+1}(N_{\sqrt{\ln n}[G\backslash E]}(v')))|\\
&\le \frac{1}{\ln^{2m+2} n}\cdot \frac{c^m}{n^m}\prod_{i=1}^{m} |(1-c)\lambda_i|^{r+r'}
\end{align*}
\end{lemma}

\begin{proof}
First, note that
\begin{align*}
&|w_i(N_{r[G/E]}(v))\cdot Qw_i(N_{r'[G/E]}(v'))\\
&\indent\indent-((1-c)\lambda_i)^{r+r'-2\sqrt{\ln n}}w_i(N_{\sqrt{\ln n}[G/E]}(v))\cdot Qw_i(N_{\sqrt{\ln n}[G/E]}(v'))|\\
&\le |\lambda_i[w_i(N_{r[G/E]}(v))\cdot P^{-1}(w_i(N_{r'[G/E]}(v'))-((1-c)\lambda_i)^{r'-\sqrt{\ln n}}w_i(N_{\sqrt{\ln n}[G/E]}(v')))\\
&\phantom{xxx}+(w_i(N_{r[G/E]}(v))-((1-c)\lambda_i)^{r-\sqrt{\ln n}}w_i(N_{\sqrt{\ln n}[G/E]}(v)))\\
&\indent\indent\indent\cdot P^{-1}((1-c)\lambda_i)^{r'-\sqrt{\ln n}}w_i(N_{\sqrt{\ln n}[G/E]}(v'))]|\\
&\le |\lambda_i|\sqrt{w_i(N_{r[G/E]}(v))\cdot P^{-1}w_i(N_{r[G/E]}(v))}\cdot \frac{x|(1-c)\lambda_i|^{r'}}{2^{\sqrt{\ln n}}}\\
&\phantom{xxx}+|\lambda_i|\frac{x|(1-c)\lambda_i|^{r}}{2^{\sqrt{\ln n}}}\cdot |(1-c)\lambda_i|^{r'-\sqrt{\ln n}}\sqrt{w_i(N_{\sqrt{\ln n}[G/E]}(v'))\cdot P^{-1}w_i(N_{\sqrt{\ln n}[G/E]}(v'))}\\
&\le\frac{2x(1-c)^{r+r'}|\lambda_i|^{r+r'+1}}{2^{\sqrt{\ln n}}}((\min p_j)^{-1/2}+x)=o(|(1-c)\lambda_i|^{r+r'}/\ln^{2m+2}(n))
\end{align*}

Also, \[|\lambda_{m+2}|^{r+r'}/|\lambda_m|^{r+r'}\le n^{-\ln(|\lambda_m/\lambda_{m+2}|)/\ln((1-c)\lambda_{h'}^2/(2\lambda_1))}=o(1/\ln^{3m+3} n)\] and \[|\lambda_{m+1}|^{r+r'}/|\lambda_{m-1}|^{r+r'}\le n^{-\ln(|\lambda_{m-1}/\lambda_{m+1}|)/\ln((1-c)\lambda_{h'}^2/(2\lambda_1))}=o(1/\ln^{3m+3} n)\] Combining these inequalities with the determinant lemma yields the desired result.
\end{proof}

In some sense, this establishes that 
\begin{align*}
&\det(M_{m,r,r'[E]}(v\cdot v'))\approx \frac{c^m}{n^m}\prod_{i=1}^{m-1}  ((1-c)\lambda_i)^{r+r'-2\sqrt{\ln n}} w_i(N_{\sqrt{\ln n}[G\backslash E]}(v))\cdot Qw_i(N_{\sqrt{\ln n}[G\backslash E]}(v'))\\
&\indent\indent\cdot (\gamma((1-c)\lambda_m)^{r+r'-2\sqrt{\ln n}}w_m(N_{\sqrt{\ln n}[G\backslash E]}(v))\cdot Qw_m(N_{\sqrt{\ln n}[G\backslash E]}(v'))\\
&\indent\indent\indent\indent+\gamma'((1-c)\lambda_{m+1})^{r+r'-2\sqrt{\ln n}}w_{m+1}(N_{\sqrt{\ln n}[G\backslash E]}(v))\cdot Qw_{m+1}(N_{\sqrt{\ln n}[G\backslash E]}(v')))
\end{align*}
 However, in order to know this in a useful sense it is necessary to prove that these terms are large relative to the error terms. In order to do that, we need the following:

\begin{lemma}
For any $p\in (0,1)^k$ with $\sum p_i=1$ and $k\times k$ matrix $Q$ with nonnegative entries such that every entry of $Q^k$ is positive, there exists a unique $w\in (0,\infty)^k$ such that $w$ is an eigenvector of $PQ$ with eigenvalue $\lambda_1$ and $w\cdot P^{-1}w=1$. Now, let $G$ be drawn from $\gss(n,p,Q)$, $v\in G$, and $w'$ be an eigenvector of $PQ$ with eigenvalue $\lambda_i$ such that $\lambda_i^2>\lambda_1$. With probability $1-o(1)$ either $v$ is $(\sqrt{\ln n},\min (P^{-1}w)_i/2)$-bad, or $|w' \cdot P^{-1}N_{\sqrt{\ln n}}(v)|\ge \lambda_i^{\sqrt{\ln n}}/\ln n$.
\end{lemma}

\begin{proof}
Every entry in $(PQ)^k$ is positive, so its eigenvector of largest eigenvalue is unique up to multiplication, and its entries all have the same sign. That means that it has a unique multiple, $w$, such that $w\cdot P^{-1}w=1$ and $w_1>0$. In this case all of $w$'s entries must be positive, and since $(PQ)^kw=\lambda_1^k w$, it must be the case that $PQw=\lambda_1 w$.

Given any $(\sqrt{\ln n},\min (P^{-1}w)_j/2)$-good vertex $v$, and any $r<\sqrt{\ln n}$, 
\[w\cdot P^{-1}N_r(v)\ge \lambda_1^r w\cdot P^{-1}\{v\}-\frac{\min (P^{-1}w)_j}{2}\sum_{j=0}^{r-1} 2^{-j-1}\lambda_1^r\ge \lambda_1^r \min (P^{-1}w)_j/2\]

For any eigenvector $w''$ with eigenvalue $\lambda_{i'}$ and $w''\cdot P^{-1}w''=1$,
\[|w''\cdot P^{-1}N_r(v)|\le |\lambda_{i'}|^r w''\cdot P^{-1}\{v\}+\frac{\min (P^{-1}w)_j}{2}\sum_{j=0}^{r-1} 2^{-j-1}|\lambda_{i'}|^r\le |\lambda_{i'}|^r(\min p_j^{-1/2}+\min (P^{-1}w)_j/2)\]

$\lambda_1>\lambda_2$, so there exists a constant $r_0$ such that for any $r>r_0$, 
\[N_r(v)_j\ge \frac{\min(P^{-1} w)_j}{4} \lambda_1^rw_j\]
 for all $j$. For $r_0<r\le\sqrt{\ln n}$, $1\le j\le k$, and a fixed value of $N_r(v)$, the probability distribution of $N_{r+1}(v)_j$ is within $o(1)$ of a poisson distribution with expected value $(PQN_r(v))_j$. Furthermore, $N_{r+1}(v)_{j'}$ has negligible dependence on $N_{r+1}(v)_j$ for all $j'\ne j$. So, $w'\cdot P^{-1}N_{r+1}(v)$ has an expected value of $\lambda_iw'\cdot P^{-1}N_r(v)+o(1)$ and a standard deviation of \[\sqrt{\sum_{j=1}^k w_j^{\prime 2}p_j^{-2} (PQN_r(v))_j}+o(1)\le \sqrt{2\lambda_1^{r+1}\min (P^{-1}w)_j \sum_{j=1}^k w_j^{\prime 2} p^{-2}_jw_j }+o(1)\]

This implies that 
\begin{align*}
&E\left[|w'\cdot P^{-1}N_{\sqrt{\ln n}}(v)-\lambda_i^{\sqrt{\ln n}-r}w'\cdot P^{-1}N_r(v)|\right]\\
&\le \sum_{r'=r}^{\sqrt{\ln n}-1} E\left[|\lambda_i^{\sqrt{\ln n}-1-r'} w'\cdot P^{-1}N_{r'+1}(v)-\lambda_i^{\sqrt{\ln n}-r'} w'\cdot P^{-1}N_{r'}(v)|\right]\\
&\le \sum_{r'=r}^{\sqrt{\ln n}-1} \lambda_i^{\sqrt{\ln n}-1-r'}\left(\sqrt{2\lambda_1^{r'+1}\min (P^{-1}w)_j \sum_{j=1}^k w_j^{\prime 2} p^{-2}_jw_j }+o(1)\right)\\
&\le \lambda_i^{\sqrt{\ln n}-r}\lambda_1^{r/2} \cdot\frac{1}{\sqrt{\lambda_i^2/\lambda_1}-1}\left(\sqrt{2\min (P^{-1}w)_j \sum_{j=1}^k w_j^{\prime 2}p_j^{-2} w_j }+o(1)\right)\\
\end{align*}
In particular, this means that if there exists $r_0< r<2\ln\ln n/\ln(\lambda_i^2/\lambda_1)$ such that $|w'\cdot P^{-1}N_r(v)|>\lambda_1^{r/2}\ln\ln\ln\ln n$ then with probability $1-o(1)$, 
\begin{align*}
|w'\cdot P^{-1}N_{\sqrt{\ln n}}(v)|&\ge \lambda_i^{\sqrt{\ln n}-r}\lambda_1^{r/2}\ln\ln\ln\ln n -\lambda_i^{\sqrt{\ln n}-r}\lambda_1^{r/2}\sqrt{\ln\ln\ln\ln n}\\
&\ge \lambda_i^{\sqrt{\ln n}}\cdot (\lambda_i^2/\lambda_1)^{-r/2}\\
&\ge \lambda_i^{\sqrt{\ln n}}/\ln n
\end{align*}

Furthermore, since the probability distribution of $w'\cdot P^{-1}N_r(v)$ for a fixed value of $N_{r-1}(v)$ is a sum of constant multiples of poisson distributions with expected values of $\Omega(\lambda_1^r)$, it must be the case that $|w'\cdot P^{-1}N_r(v)|>\lambda_1^{r/2}\ln\ln\ln\ln n$ with probability $e^{-O(\ln^2\ln\ln\ln n)}=\omega(1/\ln\ln n)$. Therefore, there exists $r_0< r<2\ln\ln n/\ln(\lambda_i^2/\lambda_1)$ such that this holds with probability $1-o(1)$, and the lemma follows.
\end{proof}

This also implies that for any $v$, $v'$, and $i$, either $v$ is $(\sqrt{\ln n},\min (P^{-1}w)_i/2)$-bad, $v'$ is $(\sqrt{\ln n},\min (P^{-1}w)_i/2)$-bad, or $|w_i(N_{\sqrt{\ln n}}(v))\cdot P^{-1} w_i(N_{\sqrt{\ln n}}(v'))|\ge \lambda_i^{2\sqrt{\ln n}}/\ln^2 n$ with probability $1-o(1)$, since the degree of dependence between $N_{\sqrt{\ln n}}(v)$ and $N_{\sqrt{\ln n}}(v')$ is negligible. This allows me to attempt to approximate $PQ$'s eigenvalues as follows.
\begin{algorithm}
Basic-Eigenvalue-approximation-algorithm(E,c,v):

\phantom{xxx} Compute $N_{r[G\backslash E]}(v)$ for each $r$ until $|N_{r[G\backslash E]}(v)|>\sqrt{n}$, and then set $\lambda_1''=\sqrt[2r]{n}/(1-c)$.

\phantom{xxx} Set $r=r'=\frac{2}{3}\log n/\log ((1-c)\lambda''_1)-\sqrt{\ln n}$. Then, compute 
\[\sqrt[2r]{\frac{n \max(|\det M_{m,r,r[E]}(v\cdot v)|,|\det M_{m,r+1,r[E]}(v\cdot v)|)}{c\max(|\det M_{m-1,r,r[E]}(v\cdot v)|,|\det M_{m-1,r+1,r[E]}(v\cdot v)|)}}\]
 until an m is found for which this expression is less than $((1-c)\lambda''_1)^{3/4}+\frac{1}{\sqrt{\ln n}}$. Then, set $h''=m-1$.

\phantom{xxx}

\phantom{xxx} Then, set \[|\lambda'_i|=\frac{1}{1-c}\sqrt{\det(M_{i,r+3,r'[E]}(v\cdot v'))/\det(M_{i,r+1,r'[E]}(v\cdot v'))}/\prod_{j=1}^{i-1}(1-c)|\lambda'_j|\] unless $|\det(M_{i,r+1,r'[E]}(v\cdot v'))|<\sqrt{|\det(M_{i,r,r'[E]}(v\cdot v'))|\cdot|\det(M_{i,r+2,r'[E]}(v\cdot v'))|}$, in which case set $|\lambda'_i|=\frac{1}{1-c}\sqrt{\det(M_{i,r+2,r'[E]}(v\cdot v'))/\det(M_{i,r,r'[E]}(v\cdot v'))}/\prod_{j=1}^{i-1}(1-c)|\lambda'_j|$. Repeat this for each $i\le h''$

\phantom{xxx}

\phantom{xxx} Next, for each $i<h''$, if $||\lambda'_i|-|\lambda'_{i+1}||<\frac{1}{\ln n}$ then set $\lambda'_i=|\lambda'_i|$ and $\lambda'_{i+1}=-|\lambda'_{i+1}|$. For each $i\le h''$ such that $||\lambda'_i|-|\lambda'_{i+1}||\ge\frac{1}{\ln n}$ and $||\lambda'_{i-1}|-|\lambda'_i||\ge\frac{1}{\ln n}$ set\[\lambda'_i=\frac{1}{1-c}\det(M_{i,r+1,r'[E]}(v\cdot v'))/\det(M_{i,r,r'[E]}(v\cdot v'))/\prod_{j=1}^{i-1}(1-c)\lambda'_j\]

\phantom{xxx} Return $(\lambda'_1,...,\lambda'_{h''})$

\end{algorithm}

\begin{lemma}
Let $0<c<1$, $\min (P^{-1}w)_i/2>x>0$, $G$ be drawn from $\gss(n,p,Q)$, $E$ be a subset of $G$'s edges that independently contains each edge with probability $c$, $v\in G$, such that $(1-c)(\lambda_{h'}^2/2)^{4}>\lambda_1^{7}$.

With probability $1-o(1)$, either $v$ is $(\frac{2}{3}\cdot \log n/\log ((1-c)\lambda_1),x)$-bad or {\it Basic-Eigenvalue-approximation-algorithm}(E,c,v) runs in $O(n)$ time and returns $(\lambda'_1,...,\lambda'_{h''})$ such that $h'=h''$ and $|\lambda'_i-\lambda_i|<\ln^{-3/2}(n)$ for all $i$.
\end{lemma}

\begin{proof}
If $v$ is $(\frac{2}{3}\cdot \log n/\log ((1-c)\lambda_1,x)$-good, there exists a constant $r$ such that for any  $r<r''<\frac{2}{3}\cdot \log n/\log ((1-c)\lambda_1)$, 
\[\frac{\min (P^{-1}w)_j}{4}((1-c)\lambda_1)^{r''}\sum_{j=1}^k w_j\le |N_{r''[G\backslash E]}(v[0])|\le ((1-c)\lambda_1)^{r''} \sqrt{k}((\min p_i)^{-1/2}+x)\]
 So, the minimum $r''$ such that $|N_{r''[G\backslash E]}(v)|>\sqrt{n}$ is within a constant of $\log n/2\log ((1-c)\lambda_1)$. That means that $|\lambda_1-\lambda''_1|$ is upper bounded by a constant multiple of $1/\log n$, and that $r=r'$ is within a constant of the value it would have if $\lambda''_1$ were $\lambda_1$. This also implies that if $n$ is sufficiently large it is less than $\frac{2}{3}\cdot \log n/\log ((1-c)\lambda_1)-2h'-3$. 

For $m\le h'$, $r''\in \{r,r+1\}$, and $v\in G$ if $|\lambda_i|\ne |\lambda_{i+1}|$ then 
\[|\frac{\det(M_{m,r'',r''[E]}(v\cdot v))}{\frac{\gamma(\{(1-c)\lambda_j\},m) c^m}{n^m}\prod_{i=1}^{m} ((1-c)\lambda_i)^{2r''-2\sqrt{\ln n}} w_i(N_{\sqrt{\ln n}[G\backslash E]}(v))\cdot Qw_i(N_{\sqrt{\ln n}[G\backslash E]}(v))}-1|=o(1)\]
with probability $1-o(1)$. If $|\lambda_i|=|\lambda_{i+1}|$ then $\lambda_i=-\lambda_{i+1}$,
so either \\$\gamma((1-c)\lambda_m)^{2r-2\sqrt{\ln n}} w_m(N_{\sqrt{\ln n}[G\backslash E]}(v))\cdot Qw_m(N_{\sqrt{\ln n}[G\backslash E]}(v))$ has the same sign as $\gamma'((1-c)\lambda_{m+1})^{2r-2\sqrt{\ln n}}w_{m+1}(N_{\sqrt{\ln n}[G\backslash E]}(v))\cdot Qw_{m+1}(N_{\sqrt{\ln n}[G\backslash E]}(v))$
 or \\$ \gamma((1-c)\lambda_m)^{2r+1-2\sqrt{\ln n}} w_m(N_{\sqrt{\ln n}[G\backslash E]}(v))\cdot Qw_m(N_{\sqrt{\ln n}[G\backslash E]}(v))$ has the same sign as $\gamma'((1-c)\lambda_{m+1})^{2r+1-2\sqrt{\ln n}}w_{m+1}(N_{\sqrt{\ln n}[G\backslash E]}(v))\cdot Qw_{m+1}(N_{\sqrt{\ln n}[G\backslash E]}(v)))$. Either way, we have that 

\begin{align*}
&(1-o(1))|\frac{\gamma(\{(1-c)\lambda_j\},m) c^m}{n^m}\prod_{i=1}^{m} ((1-c)\lambda_i)^{2r-2\sqrt{\ln n}} w_i(N_{\sqrt{\ln n}[G\backslash E]}(v))\cdot Qw_i(N_{\sqrt{\ln n}[G\backslash E]}(v))|\\
&\le \max(|\det(M_{m,r,r[E]}(v\cdot v))|,|\det(M_{m,r+1,r[E]}(v\cdot v))|)\\
&\le (1+o(1)) |\frac{c^m}{n^m}\prod_{i=1}^{m-1} ((1-c)\lambda_i)^{2r+1-2\sqrt{\ln n}} w_i(N_{\sqrt{\ln n}[G\backslash E]}(v))\cdot Qw_i(N_{\sqrt{\ln n}[G\backslash E]}(v))|\\
&\cdot (|\gamma((1-c)\lambda_m)^{2r+1-2\sqrt{\ln n}} w_m(N_{\sqrt{\ln n}[G\backslash E]}(v))\cdot Qw_m(N_{\sqrt{\ln n}[G\backslash E]}(v))|\\
&+|\gamma'((1-c)\lambda_{m+1})^{2r+1-2\sqrt{\ln n}}w_{m+1}(N_{\sqrt{\ln n}[G\backslash E]}(v))\cdot Qw_{m+1}(N_{\sqrt{\ln n}[G\backslash E]}(v))|)
\end{align*}
with probability $1-o(1)$. Furthermore, by lemma $8$, these bounds are within a factor of $O(\ln^2 n)$ of each other with probability $1-o(1)$. 

That means that 
\[\frac{n \max(|\det M_{m,r,r[E]}(v\cdot v)|,|\det M_{m,r+1,r[E]}(v\cdot v)|)}{c\max(|\det M_{m-1,r,r[E]}(v\cdot v)|,|\det M_{m-1,r+1,r[E]}(v\cdot v)|)}\] is within a factor of $\ln^3(n)$ of $|((1-c)\lambda_m)^{2r_1-2\sqrt{\ln n}}w_m(N_{\sqrt{\ln n}[G\backslash E]}(v))\cdot Qw_m(N_{\sqrt{\ln n}[G\backslash E]}(v))|$ with probability $1-o(1)$, and thus that 
\[\sqrt[2r_1]{\frac{n \max(|\det M_{m,r,r[E]}(v\cdot v)|,|\det M_{m,r+1,r[E]}(v\cdot v)|)}{c\max(|\det M_{m-1,r,r[E]}(v\cdot v)|,|\det M_{m-1,r+1,r[E]}(v\cdot v)|)}}=|(1-c)\lambda_m|\pm o(1)\]
with probability $1-o(1)$. $|(1-c)\lambda_m|\ge \sqrt[4]{(4(1-c)^3\lambda_1^3}$, so with probability $1-o(1)$, this expression is not less than $((1-c)\lambda''_1)^{3/4}+\frac{1}{\sqrt{\ln n}}$.

If $m=h'+1$, $r''\in \{r,r+1\}$, and $v\in G$, then with probability $1-o(1)$ either $v$ is $(r''+2m+1,x)$-bad or 
\begin{align*}
&|\det(M_{m,r,r'[E]}(v\cdot v))|\\
&\le \frac{c^{h'+1}}{n^{h'+1}}\ln^2(n)(1-c)^{2h'\cdot r''}\prod_{i=1}^{h'}|\lambda_i|^{2r''} \left(\frac{((1-c)\lambda_1)^{2r''}}{n}+((1-c)\lambda_1)^{r''/2}\right)\cdot (1-c)^{r''}\lambda_1^{r''}\\
&\le 2\frac{c^{h'+1}}{n^{h'+1}}\ln^2(n)(1-c)^{2h'\cdot r''}\prod_{i=1}^{h'}|\lambda_i|^{2r''} \cdot ((1-c)\lambda_1)^{3r''/2}\\
\end{align*}
If v falls under the later case,
\[\frac{n \max(|\det M_{m,r,r[E]}(v\cdot v)|,|\det M_{m,r+1,r[E]}(v\cdot v)|)}{c\max(|\det M_{m-1,r,r[E]}(v\cdot v)|,|\det M_{m-1,r+1,r[E]}(v\cdot v)|)}=O(\ln^2(n)((1-c)\lambda_1)^{3r/2})\]
, so \[\sqrt[2r]{\frac{n \max(|\det M_{m,r,r[E]}(v\cdot v)|,|\det M_{m,r+1,r[E]}(v\cdot v)|)}{c\max(|\det M_{m-1,r,r[E]}(v\cdot v)|,|\det M_{m-1,r+1,r[E]}(v\cdot v)|)}}=((1-c)\lambda_1)^{3/4}+O(\ln(\ln(n))/\ln(n))\] 
Therefore, with probability $1-o(1)$, either $v$ is $(\frac{2}{3}\cdot \log n/\log ((1-c)\lambda_1),x)$-bad or $h''=h'$.

By the previous two lemmas, with probability $1-o(1)$ either $v$ is $(r+2m+4,x)$-bad, or \[w_i(N_{\sqrt{\ln n}}(v))\cdot P^{-1} w_i(N_{\sqrt{\ln n}}(v'))|\ge \lambda_i^{2\sqrt{\ln n}}/\ln^2 n\] for all $i\le h$ and 
\begin{align*}
|\det(M_{m,r'',r'[E]}&(v\cdot v'))\\
&-\frac{c^m}{n^m}\prod_{i=1}^{m-1} ((1-c)\lambda_i)^{r''+r'-2\sqrt{\ln n}} w_i(N_{\sqrt{\ln n}[G\backslash E]}(v))\cdot Qw_i(N_{\sqrt{\ln n}[G\backslash E]}(v'))\\
&\cdot (\gamma((1-c)\lambda_m)^{r''+r'-2\sqrt{\ln n}} w_m(N_{\sqrt{\ln n}[G\backslash E]}(v))\cdot Qw_m(N_{\sqrt{\ln n}[G\backslash E]}(v'))\\
&+\gamma'((1-c)\lambda_{m+1})^{r''+r'-2\sqrt{\ln n}}w_{m+1}(N_{\sqrt{\ln n}[G\backslash E]}(v))\cdot Qw_{m+1}(N_{\sqrt{\ln n}[G\backslash E]}(v')))|\\
&\le \frac{1}{\ln^{2m+2} n}\cdot \frac{c^m}{n^m}\prod_{i=1}^{m} |(1-c)\lambda_i|^{r''+r'}
\end{align*} for all $m\le h'+1$ and $r\le r''<r+4$.

Assume that the later holds. If $|\lambda_{m+1}|<|\lambda_m|$ then if $n$ is sufficiently large the $\gamma'$ term will be negligable relative to the $\gamma$ term, while if $|\lambda_{m+1}|=|\lambda_m|$ increasing $r''$ by two would multiply both of these terms by $(1-c)^2\lambda_m^2$. Either way, we have that for any $k\le h'$ and $r\le r''\le r+1$, \[|\det(M_{m,r''+2,r'[E]}(v\cdot v'))/\det(M_{m,r'',r'[E]}(v\cdot v'))-\prod_{j=1}^{k} ((1-c)\lambda_j)^2|\le \frac{1}{\ln^{7/8}n}\] as long as $n$ is sufficiently large, and 
\begin{align*}
&| \gamma((1-c)\lambda_m)^{r''+r'-2\sqrt{\ln n}} w_m(N_{\sqrt{\ln n}[G\backslash E]}(v))\cdot Qw_m(N_{\sqrt{\ln n}
[G\backslash E]}(v'))\\
&\phantom{xxx}+\gamma'((1-c)\lambda_{m+1})^{r''+r'-2\sqrt{\ln n}}w_{m+1}(N_{\sqrt{\ln n}[G\backslash E]}(v))\cdot Qw_{m+1}(N_{\sqrt{\ln n}[G\backslash E]}(v'))|\\
&\ge \frac{1}{2}| \gamma((1-c)\lambda_m)^{r''+r'-2\sqrt{\ln n}} w_m(N_{\sqrt{\ln n}[G\backslash E]}(v))\cdot Qw_m(N_{\sqrt{\ln n}
[G\backslash E]}(v'))|
\end{align*}
We assumed that $n$ is large, so that can only fail if $|\lambda_m|=|\lambda_{m+1}|$. In this case, $\lambda_m=-\lambda_{m+1}$, so increasing $m$ by one would result in both terms having the same sign, at which point the condition is satisfied. So, 
\[||\lambda_i'|^2-|\lambda_i|^2|\le 4\ln^{-7/8} n\]
for any $i\le h'$. For any $i<h'$, if $|\lambda_i|>|\lambda_{i+1}|$, then for sufficiently large $n$, $|\lambda'_i|-|\lambda'_{i+1}|$ will be greater than $\frac{1}{\ln n}$. On the other hand, if $|\lambda_i|=|\lambda_{i+1}|$ and $n$ is sufficiently large, $|\lambda'_i|-|\lambda'_{i+1}|$ will be less than $\frac{1}{\ln n}$. So, the algorithm will suceed at determining which eigenvalues have the same absolute values and assign $\lambda'_i$ the same sign as $\lambda_i$ for each $i$. So, $|\lambda'_i-\lambda_i|<\ln^{-3/2}(n)$ as desired. Assuming this all works, the slowest part of the algorithm is computing expressions of the form $N_{r'',r'''[E]}(v\cdot v)$, each such expression can be computed in $O(n)$ time, and only a constant number of them need to be computed. So, the algorithm runs in $O(n)$ time.
\end{proof}

So, this algorithm sometimes works, but it fails if $v$ is bad. This risk can be mitigated by using multiple vertices as follows.

\begin{algorithm}
Improved-Eigenvalue-approximation-algorithm(c):

\phantom{xxx} Create a set of edges $E$, that each of $G$'s edges is independently assigned to with probability $c$.

\phantom{xxx} Randomly select $\sqrt{\ln n}$ of $G$'s vertices, $v[1]$, $v[2]$,..., $v[\sqrt{\ln n}]$.

\phantom{xxx} Run {\it Basic-Eigenvalue-approximation-algorithm}(E,c,v[i]) for each $i\le \sqrt{\ln n}$, stopping the algorithm prematurely if it takes more than $O(n\sqrt{\ln n})$ time.

\phantom{xxx} Return $(\lambda'_1,...,\lambda'_{h''})$ where $h''$ and $\lambda'_i$ are the median outputs of the executions of {\it Basic-Eigenvalue-approximation-algorithm} for each $i$.

\end{algorithm}

\begin{lemma}
Let $0<c<1$, $\min (P^{-1}w)_i/2>x>0$, and $G$ be drawn from $\gss(n,p,Q)$. {\it Improved-Eigenvalue-approximation-algorithm}(c) runs in $O(n\log n)$ time. Furthermore, if $(1-c)(\lambda_{h'}^2/2)^{4}>\lambda_1^{7}$, $x<\frac{\lambda_1k}{\lambda_{h'}\min p_i}$, and 
\[2ke^{-\frac{x^2(1-c)\lambda_{h'}^2\min p_i}{16\lambda_1 k^{3/2}((\min p_i)^{-1/2}+x)}}/\left(1-e^{-\frac{x^2(1-c)\lambda_{h'}^2\min p_i}{16\lambda_1k^{3/2}((\min p_i)^{-1/2}+x)}\cdot((\frac{(1-c)\lambda_{h'}^4}{4\lambda_1^3})-1)}\right)<\frac{1}{2}\]
then {\it Improved-Eigenvalue-approximation-algorithm}(c) returns $(\lambda'_1,...,\lambda'_{h''})$ such that $h'=h''$ and $|\lambda'_i-\lambda_i|<\ln^{-3/2}(n)$ for all $i$ with probability $1-o(1)$.
\end{lemma}

\begin{proof}
Generating $E$ takes $O(n)$ time, picking $v[i]$ takes $o(n)$ time, each execution of {\it Basic-Eigenvalue-approximation-algorithm}(c) runs in $O(n\sqrt{\log n})$ time, and combining their outputs takes $o(n)$ time. So, this algorithm runs in $O(n\log n)$ time.

Assuming the conditions are satisfied, there exists $y<\frac{1}{2}$ such that $1-y$ of $G$'s vertices are $(\frac{2}{3}\cdot \log n/\log ((1-c)\lambda_1,x)$-good with probability $1-o(1)$. So, with probability $1-o(1)$, the majority of the selected vertices of $G$ are $(\frac{2}{3}\cdot \log n/\log ((1-c)\lambda_1),x)$-good and the majority of the executions of {\it Basic-Eigenvalue-approximation-algorithm} give good output. If this happens, then the median value of $h''$ is $h'$, and for each $1\le i\le h'$, the median value of $\lambda'_i$ is within $\ln^{-3/2}(n)$ of $\lambda_i$ for each $i$, as desired.
\end{proof}

Given approximations of $PQ$'s eigenvalues, one can attempt to approximate $N_i(\{v\})\cdot P^{-1} N_i(\{v'\})$ as follows.

\begin{algorithm}
Vertex-product-approximation-algorithm(v,v',r,r',E,c,$(\lambda'_1,...,\lambda'_{h''})$):

(Assumes that $N_{r''[G\backslash E]}(v)$ has already been computed for $r''\le r+2h''+3$ and that $N_{r''[G\backslash E]}(v')$ has already been computed for $r''\le r'$) 

\phantom{xxx}

\phantom{xxx} For each $i\le h''$, set 
\begin{align*}
z_i(v\cdot v')&=\frac{\det(M_{i,r+1,r'[E]}(v\cdot v')-(1-c)^i\lambda'_{i+1}\prod_{j=1}^{i-1} \lambda'_j\det(M_{i,r,r'[E]}(v\cdot v')}{\det(M_{i-1,r+1,r'[E]}(v\cdot v')-(1-c)^{i-1}\lambda'_{i}\prod_{j=1}^{i-2} \lambda'_j \det(M_{i-1,r,r'[E]}(v\cdot v')}\\
&\cdot\frac{ n(\lambda'_{i-1}-\lambda'_i)\gamma(\{(1-c)\lambda'_j\},i-1)}{c\lambda'_{i-1}(\lambda'_i-\lambda'_{i+1})\gamma(\{(1-c)\lambda'_j\},i)}((1-c)\lambda'_i)^{-r-r'-1}
\end{align*}

\phantom{xxx}

\phantom{xxx} Return $(z_1(v\cdot v'),...,z_{h''}(v\cdot v'))$.
\end{algorithm}

\begin{lemma}
Let $0<c<1$, $\min (P^{-1}w)_i/2>x>0$, $G$ be drawn from $\gss(n,p,Q)$, $E$ be a subset of $G$'s edges that independently contains each edge with probability $c$, $v,v'\in G$ and $\sqrt{\ln n}\le r'\le r\in \mathbb{Z}^+$, such that $((1-c)\lambda_{h'}^2/2)^{r+r'}>\lambda_1^{r+r'}n$. Also let $(\lambda'_1,...,\lambda'_{h''})$ be a $h''$-tuple which may depend on $n$ such that $h''=h'$ and $|\lambda'_i-\lambda_i|\le \ln^{-3/2}(n)$ for all $i$. Assume that $N_{r''[G\backslash E]}(v)$ has already been computed for $r''\le r+2m+3$ and that $N_{r''[G\backslash E]}(v')$ has already been computed for $r''\le r'$.  {\it Vertex-product-approximation-algorithm}$(v,v',r,r',E,c,(\lambda'_1,...,\lambda'_{h''}))$ runs in $O(((1-c)\lambda_1)^{r'})$ average time. Furthermore, with probability $1-o(1)$, either $v$ is $(r+2m+4,x)$-bad, $v'$ is $(r'+1,x)$-bad, or {\it Vertex-product-approximation-algorithm}$(v,v',r,r',E,c,(\lambda'_1,...,\lambda'_{h''}))$ returns $(z_1(v\cdot v'),...,z_{h'}(v\cdot v'))$ such that $|z_i(v\cdot v')-w_i(\{v\})\cdot P^{-1}w_i(\{v'\})|<2x(\min p_j)^{-1/2}+x^2+o(1)$ for all $i$.
\end{lemma}

\begin{proof} The slowest part of the algorithm is computing the expressions of the form $N_{r'',r'[E]}(v\cdot v')$. Each of these can be computed in $O(E[|N_{r'[G\backslash E]}(v')|])=O(((1-c)\lambda_1)^{r'})$ average time, and the number of these that need to be computed is constant in $n$. So, this algorithm runs in $O(((1-c)\lambda_1)^{r'})$ average time.

With probability $1-o(1)$ either $v$ is $(r+2m+4,x)$-bad, $v'$ is $(r'+1,x)$-bad, or \[w_i(N_{\sqrt{\ln n}}(v))\cdot P^{-1} w_i(N_{\sqrt{\ln n}}(v'))|\ge \lambda_i^{2\sqrt{\ln n}}/\ln^2 n\] for all $i\le h$ and 
\begin{align*}
|\det(M_{m,r'',r'[E]}&(v\cdot v'))-\frac{c^m}{n^m}\prod_{i=1}^{m-1} ((1-c)\lambda_i)^{r''+r'-2\sqrt{\ln n}} w_i(N_{\sqrt{\ln n}[G\backslash E]}(v))\cdot Qw_i(N_{\sqrt{\ln n}[G\backslash E]}(v'))\\
&\cdot (\gamma((1-c)\lambda_m)^{r''+r'-2\sqrt{\ln n}} w_m(N_{\sqrt{\ln n}[G\backslash E]}(v))\cdot Qw_m(N_{\sqrt{\ln n}[G\backslash E]}(v'))\\
&+\gamma'((1-c)\lambda_{m+1})^{r''+r'-2\sqrt{\ln n}}w_{m+1}(N_{\sqrt{\ln n}[G\backslash E]}(v))\cdot Qw_{m+1}(N_{\sqrt{\ln n}[G\backslash E]}(v')))|\\
&\le \frac{1}{\ln^{2m+2} n}\cdot \frac{c^m}{n^m}\prod_{i=1}^{m} |(1-c)\lambda_i|^{r''+r'}
\end{align*} for all $m\le h'+1$ and $r\le r''<r+4$.

Assume that the later holds.

\begin{align*}
&|\det(M_{i,r+1,r'[E]}(v\cdot v'))-(1-c)^i\lambda'_{i+1}\prod_{j=1}^{i-1}\lambda'_j \det(M_{i,r,r'[E]}(v\cdot v'))\\
&-\gamma\frac{\lambda_i-\lambda_{i+1}}{\lambda_i}\cdot\frac{c^i}{n^i}\prod_{j=1}^{i} ((1-c)\lambda_j)^{r+1+r'-2\sqrt{\ln n}} w_j(N_{\sqrt{\ln n}[G\backslash E]}(v))\cdot Qw_j(N_{\sqrt{\ln n}[G\backslash E]}(v'))|\\
&\le \frac{1}{\ln n}\cdot |\frac{c^i}{n^i}\prod_{j=1}^{i} ((1-c)\lambda_j)^{r+1+r'-2\sqrt{\ln n}} w_j(N_{\sqrt{\ln n}[G\backslash E]}(v))\cdot Qw_j(N_{\sqrt{\ln n}[G\backslash E]}(v'))| 
\end{align*}
for any $i\le h$ with probability $1-o(1)$. So,
\begin{align*}
&|\frac{\det(M_{i,r+1,r'[E]}(v\cdot v')-(1-c)^i\lambda'_{i+1}\prod_{j=1}^{i-1} \lambda'_j\det(M_{i,r,r'[E]}(v\cdot v')}{\det(M_{i-1,r+1,r'[E]}(v\cdot v')-(1-c)^{i-1}\lambda'_{i}\prod_{j=1}^{i-2} \lambda'_j \det(M_{i-1,r,r'[E]}(v\cdot v')}\\
&-\frac{\gamma(\{(1-c)\lambda_j\},i)}{\gamma(\{(1-c)\lambda_j\},i-1)}\cdot \frac{\lambda_{i-1}(\lambda_i-\lambda_{i+1})}{\lambda_i(\lambda_{i-1}-\lambda_i)}\\
&\indent\indent\cdot\frac{c}{n} ((1-c)\lambda_i)^{r+r'-2\sqrt{\ln n}+1} w_i(N_{\sqrt{\ln n}[G\backslash E]}(v))\cdot Qw_i(N_{\sqrt{\ln n}[G\backslash E]}(v'))|\\
&\le \frac{1}{n\sqrt{\ln n}}\cdot|((1-c)\lambda_i)^{r+r'+1}| 
\end{align*} 
with probability $1-o(1)$. Therefore,
\[|z_i(v\cdot v')-\lambda_i^{-1}((1-c)\lambda_i)^{-2\sqrt{\ln n}} w_i(N_{\sqrt{\ln n}[G\backslash E]}(v))\cdot Qw_i(N_{\sqrt{\ln n}[G\backslash E]}(v'))|=o(1)\] with probability $1-o(1)$. That implies that

 \begin{align*}
&|z_i-w_i(\{v\})\cdot P^{-1}w_i(\{v'\})|\\
&\le |z_i-(1-c)((1-c)\lambda_i)^{-2\sqrt{\ln n}-1}w_i(N_{\sqrt{\ln n}[G\backslash E]}(v))\cdot Qw_i(N_{\sqrt{\ln n}[G\backslash E]}(v'))|\\
&\phantom{xxxxxx}+(1-c)\cdot|((1-c)\lambda_i)^{-2\sqrt{\ln n}-1}w_i(N_{\sqrt{\ln n}[G\backslash E]}(v))\cdot Q w_i(N_{\sqrt{\ln n}[G\backslash E]}(v'))\\
&\phantom{\phantom{xxxxxx}+(1-c)\cdot|}-((1-c)\lambda_i)^{-\sqrt{\ln n}-1}w_i(\{v\})\cdot Qw_i(N_{\sqrt{\ln n}[G\backslash E]}(v'))|\\
&\phantom{xxxxxx}+|(1-c)\cdot((1-c)\lambda_i)^{-\sqrt{\ln n}-1}w_i(\{v\})\cdot Qw_i(N_{\sqrt{\ln n}[G\backslash E]}(v'))-w_i(\{v\})\cdot P^{-1}w_i(\{v'\}))|\\
&\le |((1-c)\lambda_i)^{-\sqrt{\ln n}}w_i(N_{\sqrt{\ln n}[G\backslash E]}(v'))\cdot P^{-1}[((1-c)\lambda_i)^{-\sqrt{\ln n}}w_i(N_{\sqrt{\ln n}[G\backslash E]}(v))-w_i(\{v\})]|\\
&\phantom{xxxxxx}+|w_i(\{v\})\cdot P^{-1}[((1-c)\lambda_i)^{-\sqrt{\ln n}}w_i(N_{\sqrt{\ln n}[G\backslash E]}(v'))-w_i(\{v'\})]+o(1)
\end{align*}
By goodness of $v$ and $v'$, this is less than or equal to
\begin{align*}
& ((1-c)\lambda_i)^{-\sqrt{\ln n}}\sqrt{w_i(N_{\sqrt{\ln n}[G\backslash E]}(v'))\cdot P^{-1}w_i(N_{\sqrt{\ln n}[G\backslash E]}(v'))}\cdot x\\
&\indent\indent\indent +\sqrt{w_i(\{v\})\cdot P^{-1}w_i(\{v\})}\cdot x+o(1)\\
&\le\sqrt{w_i(\{v'\})\cdot P^{-1}w_i(\{v'\})+2w_i(\{v'\})\cdot P^{-1}[((1-c)\lambda_i)^{-\sqrt{\ln n}}w_i(N_{\sqrt{\ln n}[G\backslash E]}(v'))-w_i(\{v'\})] +}\\
&\indent\indent\indent \overline{[((1-c)\lambda_i)^{-\sqrt{\ln n}}w_i(N_{\sqrt{\ln n}[G\backslash E]}(v'))-w_i(\{v'\})]}\\
&\indent\indent\indent \overline{\cdot P^{-1}[((1-c)\lambda_i)^{-\sqrt{\ln n}}w_i(N_{\sqrt{\ln n}[G\backslash E]}(v'))-w_i(\{v'\})]}\cdot x\\
&\indent\indent\indent +x\sqrt{1/\min p_j}+o(1)\\
&\le \sqrt{1/\min p_j+2x/\sqrt{\min p_j}+x^2}x+x/\sqrt{\min p_j}+o(1)\\
&=(x^2+2x(\min p_j)^{-1/2})+o(1)
\end{align*}
with probability $1-o(1)$, as desired.
\end{proof}

For any two vertices in different communities, $v$ and $v'$, the fact that $Q$'s rows are distinct implies that $Q(\overrightarrow{\{v\}}-\overrightarrow{\{v'\}})\ne 0$. So, $w_i(\{v\})\ne w_i(\{v'\})$ for some $1\le i\le h'$. That means that for any two vertices $v$ and $v'$,  
\begin{align*}
&(w_i(\{v\})- w_i(\{v'\}))\cdot P^{-1}(w_i(\{v\})- w_i(\{v'\}))\\
&=w_i(\{v\})\cdot P^{-1}w_i(\{v\})-2w_i(\{v\})\cdot P^{-1}w_i(\{v'\})+w_i(\{v'\})\cdot P^{-1}w_i(\{v'\})\ge 0
\end{align*}
for all $1\le i\le h'$, with equality for all $i$ if and only if $v$ and $v'$ are in the same community. This also implies that given a vertex $v$, another vertex in the same community $v'$, and a vertex in a different community $v''$, 
\begin{align*}
&2w_i(\{v\})\cdot P^{-1}w_i(\{v'\})-w_i(\{v'\})\cdot P^{-1}w_i(\{v'\}) \\&\ge 2w_i(\{v\})\cdot P^{-1}w_i(\{v''\})-w_i(\{v''\})\cdot P^{-1}w_i(\{v''\})
\end{align*}
 for all $1\le i\le h'$ and the inequality is strict for at least one $i$. This suggests the following algorithms for classifying vertices.

\begin{algorithm}
Vertex-comparison-algorithm(v,v', r,r',E,x,c,$(\lambda'_1,...,\lambda'_{h''})$):

(Assumes that $N_{r''[G\backslash E]}(v)$ and $N_{r''[G\backslash E]}(v')$ have already been computed for $r''\le r+2h''+3$)

\phantom{xxx} Run {\it Vertex-product-approximation-algorithm}$(v,v',r,r',E,c,(\lambda'_1,...,\lambda'_{h''}))$, {\it Vertex-product-approximation-algorithm}$(v,v,r,r',E,c,(\lambda'_1,...,\lambda'_{h''}))$, and {\it Vertex-product-approximation-algorithm}$(v',v',r,r',E,c,(\lambda'_1,...,\lambda'_{h''}))$.

\phantom{xxx} If $\exists i: z_i(v\cdot v)-2z_i(v\cdot v')+z_i(v'\cdot v')> 5(2x(\min p_j)^{-1/2}+x^2)$ then conclude that $v$ and $v'$ are in different communities.

\phantom{xxx} Otherwise, conclude that $v$ and $v'$ are in the same community.
\end{algorithm}

\begin{lemma}
Assuming that each of $G$'s edges was independently assigned to $E$ with probability $c$, this algorithm runs in $O(((1-c)\lambda_1)^{\max(r,r')})$ average time. Furthermore, if each execution of {\it Vertex-product-approximation-algorithm} succeeds and $13(2x(\min p_j)^{-1/2}+x^2)$ is less than the minimum nonzero value of $(w_i(\{v\})- w_i(\{v'\}))\cdot P^{-1}(w_i(\{v\})- w_i(\{v'\}))$ then the algorithm returns the correct result with probability $1-o(1)$.
\end{lemma}

\begin{proof}
The slowest step of the algorithm is using {\it Vertex-product-approximation-algorithm}. This runs in an average time of $O(((1-c)\lambda_1)^{\max(r,r')})$ and must be done $3$ times. If each execution of {\it Vertex-product-approximation-algorithm} succeeds then with probability $1-o(1)$ the $z_i$ are all within $\frac{6}{5}(2x(\min p_j)^{-1/2}+x^2)$ of the products they seek to approximate, in which case $$z_i(v\cdot v)-2z_i(v\cdot v')+z_i(v'\cdot v')> 5(2x(\min p_j)^{-1/2}+x^2)$$ if and only if $$(w_i(\{v\})- w_i(\{v'\}))\cdot P^{-1}(w_i(\{v\})- w_i(\{v'\}))\ne 0,$$ which is true for some $i$ if and only if $v$ and $v'$ are in different communities.
\end{proof}

\begin{algorithm}
Vertex-classification-algorithm(v[],v', r,r',E,c,$(\lambda'_1,...,\lambda'_{h''})$):

(Assumes that $N_{r''[G\backslash E]}(v[\sigma])$have already been computed for $0\le \sigma<k$ and $r''\le r+2h''+3$, that $N_{r''[G\backslash E]}(v')$ has already been computed for all $r''\le r'$, and that $z_i(v[\sigma]\cdot v[\sigma] )$ has already been computed for each $i$ and $\sigma$)

\phantom{xxx} Run {\it Vertex-product-approximation-algorithm}$(v[\sigma],v',r,r',E,c,(\lambda'_1,...,\lambda'_{h''}))$ for each $\sigma$.

\phantom{xxx} Find a $\sigma$ that minimizes the value of 
\[\max_{\sigma'\ne\sigma, i\le h''}z_i(v[\sigma]\cdot v[\sigma])-2z_i(v[\sigma]\cdot v')-[ z_i(v[\sigma']\cdot v[\sigma'])-2z_i(v[\sigma']\cdot v')]\] and conclude that $v'$ is in the same community as $v[\sigma]$.
\end{algorithm}

\begin{lemma}
Assuming that $E$ was generated properly, this algorithm runs in $O(((1-c)\lambda_1)^{r'})$ average time. Let $x>0$ and assume that each execution of  {\it Vertex-product-approximation-algorithm} succeeds (including the previous ones to compute $z_i(v[\sigma]\cdot v[\sigma] )$), and that $v[]$ contains exactly one vertex from each community. Also, assume that $13(2x(\min p_j)^{-1/2}+x^2)$ is less than the minimum nonzero value of $(w_i(\{v\})- w_i(\{v'\}))\cdot P^{-1}(w_i(\{v\})- w_i(\{v'\}))$. Then this algorithm classifies $v'$ correctly with probability $1-o(1)$.
\end{lemma}

\begin{proof}
Again, the slowest step of the algorithm is running {\it Vertex-product-approximation-algorithm}. This runs in an average time of  $O(((1-c)\lambda_1)^{r'})$ and must be done $k$ times. If the conditions given above are satisifed, then each $z_i$ is within  $\frac{21}{20}(2x(\min p_j)^{-1/2}+x^2)$ of the product it seeks to approximate with probability $1-o(1)$. If this is the case, then \begin{align*}
&2w_i(\{v'\})\cdot P^{-1}w_i(\{v[\sigma]\})-w_i(\{v[\sigma]\})\cdot P^{-1}w_i(\{v[\sigma]\})\\
&\ge 2w_i(\{v'\})\cdot P^{-1}w_i(\{v[\sigma']\})-w_i(\{v[\sigma']\})\cdot P^{-1}w_i(\{v[\sigma']\})
\end{align*}
for all $i$ and $\sigma'$ if $v'$ is in the same community as $v[\sigma]$, and 
\begin{align*}
&2w_i(\{v'\})\cdot P^{-1}w_i(\{v[\sigma]\})-w_i(\{v[\sigma]\})\cdot P^{-1}w_i(\{v[\sigma]\})\\
&\le 2w_i(\{v'\})\cdot P^{-1}w_i(\{v[\sigma']\})-w_i(\{v[\sigma']\})\cdot P^{-1}w_i(\{v[\sigma']\})-13(2x(\min p_j)^{-1/2}+x^2)
\end{align*}
for some $i$ and $\sigma$ otherwise. So,
\[z_i(v[\sigma]\cdot v[\sigma])-2z_i(v[\sigma]\cdot v')\le z_i(v[\sigma']\cdot v[\sigma])-2z_i(v[\sigma']\cdot v') +\frac{19}{3}\cdot (2x(\min p_j)^{-1/2}+x^2)\]
for all $i$ and $\sigma$ iff $v'$ is in the same community as $v[\sigma]$. So, 
\[\max_{\sigma'\ne\sigma, i\le h''}z_i(v[\sigma]\cdot v[\sigma])-2z_i(v[\sigma]\cdot v')-[ z_i(v[\sigma']\cdot v[\sigma'])-2z_i(v[\sigma']\cdot v')]\le \frac{19}{3}\cdot (2x(\min p_j)^{-1/2}+x^2)\] 
iff $v'$ is in the same community as $v[\sigma]$. Therefore, the algorithm returns the correct result with probability $1-o(1)$.
\end{proof}

At this point, we can finally start giving algorithms for classifying a graph's vertices.
\begin{algorithm}
Unreliable-graph-classification-algorithm(G,c,m,$\epsilon$,x,$(\lambda'_1,...,\lambda'_{h''})$):

\phantom{xxx} Randomly assign each edge in $G$ to $E$ independently with probability $c$.

\phantom{xxx} Randomly select $m$ vertices in $G$, $v[0],...,v[m-1]$.

\phantom{xxx} Let $r=(1-\frac{\epsilon}{3})\log n/\log ((1-c)\lambda'_1)-\sqrt{\ln n}$ and $r'=\frac{2\epsilon}{3}\cdot \log n/\log ((1-c)\lambda'_1)$

\phantom{xxx} Compute $N_{r''[G\backslash E]}(v[i])$ for each $r''\le r+2h''+3$ and $0\le i<m$.

\phantom{xxx} Run {\it vertex-comparison-algorithm}$(v[i],v[j],r,r',E,x,(\lambda'_1,...,\lambda'_{h''}))$ for every $i$ and $j$

\phantom{xxx} If these give consistent results, randomly select one alleged member of each community $v'[\sigma]$. Otherwise, fail.

\phantom{xxx} For every $v''$ in the graph, compute $N_{r''[G\backslash E]}(v'')$ for each $r''\le r'$. Then, run\newline {\it vertex-classification-algorithm}$(v'[],v'', r,r',E,(\lambda'_1,...,\lambda'_{h''}))$ in order to get a hypothesized classification of $v''$

\phantom{xxx} Return the resulting classification.
\end{algorithm}

\begin{lemma}
Let $x'>0$, and assume that all of the following hold:
\begin{align*}
&\epsilon<1\\
&(1-c)\lambda_{h'}^4>4\lambda_1^3\\
&0<x\le x'<\frac{\lambda_1k}{\lambda_{h'}\min p_i}\\
&(2(1-c)\lambda_1^3/\lambda_{h'}^2)^{1-\epsilon/3}<(1-c)\lambda_1\\
&(1+\epsilon/3)>\log((1-c)\lambda_1)/\log ((1-c)\lambda_{h'}^2/2\lambda_1)\\
&13(2x'(\min p_j)^{-1/2}+(x')^2)<\min_{\ne0} (w_i(\{v\})- w_i(\{v'\}))\cdot P^{-1}(w_i(\{v\})- w_i(\{v'\}))\\
&\exists k \text{ such that every entry of } Q^k \text{ is positive}\\
&\exists w\in \mathbb{R}^k \text{ such that } QPw=\lambda_1w, w\cdot Pw=1, \text{ and } x\le \min w_i/2.\\
&h''=h'\\
&|\lambda_i-\lambda'_i|\le \ln^{-3/2}(n) \text{ for all } i
\end{align*}

Let $y=2ke^{-\frac{x^2(1-c)\lambda_{h'}^2\min p_i}{16\lambda_1k^{3/2}((\min p_i)^{-1/2}+x)}}/\left(1-e^{-\frac{x^2(1-c)\lambda_{h'}^2\min p_i}{16\lambda_1k^{3/2}((\min p_i)^{-1/2}+x)}\cdot((\frac{(1-c)\lambda_{h'}^4}{4\lambda_1^3})-1)}\right)$

and $y'=2ke^{-\frac{x^{\prime 2}(1-c)\lambda_{h'}^2\min p_i}{16\lambda_1k^{3/2}((\min p_i)^{-1/2}+x')}}/\left(1-e^{-\frac{x^{\prime 2}(1-c)\lambda_{h'}^2\min p_i}{16\lambda_1k^{3/2}((\min p_i)^{-1/2}+x')}\cdot((\frac{(1-c)\lambda_{h'}^4}{4\lambda_1^3})-1)}\right)$

This algorithm runs in $O(m^2 n^{1-\frac{\epsilon}{3}}+n^{1+\frac{2}{3}\epsilon})$ time. Furthermore, with probability $1-o(1)$, $G$ is such that {\it Unreliable-graph-classification-algorithm}$(G,c,m,\epsilon,x,(\lambda'_1,...,\lambda'_{h''}))$ has at least a \[1-k(1-\min p_i)^m-my\] chance of classifying at least $1-y'$ of $G$'s vertices correctly.

\end{lemma}

\begin{proof}

Let $r_0=(1-\frac{\epsilon}{3})\log n/\log ((1-c)\lambda_1)$. There exists $y''<y$ such that if these conditions hold, then with probability $1-o(1)$, at least $1-y''$ of $G$'s vertices are $(r_0,x)$-good and the number of vertices in $G$ in community $\sigma$ is within $\sqrt{n}\log n$ of $p_\sigma n$ for all $\sigma$. If this is the case, then for sufficiently large $n$, it is at least $1-k(1-\min p_i)^m-my$ likely that every one of the $m$ randomly selected vertices is $(r_0,x)$-good and at least one is selected from each community. 

If $v[i]$ is $(r_0,x)$-good for all $i$, then with probability $1-o(1)$, {\it vertex-comparison-algorithm}$(v[i],v[j],r,r',E,x, c, (\lambda'_1,...,\lambda'_{h''}))$ determines whether or not $v[i]$ and $v[j]$ are in the same community correctly for every $i$ and $j$, allowing the algorithm to pick one member of each community. If that happens, then the algorithm will classify each $(r'+h',x')$-good vertex correctly with probability $1-o(1)$. So, as long as the initial selection of $v[]$ is good, the algorithm classifies at least $1-y'$ of the graph's vertices correctly with probability $1-o(1)$.

Generating $E$ and $v[]$ takes $O(n)$ time. Computing $N_{r''[G\backslash E]}(v[i])$ for all $r''\le r+2h'+3$ takes $O(m|\cup_{r''} N_{r''[G\backslash E]}(v[i]))=O(mn)$ time, and computing $N_{r''[G\backslash E]}(v')$ for all $r''\le r'$ and $v'\in G$ takes $$O(n|\cup_{r''\le r'} N_{r''[G\backslash E]})=O(n\cdot ((1-c)\lambda_1)^{r'})=O(n^{1+\frac{2}{3}\epsilon})$$ time. Once these have been computed, running {\it Vertex-comparison-\\algorithm}$(v[i],v[j],r,r',E,x,(\lambda'_1,...,\lambda'_{h''}))$ for every $i$ and $j$ takes $O(m^2\cdot ((1-c)\lambda_1)^r)= O(m^2 n^{1-\frac{\epsilon}{3}})$ time, at which point an alleged member of each community can be found in $O(m^2)$ time.  Running {\it Vertex-classification-algorithm}$(v'[],v'', r,r',E,c,(\lambda'_1,...,\lambda'_{h''}))$ for every $v''\in G$ takes $O(n\cdot ((1-c)\lambda_1)^{r'})=O(n^{1+\frac{2}{3}\epsilon})$ time. So, the overall algorithm runs in $O(m^2n^{1-\frac{\epsilon}{3}}+n^{1+\frac{2}{3}\epsilon})$ average time.
\end{proof}

So, this algorithm can sometimes give a vertex classification that is nontrivially better than that obtained by guessing. However, it has an assymptotically nonzero failure rate and requires too much information about the graph's parameters. In order to get around that, we combine the results of multiple executions of the algorithm and add in a parameter analysis procedure as follows.
\begin{algorithm}
Reliable-graph-classification-algorithm(G,m,$\delta$, T(n)) (i.e., {\tt Agnostic-sphere-comparison}):

\phantom{xxx} Run {\it Improved-Eigenvalue-approximation-algorithm}(.1) in order to compute $(\lambda'_1,...,\lambda'_{h''})$

\phantom{xxx} Let $\lambda''_1=\lambda'_1+2\ln^{-3/2}(n)$, $\lambda''_{h''}=\lambda'_{h''}-2\ln^{-3/2}(n)$, and $k'=\lfloor 1/\delta\rfloor$  

\phantom{xxx} Let $x$ be the smallest rational number of minimal numerator such that 
\[k'(1-\delta)^m+m\cdot 2k'e^{-\frac{x^2(\lambda''_{h''})^2 \delta}{16\lambda''_1(k')^{3/2}((\delta)^{-1/2}+x)}}/\left(1-e^{-\frac{x^2(\lambda''_{h''})^2\delta}{16\lambda''_1(k')^{3/2}((\delta)^{-1/2}+x)}\cdot((\frac{(\lambda''_{h''})^4}{4(\lambda''_1)^3})-1)}\right)<\frac{1}{2}\]

\phantom{xxx} Let $\epsilon$ be the smallest rational number of the form $\frac{1}{z}$ or $1-\frac{1}{z}$ such that $(2(\lambda''_1)^3/(\lambda''_{h''})^2)^{1-\epsilon/3}<\lambda''_1$ and
$(1+\epsilon/3)>\log(\lambda''_1)/\log ((\lambda''_{h''})^2/2\lambda''_1)$

\phantom{xxx} Let $c$ be the largest unit reciprocal less than $1/9$ such that all of the following hold:
\begin{align*}
&(1-c)(\lambda''_{h''})^4>4(\lambda''_1)^3\\
&(2(1-c)(\lambda''_1)^3/(\lambda''_{h'})^2)^{1-\epsilon/3}<(1-c)\lambda''_1\\
&(1+\epsilon/3)>\log((1-c)\lambda''_1)/\log ((1-c)(\lambda''_{h'})^2/2\lambda''_1)\\
&k'(1-\delta)^m+m\cdot 2k'e^{-\frac{x^2(1-c)(\lambda''_{h'})^2\delta}{16\lambda''_1(k')^{3/2}((\delta)^{-1/2}+x)}}/\left(1-e^{-\frac{x^2(1-c)(\lambda''_{h''})^2\delta}{16\lambda''_1(k')^{3/2}((\delta)^{-1/2}+x)}\cdot((\frac{(1-c)(\lambda''_{h''})^4}{4(\lambda''_1)^3})-1)}\right)<\frac{1}{2}
\end{align*}

\phantom{xxx} Run {\it Unreliable-graph-classification-algorithm}$(G,c,m,\epsilon,x,(\lambda'_1,...,\lambda'_{h''}))$ $T(n)$ times and record the resulting classifications.

\phantom{xxx} Find the smallest $y''$ such that there exists a set of more than half of the classifications no two of which have more than $y''$ disagreement, and discard all classifications not in the set. In this step, define the disagreement between two classifications as the minimum disagreement over all bijections between their communities.

\phantom{xxx} For every vertex in $G$, randomly pick one of the remaining classifications and assert that it is in the community claimed by that classification, where a community from one classification is assumed to correspond to the community it has the greatest overlap with in each other classification.

\phantom{xxx} Return the resulting combined classification.
\end{algorithm}

\begin{lemma}
Assume that there exist $x$, $x'$, and $\epsilon$ such that $x$ is either a unit reciprocal or an integer, $\epsilon$ is a rational number of the form $\frac{1}{z}$ or $1-\frac{1}{z}$, and all of the following hold:
\begin{align*}
&2ke^{-\frac{.9x^2\lambda_{h'}^2\min p_i}{16\lambda_1 k^{3/2}((\min p_i)^{-1/2}+x)}}/\left(1-e^{-\frac{.9x^2\lambda_{h'}^2\min p_i}{16\lambda_1k^{3/2}((\min p_i)^{-1/2}+x)}\cdot((\frac{.9\lambda_{h'}^4}{4\lambda_1^3})-1)}\right)<\frac{1}{2}\\
&.9(\lambda_{h'}^2/2)^{4}>\lambda_1^{7}\\
&\lambda_{h'}^4>4\lambda_1^3\\
&0<x\le x'<\frac{\lambda_1k}{\lambda_{h'}\min p_i}\\
&(2\lambda_1^3/\lambda_{h'}^2)^{1-\epsilon/3}<\lambda_1\\
&(1+\epsilon/3)>\log(\lambda_1)/\log (\lambda_{h'}^2/2\lambda_1)\\
&13(2x'(\min p_j)^{-1/2}+(x')^2)<\min_{\ne0} (w_i(\{v\})- w_i(\{v'\}))\cdot P^{-1}(w_i(\{v\})- w_i(\{v'\}))\\
&\text{every entry of } Q^k \text{ is positive}\\
&\exists w\in \mathbb{R}^k \text{ such that } QPw=\lambda_1w, w\cdot Pw=1, \text{ and } x\le \min w_i/2.\\
&\delta\le \min p_i\\
& \lfloor 1/\delta\rfloor\cdot (1-\delta)^m+m\cdot 2\lfloor 1/\delta\rfloor e^{-\frac{x^2\lambda_{h'}^2\delta}{16\lambda_1\lfloor 1/\delta\rfloor^{3/2}(\delta^{-1/2}+x)}}/\left(1-e^{-\frac{x^2\lambda_{h'}^2\delta}{16\lambda_1\lfloor 1/\delta\rfloor^{3/2}(\delta^{-1/2}+x)}\cdot((\frac{\lambda_{h'}^4}{4\lambda_1^3})-1)}\right)< \frac{1}{2}\\
&T(n)=w(1)\\
&T(n)\le \ln(n)\\
&\min p_i>8ke^{-\frac{.9x^{\prime 2}\lambda_{h'}^2\min p_i}{16\lambda_1k^{3/2}((\min p_i)^{-1/2}+x')}}/\left(1-e^{-.9\frac{x^{\prime 2}\lambda_{h'}^2\min p_i}{16\lambda_1k^{3/2}((\min p_i)^{-1/2}+x')}\cdot((\frac{.9\lambda_{h'}^4}{4\lambda_1^3})-1)}\right)
\end{align*}

With probability $1-o(1)$, {\it Reliable-graph-classification-algorithm}(G,m,$\delta$, T(n)) runs in $O(m^2 n^{1-\frac{\epsilon}{3}}T(n)+n^{1+\frac{2}{3}\epsilon}T(n))$ time and classifies at least $1-3y'$ of $G$'s vertices correctly, where 
\[y'=2ke^{-\frac{.9x^{\prime 2}\lambda_{h'}^2\min p_i}{16\lambda_1k^{3/2}((\min p_i)^{-1/2}+x')}}/\left(1-e^{-.9\frac{x^{\prime 2}\lambda_{h'}^2\min p_i}{16\lambda_1k^{3/2}((\min p_i)^{-1/2}+x')}\cdot((\frac{.9\lambda_{h'}^4}{4\lambda_1^3})-1)}\right)\]
\end{lemma}

\begin{proof}
With probability $1-o(1)$, {\it Improved-Eigenvalue-approximation-algorithm} gives output such that $h''=h'$ and $|\lambda_i-\lambda'_i|\le \ln^{-3/2}(n)$ for each $i$. Assuming that this holds, the algorithm finds the largest $x$ and $\epsilon$ that satisfy the conditions above and the largest unit reciprocal less than $1/9$, $c$, that satisifes the conditions for {\it Unreliable-graph-classification-algorithm}$(G,c,m,\epsilon,x,(\lambda'_1,...,\lambda'_{h''}))$ to have a greater than $1/2$ success rate for any $k\le \lfloor 1/\delta\rfloor$ with probability $1-o(1)$. Since its success rate is greater than $1/2$ and $T(n)=\omega(1)$, more than half of the executions of {\it Unreliable-graph-classification-algorithm} give classifications with error $y'$ or less with probability $1-o(1)$. That means that $y''\le 2y'$ and at least one of the classifications in the selected set must have error $y'$ or less. Thus, all of the selected classifications have error $3y'$ or less. The requirement that $\min p_i>4y'$ ensures that the bijection between any two of these classifications' communities that minimizes disagreeement is the identity. Therefore, this algorithm classifies at least $1-3y'$ of the vertices correctly, as desired.

Assuming this all works correctly, {\it Improved-Eigenvalue-approximation-algorithm} runs in $O(n\ln(n))$ time. Finding $x$, $\epsilon$, and $c$ takes constant time, and running {\it Unreliable-graph-classification-algorithm} $T(n)$ times takes $O(m^2n^{1-\frac{\epsilon}{3}}T(n)+n^{1+\frac{2}{3}\epsilon}T(n))$. Computing the degree of agreement between each pair of classifications takes $O(n\ln^2(n))$ time. $T(n)<\ln n$, so the brute force algorithm finds $y''$ and the corresponding set in $O(n)$ time. Combining the classifications takes $O(n)$ time. Therefore, this whole algorithm runs in $O(m^2 n^{1-\frac{\epsilon}{3}}T(n)+n^{1+\frac{2}{3}\epsilon}T(n))$ time with probability $1-o(1)$, as desired.
\end{proof}

\begin{proof}[Proof of Theorem \ref{thm1}]
If the conditions hold, {\it Reliable-graph-classification-\\algorithm}$(G,\ln(4\lfloor 1/\delta\rfloor)/\delta,\delta, \ln(n))$ has the desired properties by the previous lemma.
\end{proof}



\subsection{Exact recovery}\label{exact-sec}

Recall that $p$ is a probability vector of dimension $k$, $Q$ is a $k \times k$ symmetric matrix with positive entries, and 
$\gs(n,p,Q)$ denotes the stochastic block model with community prior $p$ and connectivity matrix $\ln(n)Q/n$. A random graph $G$ drawn under $\gs(n,p,Q)$ has a planted community assignment, which we denote by $\sigma \in [k]^n$ and call sometime the true community assignment.

Recall also that exact recovery is solvable for a community partition $[k] = \sqcup_{s=1}^t A_s$, if there exists an algorithm that assigns to each node in $G$ an element of $\{A_1,\dots,A_t\}$ that contains its true community\footnote{Up to a relabelling of the communities.} with probability $1-o_n(1)$.  Exact recovery is solvable in $SBM(n,p,W)$ if it is solvable for the partition of $[k]$ into $k$ singletons, i.e., all communities can be recovered. 

\subsubsection{Formal results}
\begin{definition}
Let $\mu,\nu$ be two positive measures on a discrete set $\X$, i.e., two functions from $\mathcal{X}$ to $\mR_+$. We define the CH-divergence between $\mu$ and $\nu$ by 
\begin{align} 
\dd(\mu, \nu) :=\max_{t \in [0,1]} \sum_{x \in \X} \left( t \mu(x) + (1-t)\nu(x)- \mu(x)^t \nu(x)^{1-t} \right). \label{h-div}
\end{align}
\end{definition}
Note that for a fixed $t$, $$\sum_{x \in \X} \left( t \mu(x) + (1-t)\nu(x)- \mu(x)^t \nu(x)^{1-t} \right)$$ is an $f$-divergence. 
For $t=1/2$, i.e., the gap between the arithmetic and geometric means, we have
\begin{align}
\sum_{x \in \X} t\mu(x) + (1-t)\nu(x)- \mu(x)^t \nu(x)^{1-t} = \frac{1}{2} \| \sqrt{\mu}- \sqrt{\nu} \|_2^2
\end{align}
which is the Hellinger divergence (or distance), and the maximization over $t$ of the part $\sum_x \mu(x)^t \nu(x)^{1-t}$ is the exponential of the Chernoff divergence. 
We refer to Section 8.3 for further discussions on $\dd$. Note also that we will often evaluate $\dd$ as $\dd(x,y)$ where $x,y$ are vectors instead of measures.

\begin{definition}
For the SBM $\gs(n,p,Q)$, where $p$ has dimension $k$ (i.e., there are $k$ communities), the finest partition of $[k]$ is the partition of $[k]$ in to the largest number of subsets such that $\dd((PQ)_i,(PQ)_j) \geq 1$ for all $i,j$ that are in different subsets.
\end{definition}

We next present our main theorem for exact recovery. We first provide necessary and sufficient conditions for exact recovery of partitions, and then provide an agnostic algorithm that solves exact recovery efficiently, more precisely, in quasi-linear time. 

Recall that from \cite{colin1} exact recovery is solvable in the stochastic block model $\gs(n,p,Q)$ for a partition $[k] = \sqcup_{s=1}^t A_s$ if and only if for all $i$ and $j$ in different subsets of the partition,
\begin{align}
\dd ((PQ)_i , (PQ)_j) \geq 1, \label{d1}
\end{align}
where $(PQ)_i$ denotes the $i$-th row of the matrix $PQ$. In particular, exact recovery is solvable in $\gs(n,p,Q)$ if and only if $\min_{i,j \in [k], i \neq j} \dd ((PQ)_i , (PQ)_j) \geq 1$.

\begin{theorem}\label{thm2}
Let $k \in \mZ_+$ denote the number of communities, $p \in (0,1)^k$ with $|p|=1$ denote the community prior, $P=\diag(p)$, and let $Q \in (0,\infty)^{k \times k}$ symmetric with no two rows equal. For $G \sim \gs(n,p,Q)$, the algorithm {\tt Agnostic-degree-profiling}$(G,\frac{\ln\ln n}{4 \ln n})$ recovers the finest partition, runs in $o(n^{1+\epsilon})$ time for all $\epsilon>0$, and does not need to know the parameters (it uses no input except the graph in question). 
\end{theorem}
Note that the second item in the theorem implies that {\tt Agnostic-degree-profiling} solves exact recovery efficiently whenever the parameters $p$ and $Q$ allow for exact recovery to be solvable. 

\begin{remark}\label{qzero}
If $Q_{ij}=0$ for some $i$ and $j$ then the results above still hold, except that if for all $i$ and $j$ in different subsets of the partition,
\begin{align}
\dd ((PQ)_i , (PQ)_j) \geq 1, \label{d1}
\end{align}
but there exist $i$ and $j$  in different subsets of the partition such that $\dd ((PQ)_i , (PQ)_j) = 1$ and $((PQ)_{i,k}\cdot (PQ)_{j,k}\cdot ((PQ)_{i,k}-(PQ)_{j,k})=0$ for all $k$, then the optimal algorithm will have an assymptotically constant failure rate. The recovery algorithm also needs to be modified to accomodate $0$'s in $Q$.
\end{remark}



The algorithm {\tt Agnostic-degree-profiling} is given in Section \ref{pt1} and replicated below. The idea is to recover the communities with a two-step procedure, similarly to one of the algorithms used in \cite{abh} for the two-community case. In the first step, we run {\tt Agnostic-sphere-comparison} on a sparsified version of $\gs(n,p,Q)$ which has a slowly growing average degree. Hence, from Corollary \ref{partial-delta}, {\tt Agnostic-sphere-comparison} recovers correctly a fraction of nodes that is arbitrarily close to 1 (w.h.p.). In the second step, we proceed to an improvement of the first step classification by making local checks for each node in the residue graph and deciding whether the node should be moved to another community or not. This step requires solving a hypothesis testing problem for deciding the local degree profile of vertices in the SBM. The CH-divergence appears when resolving this problem, as the mis-classification error exponent. We present this result of self-interest in Section \ref{testing}. The proof of Theorem \ref{thm2} is given in Section \ref{proof-thm-2}. \\

{\tt Agnostic-degree-profiling} algorithm. \\
Inputs: a graph $G=([n],E)$, and a splitting parameter $\gamma \in [0,1]$ (see Theorem \ref{thm2} for the choice of $\gamma$).\\
Output: Each node $v \in [n]$ is assigned a community-list $A(v) \in \{A_1,\dots,A_t\}$, where $A_1,\dots,A_t$ is intended to be the partition of $[k]$ in to the largest number of subsets such that $\dd((pQ)_i,(pQ)_j) \geq 1$ for all $i,j$ in $[k]$ that are in different subsets.\\
Algorithm:\\
(1) Define the graph $G'$ on the vertex set $[n]$ by selecting each edge in $G$ independently with probability $\gamma$, and define the graph $G''$ that contains the edges in $G$ that are not in $G'$. \\
(2) Run {\tt Agnostic-sphere-comparison} on $G'$ to obtain the preliminary classification $\sigma' \in [k]^n$ (see Section \ref{partial-sec} and Corollary \ref{partial-delta}.) \\
(3) Estimate $p$ and $Q$ based on the alleged communities' sizes and the edge densities between them.\\
(4) For each node $v \in [n]$, determine in which community node $v$ is most likely to belong to based on its degree profile computed from the preliminary classification $\sigma'$ and the estimates of $p$ and $Q$ (see Section \ref{testing}), and call it $\sigma''_v$\\
(5) Re-estimate $p$ and $Q$ based on the sizes of the communities claimed by $\sigma''$ and the edge densities between them.\\
(6) For each node $v \in [n]$, determine in which group $A_1,\dots,A_t$ node $v$ is most likely to belong to based on its degree profile computed from the preliminary classification $\sigma''$ and the new estimates of $P$ and $Q$ (see Section \ref{testing}). 

\subsubsection{Testing degree profiles}\label{testing}
In this section, we consider the problem of deciding which community a node in the SBM belongs to based on its degree profile. 
The notions are the same as in \cite{colin1}, we repeat them to ease the reading and explain how the {\tt Agnostic-degree-profiling} algorithm works.   
\begin{definition}
The degree profile of a node $v \in [n]$ for a partition of the graph's vertices into $k$ communities is the vector $d(v) \in \mZ_+^k$, where the $j$-th component $d_j(v)$ counts the number of edges between $v$ and the vertices in community $j$.  Note that $d(v)$ is equal to $N_1(v)$ as defined in Definition \ref{def-n1}.
\end{definition}

For $G \sim \gs(n,p,Q)$, community $i \in [k]$ has a relative size that concentrates exponentially fast to $p_i$. Hence, for a node $v$ in community $j$, $d(v)$ is approximately given by $\sum_{i \in [k]} X_{ij}e_i$, 
where $X_{ij}$ are independent and distributed as $\bin(np_i,\ln(n)Q_{i,j}/n)$, and where $\bin(a,b)$ denotes\footnote{$\bin(a,b)$ refers to $\bin( \lfloor a \rfloor,b)$ if $a$ is not an integer.} the binomial distribution with $a$ trials and success probability $b$. Moreover, the Binomial is well-enough approximated by a Poisson distribution of the same mean in this regime. In particular, Le Cam's inequality gives 
\begin{align}
\left\| \bin \left(n a, \frac{\ln(n)}{n} b \right) -  \mathcal{P}\left(a b  \ln(n) \right) \right\|_{TV} \leq 2 \frac{a b^2 \ln^2(n)}{n},
\end{align}
hence, by the additivity of Poisson distribution and the triangular inequality, 
\begin{align}
\| \mu_{d(v)} -  \mathcal{P}(\ln(n) \sum_{i \in [k]} p_i Q_{i,j}e_i) \|_{TV} = O \left(\frac{\ln^2(n)}{n} \right).
\end{align} 
We will rely on a simple one-sided bound (see \eqref{bipo}) to approximate our events under the Poisson measure. 

Consider now the following problem. Let $G$ be drawn under the $\gs(n,p,Q)$ SBM and assume that the planted partition is revealed except for a given vertex. Based on the degree profile of that vertex, is it possible to classify the vertex correctly with high probability? We have to resolve a hypothesis testing problem, which involves multivariate Poisson distributions in view of the previous observations. We next study this problem.\\ 

{\bf Testing multivariate Poisson distributions.} Consider the following Bayesian hypothesis testing problem with $k$ hypotheses.
The random variable $H$ takes values in $[k]$ with $\pp\{H=j\}=p_j$ (this is the a priori distribution of $H$). Under $H=j$, an observed random variable $D$ is drawn from a multivariate Poisson distribution with mean $\lambda(j) \in \mR_+^k$, i.e.,  
\begin{align}
\pp\{D=d|H=j\}=\mathcal{P}_{\lambda(j)}(d), \quad d \in \mZ_+^k,
\end{align}
where 
\begin{align}
\mathcal{P}_{\lambda(j)}(d) = \prod_{i \in [k]} \mathcal{P}_{\lambda_i(j)}(d_i), 
\end{align}
and
\begin{align}
\mathcal{P}_{\lambda_i(j)}(d_i) = \frac{\lambda_i(j)^{d_i}}{d_i!} e^{- \lambda_i(j)}.
\end{align} 
In other words, $D$ has independent Poisson entries with different means. We use the following notation to summarize the above setting:
\begin{align}
D|H=j \,\, \sim \mathcal{P}(\lambda(j)), \quad j \in [k].
\end{align}
Our goal is to infer the value of $H$ by observing a realization of $D$. To minimize the error probability given a realization of $D$, we must pick the most likely hypothesis conditioned on this realization, i.e., 
\begin{align}
\argmax_{j \in [k]} \pp \{D=d | H=j\} p_j, \label{map-rule}
\end{align}
which is the Maximum A Posteriori (MAP) decoding rule.\footnote{Ties can be broken arbitrarily.} To resolve this maximization, we can proceed to a tournament of $k-1$ pairwise comparisons of the hypotheses. Each comparison allows us to eliminate one candidate for the maxima, i.e., 
\begin{align}
\pp \{D=d | H=i\} p_i >  \pp \{D=d | H=j\} p_j \quad \Rightarrow \quad H \neq  j. 
\end{align}
The error probability $P_e$ of this decoding rule is then given by,
\begin{align}
P_e = \sum_{i \in [k]} \pp\{ D \in \text{Bad}(i)  | H=i\} p_i, \label{error1}
\end{align}
where $\text{Bad}(i)$ is the region in $\mZ_+^k$ where $i$ is not maximizing \eqref{map-rule}. Moreover, for any $i \in [k]$, 
\begin{align}
\pp\{ D \in \text{Bad}(i) | H=i\} \leq \sum_{j \neq i} \pp\{ D \in \text{Bad}_j(i) |H=i\} \label{pair1}
\end{align}
where $\text{Bad}_j(i)$ is the region in $\mZ_+^k$ where $\pp \{D=x | H=i\} p_{i} \leq  \pp \{D=x | H=j\} p_{j}$. 
Note that with this upper-bound, we are counting the overlap regions where $\pp \{D=x | H=i\} p_{i} \leq  \pp \{D=x | H=j\} p_{j}$ for different $j$'s multiple times, but no more than $k-1$ times. Hence,  
\begin{align}
\sum_{j \neq i} \pp\{ D \in \text{Bad}_j(i) |H=i\}  \leq (k-1) \pp\{ D \in \text{Bad}(i) | H=i\}. \label{dbound1}
\end{align}
Putting \eqref{error1} and \eqref{pair1} together, we have 
\begin{align}
P_e &\leq  \sum_{i \neq j} \pp\{ D \in \text{Bad}_j(i)  | H=i\}   p_i,\\
 &=\sum_{i < j} \sum_{d \in \mZ_+^k} \min(\pp \{D=d | H=i\} p_{i} ,  \pp \{D=d | H=j\} p_{j}) \label{ub1}
\end{align}
and from \eqref{dbound1},
\begin{align}
P_e &\geq \frac{1}{k-1} \sum_{i < j} \sum_{d \in \mZ_+^k} \min(\pp \{D=d | H=i\} p_{i} ,  \pp \{D=d | H=j\} p_{j}). \label{ub2}
\end{align}
Therefore the error probability $P_e$ can be controlled by estimating the terms $\sum_{d \in \mZ_+^k} \min(\pp \{D=d | H=i\} p_{i} ,  \pp \{D=d | H=j\} p_{j})$. In our case, recall that  
\begin{align}
\pp \{D=d | H=i\} = \mathcal{P}_{\lambda(i)}(d),
\end{align}
which is a multivariate Poisson distribution. In particular, we are interested in the regime where $k$ is constant and $\lambda(i) =  \ln(n) c_i$, $c_i \in \mR_+^k$, and $n$ diverges. 
Due to \eqref{ub1}, \eqref{ub2}, we can then control the error probability by controlling $\sum_{x \in \mZ_+^k} \min(\mathcal{P}_{\ln(n) c_i}(x) p_{i} ,  \mathcal{P}_{\ln(n) c_j}(x) p_{j})$, which we will want to be $o(1/n)$ to classify vertices in the SBM correctly with high probability based on their degree profiles (see next section). The following lemma provides the relevant estimates.  

\begin{lemma}\label{hell-expo}
For any $c_1, c_2 \in (\mR_+\setminus \{0\})^k$ with $c_1 \neq c_2$ and $p_1,p_2 \in \mR_+\setminus \{0\}$, 
\begin{align}
& \sum_{x \in \mZ_+^k} \min(\mathcal{P}_{\ln(n) c_1}(x) p_{1} ,  \mathcal{P}_{\ln(n) c_2}(x) p_{2}) = O\left(n^{- \dd(c_1,c_2) - \frac{\ln\ln(n)}{2 \ln(n)}} \right),\\
& \sum_{x \in \mZ_+^k} \min(\mathcal{P}_{\ln(n) c_1}(x) p_{1} ,  \mathcal{P}_{\ln(n) c_2}(x) p_{2}) = \Omega \left(n^{- \dd(c_1,c_2) - \frac{k \ln\ln(n)}{2 \ln(n)}} \right),
\end{align}
where $\dd(c_1,c_2)$ is the CH-divergence as defined in \eqref{h-div}.
\end{lemma}
In other words, the CH-divergence provides the error exponent for deciding among multivariate Poisson distributions. We did not find this result in the literature, but found a similar result obtained by Verd\'u \cite{verdu-hell}, who shows that the Hellinger distance (the special case with $t=1/2$ instead of the maximization over $t$) appears in the error exponent for testing Poisson point-processes, although \cite{verdu-hell} does not investigate the exact error exponent. This lemma was proven in \cite{colin1}. 


This lemma together with previous bounds on $P_e$ imply that if $\dd(c_i,c_j) > 1$ for all $i \neq j$, the true hypothesis is correctly recovered with probability $o(1/n)$. However, it may be that $\dd(c_i,c_j) > 1$ only for a subset of $(i,j)$-pairs. What can we then infer? While we may not recover the true value of $H$ with probability $o(1/n)$, we may narrow down the search within a subset of possible hypotheses with that probability of error.\\

{\bf Testing composite multivariate Poisson distributions.} We now consider the previous setting, but we are no longer interested in determining the true hypothesis, but in deciding between two (or more) disjoint subsets of hypotheses. Under hypothesis 1, the distribution of $D$ belongs to a set of possible distributions, namely $\mathcal{P}(\lambda_i)$ where $i \in A$, and under hypothesis 2, the distribution of $D$ belongs to another set of distributions, namely $\mathcal{P}(\lambda_i)$ where $i \in B$. Note that $A$ and $B$ are disjoint subsets such that $A \cup B=[k]$. In short,
\begin{align}
D|\tilde{H}=1 \,\, \sim \mathcal{P}(\lambda_i), \,\,\, \text{for some } i \in A, \\
D|\tilde{H}=2 \,\, \sim \mathcal{P}(\lambda_i), \,\,\, \text{for some } i \in B,
\end{align}
and as before the prior on $\lambda_i$ is $p_i$. To minimize the probability of deciding the wrong hypothesis upon observing a realization of $D$, we must pick the hypothesis which leads to the larger probability between $\pp\{\tilde{H} \in A | D=d\}$ and $\pp\{\tilde{H} \in B | D=d\}$, or equivalently, 
\begin{align}
\sum_{i \in A} \mathcal{P}_{\lambda(i)}(d) p_{i} \geq \sum_{i \in B} \mathcal{P}_{\lambda(i)}(d) p_{i} \quad \Rightarrow \quad \tilde{H}=1,\\
\sum_{i \in A} \mathcal{P}_{\lambda(i)}(d) p_{i} < \sum_{i \in B} \mathcal{P}_{\lambda(i)}(d) p_{i} \quad \Rightarrow \quad \tilde{H}=2.
\end{align}
In other words, the problem is similar to the previous one, using the above mixture distributions. 
If we denote by $\tilde{P}_e$ the probability of making an error with this test, we have 
\begin{align}
\tilde{P}_e &= \sum_{x \in \mZ_+^k} \min\left(\sum_{i \in A} \mathcal{P}_{\lambda(i)}(x) p_{i} , \sum_{i \in B} \mathcal{P}_{\lambda(i)}(x) p_{i} \right).
\end{align}
Moreover, applying bounds on the minima of two sums,   
\begin{align}
\tilde{P}_e & \leq  \sum_{x \in \mZ_+^k} \sum_{i \in A, j \in B} \min\left(\mathcal{P}_{\lambda(i)}(x) p_{i} ,  \mathcal{P}_{\lambda(j)}(x) p_{j}\right),\\
\tilde{P}_e & \geq  \frac{1}{|A||B|} \sum_{x \in \mZ_+^k} \sum_{i \in A, j \in B} \min\left(\mathcal{P}_{\lambda(i)}(x) p_{i} ,  \mathcal{P}_{\lambda(j)}.(x) p_{j}\right) .
\end{align} 
Therefore, for constant $k$ and $\lambda(i) =  \ln(n) c_i$, $c_i \in \mR_+^k$, with $n$ diverging, it suffices to control the decay of $\sum_{x \in \mZ_+^k}  \min(\mathcal{P}_{\lambda(i)}(x) p_{i} ,  \mathcal{P}_{\lambda(j)}(x) p_{j})$ when $i \in A$ and $j \in B$, in order to bound the error probability of deciding whether a vertex degree profile belongs to a group of communities or not. 

The same reasoning can be applied to the problem of deciding whether a given node belongs to a group of communities, with more than two groups. Also, for any $p$ and $p'$ such that $|p_j-p'_j|<\ln n/\sqrt{n}$ for each $j$, $Q$, $\gamma(n)$, and $i$, 
\begin{align}
\sum_{x\in \mZ_+^k}\max\left(Bin_{\left(np',\frac{(1-\gamma(n))\ln(n)}{n}Q_i\right)}(x)- 2\mathcal{P}_{PQ_i(1-\gamma(n))\ln(n)/n}(x),0\right)=O(1/n^2). \label{bipo} 
\end{align}
So, the error rate for any algorithm that classifies vertices based on their degree profile in a graph drawn from a sparse SBM is at most $O(1/n^2)$ more than twice what it would be if the probability distribution of degree profiles really was the poisson distribution.

In summary, we have proved the following. 
\begin{lemma}\label{testing-lemma}
Let $k \in \mZ_+$ and let $A_1,\dots,A_t$ be disjoint subsets of $[k]$ such that $\cup_{i=1}^t A_i = [k]$. Let $G$ be a random graph drawn under $\gs(n,p,(1-\gamma(n))Q)$. Assigning the most likely community subset $A_i$ to a node $v$ based on its degree profile $d(v)$ gives the correct assignment with probability $$1-O \left(n^{-(1-\gamma(n))\Delta -\frac{1}{2}\ln((1-\gamma(n))\ln n)/\ln n } +\frac{1}{n^2}\right),$$ 
where
\begin{align}
\Delta = \min_{r,s \in [t] \atop{r \neq s}} \min_{ i \in A_r, j \in A_s}  \dd((pQ)_i, (pQ)_j).
\end{align}
\end{lemma}
Moreover, in order to prove  Theorem \ref{thm2}, we will need a version of this lemma that still holds when some of our information is inaccurate. First of all, consider attempting the previous testing procedure when one thinks the distributions are $\lambda'$ with probability $p'$ instead of $\lambda$ with probability $p$. Assume that there exists $\epsilon$ such that for all $i$ and $j$, $|\lambda'(i)_j-\lambda(i)_j|\le\epsilon \min(\lambda'(i)_j,\lambda(i)_j)$ and $|p'_i-p_i|\le\epsilon\min(p'_i,p_i)$. Then for any $i\in A$, $x\in  \mZ_+^k$, the previous hypothesis testing procedure will not classify $x$ as being in $B$ unless 
\[\sum_{j\in B} (1+\epsilon)^{|x|+1} \mathcal{P}_{\lambda(j)}(x) p_{j}\ge (1+\epsilon)^{-|x|-1}\mathcal{P}_{\lambda(i)}(x) p_{i}\] That means that for any $i\in A$, $x\in  \mZ_+^k$, the probability that $x$ arises from $\mathcal{P}_{\lambda(i)}$ and is then misclassified as being in $B$ is at most
\[\mathcal{P}_{\lambda(i)}(x)p_i\le \sum_{j\in B} (1+\epsilon)^{2|x|+2} \mathcal{P}_{\lambda(j)}(x) p_{j}\]

So, this hypothesis testing procedure has an error probability of at most
\[\sum_{x \in \mZ_+^k} (1+\epsilon)^{2|x|+2}\sum_{i \in A, j \in B}\min\left(\mathcal{P}_{\lambda(i)}(x) p_{i} ,\mathcal{P}_{\lambda(j)}(x) p_{j}\right)\]

Furthermore,
\begin{lemma}\label{estim-lemma}
For any $p$ and $Q$ there exists $c$ such that the following holds for all $\delta<\min p_i/2$. Let $G$ be a random graph drawn from $\gs(n,p,Q)$, and $\sigma'$ be a classification of $G$'s vertices that misclassifies $\delta n$ vertices. Also, let $p'$ be the frequencies with which vertices are classified as being in the communities, and $Q'$ be $n/\ln n$ times the edge densities between the alleged communities. Then with probability $1-o(1)$, every element of $p'$ or $Q'$ is within $c\delta+\sqrt{\ln n/n}$ of the corresponding element of $p$ or $Q$.
\end{lemma}

\begin{proof}
With probability $1-o(1)$, the size of every community is within $\sqrt{\ln n/n}$ of its expected value, and the edge density between each pair of communities is within $\sqrt{n}\ln n$ of its expected value. Also, there exists a constant $c'$ such that with probability $1-o(1)$, no vertex has degree more than $c'\ln n$. Assuming this is the case, the misclassification of vertices can not change the apparent number of edges between any two communities by more than $c'\delta n\ln n$, and it can not change the apparent size of any community by more than $\delta n$. Thus, $|p'_i-p_i|\le \sqrt{\ln n/n}+\delta$ and 
\[|Q'_{i,j}-Q_{i,j}|\le \sqrt{\ln n/n}+4\frac{c' \delta}{p_i\cdot p_j}+4Q_{i,j}\frac{\delta p_i+\delta p_j+\delta^2}{p_i p_j}\]

\end{proof}

This lets us prove the following ``robust'' version of this lemma to prove Theorem \ref{thm2}. 
\begin{lemma}\label{testing-lemma2}
Let $k \in \mZ_+$ and let $A_1,\dots,A_t$ be disjoint subsets of $[k]$ such that $\cup_{i=1}^t A_i = [k]$. Let $G$ be a random graph drawn under $\gs(n,p,(1-\gamma(n))Q)$. There exist $c_1$, $c_2$, and $c_3$ such that with probability $1-o(1)$ $G$ is such that for any sufficiently small $\delta$, assigning the most likely community subset $A_i$ to a node $v$ based on a distortion of its degree profile that independently gets each node's community wrong with probability at most $\delta$ and estimates of $p$ and $Q$ based on a classification of the graphs vertices that misclassifies each vertex with probability $\delta$ gives the correct assignment with probability at least $$1- c_2\cdot (1+c_1\delta)^{c_3\ln n} \cdot \left(n^{-(1-\gamma(n))\Delta -\frac{1}{2}\ln((1-\gamma(n))\ln n)/\ln n } \right)-\frac{c_2}{n^2},$$ 
where
\begin{align}
\Delta = \min_{r,s \in [t] \atop{r \neq s}} \min_{ i \in A_r, j \in A_s}  \dd((pQ)_i, (pQ)_j).
\end{align}
\end{lemma}

\begin{proof}
First, note that by the previous lemma, with probability $1-o(1)$, $G$ is such that every element of the estimates of $p$ and $Q$ is within $c\delta+\sqrt{\ln n/n}$ of its true value. Choose $c_3$ such that every vertex in the graph has degree less than $c_3\ln n$ with probability $1-o(1)$. If this holds, then the probability of misclassifying a vertex based on its true degree profile and the estimates of its parameters is at most \[2(1+2c\delta/\min( p_i,(PQ)_{i,j}))^{2c_3\ln n+2}\cdot \left(n^{-(1-\gamma(n))\Delta -\frac{1}{2}\ln((1-\gamma(n))\ln n)/\ln n } \right)+O\left(\frac{1}{n^2}\right)\] based on the previous bounds on the probability that classifying a vertex based on its degree profile fails.

Now, let $$c'_1=\max_{i,j} \sum p_{i'}q_{i',j}/(p_i q_{i,j}).$$ The key observation is that $v$'s $m$th neighbor had at least a $\min_{i,j} (p_i q_{i,j})/\sum p_{i'}q_{i',j}$ chance of actually being in community $\sigma$ for each $\sigma$, so its probability of being reported as being in community $\sigma$ is at most $1+c'_1\delta$ times the probability that it actually is. So, the probability that its reported degree profile is bad is at most $(1+c'_1\delta)^{|N_1(v)|}$ times the probability that its actual degree profile is bad. The conclusion follows.
\end{proof}

\subsubsection{Proof of Theorem \ref{thm2}}\label{proof-thm-2}
We prove the possibility result here. The converse was proven in \cite{colin1} for known parameters, hence also applies here.\\

\noindent
Let $G \sim \gs(n,p,Q)$ and $\gamma= \frac{\ln\ln n}{4 \ln n}$, and $A_1,\dots,A_t$ be a partition of $[k]$. {\tt Agnostic-degree-\\profiling}$(G,p,Q,\gamma)$ recovers the partition $[k] = \sqcup_{s=1}^t A_s$ with probability $1-o_n(1)$ if for all $i,j$ in $[k]$ that are in different subsets,
\begin{align}
\dd ((PQ)_i , (PQ)_j) \geq 1. \label{d1b}
\end{align}

The idea behind Claim 1 is contained in Lemma \ref{testing-lemma}. However, there are several technical steps that need to be handled:
\begin{enumerate}
\item The graphs $G'$ and $G''$ obtained in step 1 of the algorithm are correlated, since an edge cannot be both in $G'$ and $G''$. However, this effect can be discarded since two independent versions would share edges with low enough probability. 
\item  The classification in step 2 using {\tt Agnostic-sphere-comparison} has a vanishing fraction of vertices which are wrongly labeled, and the SBM parameters are unknown. This requires using the robust version of Lemma \ref{testing-lemma}, namely Lemma \ref{testing-lemma2}. 
\item In the case where $\dd ((PQ)_i , (PQ)_j) = 1$ a more careful classification is needed as carried in steps 3 and 4 of the algorithm.    
\end{enumerate} 

\begin{proof}
With probability $1-O(1/n)$, no vertex in the graph has degree greater than $c_3\ln n$. Assuming that this holds, no vertex's set of neighbors in $G''$ is more than $$(1-\max q_{i,j} \ln n/n)^{-c_3\ln n}\cdot (n/(n-c_3\ln n))^{c_3\ln n}=1+o(1)$$ times as likely to occur as it would be if $G''$ were independent of $G'$. So, the fact that they are not has negligible impact on the algorithm's error rate. Now, let $$\delta=(e^{\frac{(1-\gamma)}{2c_3} \min_{i\ne j} \dd((PQ)_i,(PQ)_j)}-1)/c_1.$$ By Lemma \ref{testing-lemma2}, if the classification in step $2$ has an error rate of at most $\delta$, then the classification in step $3$ has an error rate of $$O(n^{- (1-\gamma)\min_{i\ne j} \dd((PQ)_i,(PQ)_j)/2}+1/n^2),$$ observing that if $\sigma'_v\ne \sigma_v$ the error rate of $\sigma''_{v'}$ for $v'$ adjacent to $v$ is at worst multiplied by a constant. That in turn ensures that the final classification has an error rate of at most \[O\left((1+O(n^{- (1-\gamma)\min_{i\ne j} \dd((PQ)_i,(PQ)_j)/2}+1/n^2))^{c_3\ln n} \frac{1}{n}\ln{n}^{-1/4}\right)=O\left(\frac{1}{n}\ln n^{-1/4}\right) .\]
\end{proof}

\end{document}